\documentclass[titlepage,12pt]{article} 
\usepackage{hyperref}
\usepackage{amssymb,amsthm,amsmath} 
\usepackage[a4paper]{geometry}
\usepackage[utf8]{inputenc}
\usepackage[italian,english]{babel}

\date{}

\newcommand{\ep}{\varepsilon}
\newcommand{\re}{\mathbb{R}}

\newcommand{\al}{a_{\lambda}}
\newcommand{\albar}{{a}^*_{\lambda}}
\newcommand{\bl}{b_{\lambda}}
\newcommand{\cl}{c_{\lambda}}
\newcommand{\blbar}{{b}^*_{\lambda}}
\newcommand{\el}{\ep_{\lambda}}
\newcommand{\ul}{u_{\lambda}}
\newcommand{\rhol}{\rho_{\lambda}}
\newcommand{\thel}{\theta_{\lambda}}
\newcommand{\psil}{\psi_{\lambda}}
\newcommand{\Hl}{H_{\lambda}}
\newcommand{\Ekov}{E_{\mathrm{kov},\lambda}}
\newcommand{\Ehyp}{E_{\mathrm{hyp},\lambda}}

\newcommand{\D}{\mathcal{D}}
\newcommand{\F}{\mathcal{F}}
\newcommand{\G}{\mathcal{G}}
\newcommand{\M}{\mathcal{M}}
\newcommand{\Sp}{\mathcal{S}}

\newcommand{\osc}{\operatorname{osc}}

\newtheorem{thm}{Theorem}[section]

\newtheorem{rmk}[thm]{Remark}
\newtheorem{prop}[thm]{Proposition}
\newtheorem{defn}[thm]{Definition}
\newtheorem{cor}[thm]{Corollary}
\newtheorem{ex}[thm]{Example}
\newtheorem{lemma}[thm]{Lemma}

\title{Fifty Years of Wave Equations with Time-Dependent Propagation Speeds: A Variational Perspective and New Frontiers}

\author{Marina Ghisi\vspace{1ex}\\ 
{\normalsize Università degli Studi di Pisa} \\
{\normalsize Dipartimento di Matematica}\\ 
{\normalsize PISA (Italy)}\\
{\normalsize e-mail: \texttt{marina.ghisi@unipi.it}}
\and
Massimo Gobbino\vspace{1ex}\\ 
{\normalsize Università degli Studi di Pisa} \\
{\normalsize Dipartimento di Matematica}\\ 
{\normalsize PISA (Italy)}\\  
{\normalsize e-mail: \texttt{massimo.gobbino@unipi.it}}
}

\begin{document}

\maketitle

\begin{abstract}

We study abstract wave equations with time-dependent propagation speed under the strict hyperbolicity assumption. Our starting point is the observation that a substantial part of the classical energy theory, developed through apparently different constructions, can be organized around a common variational mechanism. Approximate energies naturally lead to regularity--fidelity problems in which a smooth approximation of the reciprocal propagation speed is chosen by balancing the size of its derivatives against its distance from the original coefficient.

We combine this approximation mechanism with higher-order energy corrections and obtain a hierarchy of variational problems together with a corresponding cascade of almost conserved energies of arbitrary order. The resulting variational problems are then studied independently of the evolution equation. A qualitative distinction emerges between the first and the higher orders: for arbitrary prescribed growth, all higher orders generate the same classes, while the first-order problem may behave differently. We also establish connections with fractional and generalized fractional Sobolev regularity.

This framework provides a unified interpretation of several classical sufficient conditions for energy control and derivative-loss estimates. At the same time, it identifies a natural limitation of the traditional approach, namely the use of absolute integrability of the energy errors. Once this limitation is recognized, a different regime becomes accessible: by retaining the oscillatory structure of the energy errors and exploiting cancellations, we construct strictly hyperbolic propagation speeds beyond the regularity thresholds detected by the variational theory, for which uniform energy estimates and hence no derivative loss nevertheless hold.

Thus the variational viewpoint both reveals an unexpected common structure behind a substantial part of the classical theory and indicates a genuinely different regime beyond it.
\vspace{6ex}

\noindent{\bf Mathematics Subject Classification 2020 (MSC2020):} 
35L90 (primary), 35L05, 35B65, 35B45, 49J45 (secondary).


\vspace{6ex}


\end{abstract}



\section{Introduction}

We consider the abstract evolution equation
\begin{equation}
    u''(t)+c(t)^2Au(t)=0,
    \label{eqn:basic}
\end{equation}
where $A$ is a nonnegative self-adjoint operator on a Hilbert space $H$, and $c:(0,T)\to\mathbb R$ is a time-dependent propagation speed. The model example is the wave equation
\begin{equation*}
    u_{tt}(t,x)-c(t)^2\Delta u(t,x)=0.
\end{equation*}

By Fourier or spectral decomposition, the analysis reduces to the family of scalar equations
\begin{equation}
    u_\lambda''(t)+\lambda^2c(t)^2u_\lambda(t)=0,
    \label{eqn:ODE}
\end{equation}
where $\lambda$ is the frequency parameter. The regularity of solutions to the abstract problem is therefore governed by the high-frequency behavior of the energy of solutions to~(\ref{eqn:ODE}).

When the propagation speed is bounded away from zero and has bounded variation, the situation is classical: the hyperbolic energy
\[
    E(t):=|u'(t)|^2+c(t)^2|A^{1/2}u(t)|^2
\]
remains uniformly controlled and no derivative loss occurs. Beyond bounded variation, however, the picture changes dramatically. The seminal papers of F.~Colombini, E.~De~Giorgi and S.~Spagnolo~\cite{1979-DGCS}, and of F.~Colombini, E.~Jannelli and S.~Spagnolo~\cite{1983-CJS}, showed that low time regularity of the propagation speed can interact with the spatial regularity of the initial data and produce derivative loss, ranging from a finite loss of Sobolev regularity to instantaneous infinite loss. Subsequent examples have shown that such pathological behavior is not confined to exceptional coefficients and can even be residual in suitable spaces of propagation speeds~\cite{GhisiGobbino2023Residual}.

Over the last decades, a large number of sufficient conditions have been developed in order to prevent or quantify derivative loss for wave equations with time-dependent propagation speed. On finite time intervals, these results have been obtained through a variety of techniques, including regularization of the coefficient, approximate and progressively corrected energies, and, in a parallel language, successive diagonalization procedures for the associated first-order systems; see for instance~\cite{Yamazaki1990,CDSK2002,CDSR2003,HirosawaReissig2004,KinoshitaReissig2005,Tarama2007,Manfrin2008} and the references therein. At first sight, these approaches involve rather different assumptions and rather different constructions. The starting point of the present paper is the observation that a substantial part of this theory can instead be organized around a common variational mechanism. Our purpose is not to survey the existing results, but to identify this hidden structure, develop its consequences, and then investigate what lies beyond the regime that it describes.

\paragraph{\textmd{\textit{Aim of the paper}}}

The purpose of this paper is fourfold.

\begin{enumerate}

\item
Our first aim is to make this variational mechanism explicit and to show that several assumptions which historically required different energy constructions can be interpreted as different ways of producing efficient competitors for the same variational problem. To the best of our knowledge, this common variational structure has not been isolated explicitly in the previous literature.

\item
Our second aim is to combine the approximate-energy philosophy of~\cite{1979-DGCS} with the higher-order correction mechanism and extend the resulting construction to arbitrary order. Higher-order quadratic forms for related scalar equations have appeared previously in the literature, notably in the work of R.~Manfrin~\cite{Manfrin2008}. Here, in the recursive cascade, we replace the reciprocal propagation speed with a smooth auxiliary function $\gamma$. This yields a corresponding hierarchy of functionals of the form
\begin{equation}
    \frac1{\lambda^{k-1}}
    \int_0^T |\gamma^{(k)}(t)|\,dt
    +
    \lambda
    \int_0^T\left|\gamma(t)-\frac{1}{c(t)}\right|\,dt,
    \label{defn:Fk-intro}
\end{equation}
where the first term measures regularity and the second one fidelity.

\item
Our third aim is to study the variational problems associated with~(\ref{defn:Fk-intro}) as objects in their own right. We compare different differentiation orders and identify a genuine distinction between the first and the higher orders: for arbitrary prescribed growth, all variational problems of order $k\geq2$ generate the same classes, whereas the first-order problem may behave differently. We also establish connections with fractional and generalized fractional Sobolev regularity.

\item
Our final aim is to move beyond the regime in which energy errors are estimated through absolute integrability. The exact energy identities contain oscillatory factors which disappear when absolute values are taken. By retaining these oscillations and exploiting the resulting cancellations, we construct highly irregular propagation speeds for which the classical absolute-integrability criteria fail, while uniform energy estimates and absence of derivative loss still hold.

\end{enumerate}

The first three parts identify and develop a common structure behind a substantial portion of the classical theory. Somewhat unexpectedly, although approximate energies can be constructed at arbitrary order, the variational information relevant for controlling energy growth already saturates at second order. More concretely, the variational approach singles out a class of propagation speeds for which the corresponding wave equation is well posed in the Sobolev scale without derivative loss: boundedness of the minimum values associated with the second-order functional in~(\ref{defn:Fk-intro}), with $1/c$ as target function, is sufficient, and the same class is obtained at every higher order. This class is strictly larger than the classical bounded-variation class, but is still strictly contained in the intersection of all fractional Sobolev spaces $W^{s,1}$ with $s<1$.

The last part begins precisely where this framework reaches its natural limit. The absolute-value estimates underlying the variational theory discard the oscillatory structure of the energy errors. Once this absolute-integrability paradigm is abandoned, the same energy identities reveal cancellations which are invisible to the variational theory. The examples constructed in the final section show that uniform energy estimates can persist far beyond the preceding regularity regime: the propagation speeds can be chosen arbitrarily close to the fractional Sobolev threshold of order $1/2$, while failing to belong to $W^{s,1}$ for every $s>1/2$. Thus the passage from the variational theory to the oscillatory construction is not a marginal improvement of the regularity assumptions, but the emergence of a genuinely different mechanism.

\paragraph{\textmd{\textit{Related literature}}}

Beyond the finite-time theory considered above, closely related questions arise in several other directions.

One such direction concerns the long-time behavior of the energy of solutions to~(\ref{eqn:basic}) on the half-line. In this setting one asks for conditions on the propagation speed ensuring a generalized energy conservation law, namely the existence of positive constants $C_1$ and $C_2$ such that
\[
    C_1 E(0)\leq E(t)\leq C_2 E(0)
    \qquad
    \forall t\geq0
\]
for every solution. Conditions combining regularity, oscillation, and stabilization of the propagation speed have been developed for this purpose; see for example~\cite{Hirosawa2007,HirosawaWirth2009,BoitiManfrin2012}. Related linear equations also arise in dispersive and geometric settings. In particular, wave and Klein--Gordon equations with time-dependent coefficients occur naturally on Robertson--Walker and de Sitter backgrounds; see for example~\cite{YagdjianGalstian2008,YagdjianGalstian2009}.

The interaction with Kirchhoff equations is especially relevant to the present work, because the exchange of ideas between the linear and nonlinear theories has historically gone in both directions. In abstract form, the Kirchhoff equation can be written as
\begin{equation*}
    u''(t)
    +
    m\left(\left|A^{1/2}u(t)\right|^2\right)Au(t)
    =
    0.
\end{equation*}
Along a given solution this is a linear wave equation with time-dependent propagation speed
\begin{equation*}
    c(t)^2
    =
    m\left(\left|A^{1/2}u(t)\right|^2\right).
\end{equation*}
Progress in the linear theory therefore feeds directly into the nonlinear problem, and this mechanism already played a decisive role in the early 1980s. In his 1977 Rio de Janeiro lectures, J.-L.~Lions~\cite{Lions1978} formulated several open problems for Kirchhoff-type equations, including the case of a merely continuous nonlinearity and the weakly hyperbolic case in which the coefficient is allowed to vanish. The breakthrough in the linear theory due to~\cite{1979-DGCS} provided precisely the tools needed to treat such low-regularity time-dependent coefficients. Building on this development, A.~Arosio and S.~Spagnolo~\cite{ArosioSpagnolo1984} were able to address both of Lions' problems. Later developments extended the same interaction to nonlinearities with intermediate moduli of continuity and to the corresponding Gevrey-type regularity of the initial data~\cite{Hirosawa2003Loss,GhisiGobbino09}.

The influence also goes in the opposite direction. Remarkably, the second-order correction which later appears in the linear theory was already present, in essentially the same form, in a conservation law discovered by S.~I.~Pokhozhaev~\cite{Pokhozhaev1974} for a special Kirchhoff equation. In that setting it was explicitly identified as a second-order conservation law and used to obtain higher-order a priori estimates. Thus a structure later associated with second-order corrected energies in the linear theory had already emerged from the nonlinear Kirchhoff problem several decades earlier. Higher-order quadratic forms for the scalar Liouville equation were subsequently developed in~\cite{Manfrin2008} precisely with applications to Kirchhoff equations in mind, and recent work has returned once again to the nonlinear problem and produced further higher-order conservation laws~\cite{BoitiManfrin2023,BoitiManfrin2026}.

\paragraph{\textmd{\textit{Organization of the paper}}}

Section~\ref{sec:background} recalls the reduction of the abstract equation to the scalar family~(\ref{eqn:ODE}) and introduces the growth function which measures the high-frequency amplification of the standard energy and hence the amount of derivative loss.

In Section~\ref{sec:smooth} we construct a hierarchy of higher-order almost conserved energies for smooth propagation speeds, in the spirit of the higher-order quadratic forms introduced in~\cite{Manfrin2008}. The construction is organized as a recursive cascade of corrections which progressively cancels the leading terms in the derivative of the preceding energy. We also discuss the exceptional case in which the final remainder vanishes and the corresponding higher-order energy becomes exactly conserved.

Section~\ref{sec:approx-en} extends the same construction to approximate energies. The reciprocal propagation speed is replaced by a smooth auxiliary function, and the resulting estimates lead naturally to the regularity--fidelity functionals introduced in~(\ref{defn:Fk-intro}).

Section~\ref{sec:variational} develops the theory of the corresponding minimum problems independently of the evolution equation. We compare different differentiation orders, introduce the associated function spaces, study their relation with fractional Sobolev regularity, and analyze the constrained versions required by the energy estimates.

In Section~\ref{sec:var2derloss} we return to the evolution equation and convert the variational estimates into bounds for energy growth. Several classical assumptions on the propagation speed are recovered by constructing suitable competitors, including derivative conditions of different orders, mixed regularity assumptions, three-region constructions, and Zygmund-type conditions.

Section~\ref{sec:construction} leaves the absolute-integrability framework and exploits cancellations in oscillatory integrals. We construct strictly hyperbolic propagation speeds with very low fractional Sobolev regularity for which the corresponding wave equation nevertheless admits uniform energy estimates and hence well-posedness without derivative loss.

Finally, Section~\ref{sec:open} is devoted to open questions suggested by the results of the paper, concerning in particular the regularity thresholds that may separate good and bad propagation speeds.

 
\setcounter{equation}{0}
\section{Notation and background}
\label{sec:background}

In this section we fix once and for all the notation used throughout the paper, recall the standard spectral reduction of the abstract wave equation~(\ref{eqn:basic}) to the scalar family~(\ref{eqn:ODE}), and introduce a quantity that measures the growth of the corresponding scalar energies and describes the derivative loss for the abstract problem.

Let $H$ be a Hilbert space, let $A$ be a nonnegative self-adjoint operator on $H$, and let $c:(0,T)\to\re$ be a measurable function. Throughout the paper we assume that the propagation speed satisfies the strict hyperbolicity condition
\begin{equation}
    \nu_1\leq c(t)\leq\nu_2
    \qquad
    \text{for almost every }t\in(0,T),
    \label{hp:sh}
\end{equation}
for two constants $0<\nu_1\leq\nu_2$.

\paragraph{\textmd{\textit{Reduction to a family of scalar ODEs}}}

By the spectral theorem for self-adjoint operators (see for example~\cite[Theorem~VIII.4]{ReedSimonI}), the operator $A$ admits a spectral representation in which it acts as multiplication by the square of a nonnegative measurable function $\lambda(\xi)$ in a suitable $L^2$ space. Accordingly, the study of~(\ref{eqn:basic}) reduces to the family of scalar ordinary differential equations~(\ref{eqn:ODE}), where $\lambda\geq0$ plays the role of the frequency parameter.

These are the equations satisfied by the components of $u$ in the spectral decomposition of $A$. In the sequel we often work directly with~(\ref{eqn:ODE}), and the dependence of the corresponding estimates on $\lambda$ will provide information on the regularity of solutions to the abstract evolution problem.

\paragraph{\textmd{\textit{Energy growth}}}

For every solution $\ul$ to~(\ref{eqn:ODE}) we consider the standard energy
\begin{equation}
E_\lambda(t):=\ul'(t)^2+\lambda^2\ul(t)^2.
\label{defn:standard-energy}
\end{equation}
For simplicity, we omit from the notation the dependence of $E_\lambda$ on the solution $\ul$, which will always be clear from the context.

For every $\lambda>0$, we define
\begin{equation*}
\G(c,\lambda,T):=
\log\left(
\sup\left\{
E_\lambda(t):
\ul\text{ solves~(\ref{eqn:ODE})},\
E_\lambda(0)\leq1,\ 
t\in[0,T]
\right\}
\strut\right).
\end{equation*}

Equivalently, $\G(c,\lambda,T)$ is the smallest constant $C$ for which
\begin{equation*}
E_\lambda(t)\leq
\exp(C)\cdot E_\lambda(0)
\qquad
\forall t\in[0,T]
\end{equation*}
holds for every solution to~(\ref{eqn:ODE}).

There is also a useful operator-theoretic interpretation of this quantity. For every $\lambda>0$, we endow $\re^2$ with the $\lambda$-energy norm
\begin{equation*}
\|(a,b)\|_\lambda:=\left(b^2+\lambda^2a^2\right)^{1/2},
\end{equation*}
and $C^0([0,T];\re^2)$ with the corresponding supremum norm. Then the solution operator
$$\text{initial data}\rightsquigarrow\text{solution},$$
namely the linear operator
\begin{equation*}
\re^2\ni
(\ul(0),\ul'(0))\longmapsto
(\ul,\ul')\in
C^0([0,T];\re^2),
\end{equation*}
has norm equal to $\exp\left(\G(c,\lambda,T)/2\right)$.

\paragraph{\textmd{\textit{Derivative loss for the abstract equation}}}

The asymptotic behavior of $\G(c,\lambda,T)$ as $\lambda\to+\infty$ governs the derivative loss for solutions to~(\ref{eqn:basic}) according to the classical scheme
\begin{align*}
    \G(c,\lambda,T)\text{ bounded}
    &\quad\leadsto\quad
    \text{no derivative loss},
    \\
    1\ll\G(c,\lambda,T)\ll\log\lambda
    &\quad\leadsto\quad
    \text{arbitrarily small derivative loss},
    \\
    \G(c,\lambda,T)\simeq\log\lambda
    &\quad\leadsto\quad
    \text{finite derivative loss},
    \\
    \G(c,\lambda,T)\gg\log\lambda
    &\quad\leadsto\quad
    \text{infinite derivative loss}.
\end{align*}

The implications in this scheme have two complementary aspects. Upper bounds for $\G(c,\lambda,T)$ immediately yield corresponding upper bounds for the derivative loss. Conversely, large values of $\G(c,\lambda,T)$ can produce actual derivative loss for the abstract equation provided that they occur along frequencies which are effectively represented by the spectral decomposition of $A$. This is immediate when the squares of the relevant frequencies form a sequence of eigenvalues of $A$. More generally, it is enough that the growth occurs near an unbounded sequence of frequencies whose arbitrarily small neighborhoods correspond to spectral sets of positive measure. One can then choose, on each such spectral region, scalar solutions of~(\ref{eqn:ODE}) whose energy growth nearly realizes the supremum in the definition of $\G$, and use them as spectral components of the abstract solution. This mechanism is made precise in~\cite[Proposition~4.5]{gg:OptDerLoss}. In the model case $A=-\Delta$ in the whole space, there is no spectral obstruction, since the frequency parameter ranges over the whole positive half-line.

The meaning of the four kinds of derivative loss appearing in the previous scheme can be made more precise as follows.

\begin{itemize}

\item No derivative loss means well-posedness in Sobolev spaces. More precisely, for every $\alpha\geq0$, initial data
\begin{equation}
(u(0),u'(0))\in
D(A^{\alpha+1/2})\times D(A^\alpha)
\label{eqn:abstract-data}
\end{equation}
give rise to a solution satisfying
\begin{equation*}
u\in
C^0\left([0,T];D(A^{\alpha+1/2})\right)
\cap
C^1\left([0,T];D(A^\alpha)\right).
\end{equation*}

\item Arbitrarily small derivative loss means that, for every $\ep>0$ and every $\alpha\geq\ep$, every solution with initial data~(\ref{eqn:abstract-data}) belongs to
\begin{equation*}
C^0\left([0,T];D(A^{\alpha+1/2-\ep})\right)
\cap
C^1\left([0,T];D(A^{\alpha-\ep})\right).
\end{equation*}

\item Finite derivative loss means that there exists a nondecreasing function $\delta:(0,T]\to[0,+\infty)$ such that, for every $t\in(0,T]$ and every $\alpha\geq\delta(t)$, every solution with initial data~(\ref{eqn:abstract-data}) belongs to
\begin{equation*}
C^0\left([0,t];D(A^{\alpha+1/2-\delta(t)})\right)
\cap
C^1\left([0,t];D(A^{\alpha-\delta(t)})\right).
\end{equation*}

\item Infinite derivative loss means that even arbitrarily regular initial data may give rise to solutions that instantly leave every standard Sobolev regularity scale. More precisely, one can have initial data satisfying~(\ref{eqn:abstract-data}) for every $\alpha\geq0$, for which the corresponding solution, although it exists and is unique in a suitable large space of ultradistributions, satisfies
\begin{equation*}
(u(t),u'(t))
\notin
D(A^{\alpha+1/2})\times D(A^\alpha)
\qquad
\forall\alpha\geq0
\end{equation*}
for every $t>0$. In the strongest examples, the solution at positive times may even fail to belong to the usual distribution spaces associated with the operator $A$.

\end{itemize}

Thus the growth function $\G(c,\lambda,T)$ provides the link between estimates for the scalar equation and regularity properties of solutions to the abstract evolution problem. The next sections develop the energy and variational tools that will be used to investigate its high-frequency behavior.


\setcounter{equation}{0}
\section{Smooth propagation speeds}\label{sec:smooth}

It is classical that, under the strict hyperbolicity assumption, bounded variation of the propagation speed is sufficient to obtain uniform estimates for the hyperbolic energy
\[
\frac{\ul'(t)^2}{c(t)}
+\lambda^2c(t)\ul(t)^2.
\]
The same assumption also guarantees that this energy is uniformly equivalent to the standard energy $E_\lambda$, and therefore these estimates imply well-posedness in the Sobolev scale without derivative loss.

When the propagation speed is smoother, this basic hyperbolic energy can be progressively corrected in order to obtain higher-order energies whose relative variation becomes smaller and smaller at high frequencies. These corrections are chosen recursively so as to cancel the leading contributions to the time derivative of the previous energy. The higher the regularity of the propagation speed, the further this cancellation procedure can be continued.

Higher-order corrected quadratic forms for the scalar Liouville equation were systematically developed by R.~Manfrin~\cite{Manfrin2008}, who also identified a polynomial structure of their coefficients. The construction below is closely related to that approach, but we use a recursive formulation in terms of the reciprocal propagation speed which will be particularly convenient for the approximate-energy and variational developments of the following sections.

The outcome is a hierarchy of frequency-dependent quadratic forms which, for large $\lambda$, are uniformly equivalent to the standard energy and become increasingly close to conserved quantities. More precisely, if $c$ is of class $C^k$, then one can construct a quadratic form whose relative variation on $[0,T]$ is of order $\lambda^{1-k}$ in the exponential scale.

\begin{thm}[Higher-order almost conserved energies]\label{thm:cascade}
Let $T$ be a positive real number, let $k$ be a positive integer, and let
$c\in C^k([0,T])$ satisfy the strict hyperbolicity assumption~(\ref{hp:sh}).

Then there exist three functions
\[
\alpha_{k},\beta_{k},\gamma_{k}:(0,+\infty)\times[0,T]\to\re,
\]
and four positive constants $A_{k}$, $B_{k}$, $M_{k}$, $\lambda_{k}$ such that the following property holds.

For every $\lambda\geq\lambda_k$, and for every solution $\ul$ to~(\ref{eqn:ODE}), the quantity
\begin{equation}
\mathcal{E}_k(t):=
\alpha_{k}(\lambda,t)\ul'(t)^2
+\beta_{k}(\lambda,t)\ul(t)\ul'(t)
+\gamma_{k}(\lambda,t)\lambda^2\ul(t)^2
\label{defn:higher-order-energy}
\end{equation}
is uniformly equivalent to the standard energy (\ref{defn:standard-energy}) in the sense that
\begin{equation}
A_{k}E_\lambda(t)
\leq
\mathcal{E}_k(t)
\leq
B_{k}E_\lambda(t)
\qquad
\forall t\in[0,T],
\label{th:k-equivalence}
\end{equation}
and satisfies the growth/decay estimate
\begin{equation}
\mathcal{E}_k(0)\exp\left(
-\frac{M_{k}}{\lambda^{k-1}}
\right)
\leq
\mathcal{E}_k(t)
\leq
\mathcal{E}_k(0)\exp\left(
\frac{M_{k}}{\lambda^{k-1}}
\right)
\qquad
\forall t\in[0,T].
\label{th:k-conservation}
\end{equation}
\end{thm}

\begin{proof}

We first discuss the cases with small $k$. Besides proving the first instances of the result, these computations reveal the cancellation mechanism behind the construction and lead naturally to its recursive form. We then complete the proof by describing the general cascade.

\paragraph{\textmd{\textit{The case $k=1$.}}}

We set
\begin{equation*}
a_1(t):=\frac{1}{c(t)}
\end{equation*}
and define
\begin{equation*}
\mathcal{E}_1(t):=
a_1(t)
\left(
\ul'(t)^2+\lambda^2c(t)^2\ul(t)^2
\right)
=
\frac{\ul'(t)^2}{c(t)}
+\lambda^2c(t)\ul(t)^2.
\end{equation*}

The equivalence with the standard energy follows from the strict hyperbolicity assumption (\ref{hp:sh}), and in this case is true for every $\lambda>0$.

On the other hand, using~(\ref{eqn:ODE}), a direct differentiation gives
\begin{equation*}
\mathcal{E}_1'(t)
=
a_1'(t)
\left(
\ul'(t)^2-\lambda^2c(t)^2\ul(t)^2
\right),
\end{equation*}
and therefore
\begin{equation*}
|\mathcal{E}_1'(t)|
\leq
\frac{|c'(t)|}{c(t)}
\mathcal{E}_1(t)
\leq
L_1\mathcal{E}_1(t)
\qquad
\forall t\in [0,T]
\end{equation*}
for a suitable constant $L_1$ that depends on $c'$ and on the lower bound for $c$. Integrating this differential inequality in both directions yields exactly (\ref{th:k-conservation}) with $k=1$.

\paragraph{\textmd{\textit{The case $k=2$.}}}

We introduce the first correction and set
\begin{equation*}
\mathcal{E}_2(t):=
\mathcal{E}_1(t)-a_1'(t)\ul(t)\ul'(t)=
\frac{\ul'(t)^2}{c(t)}
+\lambda^2c(t)\ul(t)^2
+\frac{c'(t)}{c(t)^2}\ul(t)\ul'(t).
\end{equation*}
Since
\begin{equation}
|\ul(t)\ul'(t)|
\leq
\frac{1}{2\lambda}E_\lambda(t),
\label{est:uuprime}
\end{equation}
we have
\begin{equation*}
|\mathcal{E}_2(t)-\mathcal{E}_1(t)|
\leq
\frac{H_{2}}{\lambda}E_\lambda(t)
\end{equation*}
for a suitable constant $H_2$ depending only on $c$ and $c'$ on $[0,T]$. Since we already know that $\mathcal{E}_1$ is equivalent to $E_\lambda$, this is enough to conclude that also $\mathcal{E}_2$ is equivalent to $E_\lambda$, at least when $\lambda$ is large enough, which proves (\ref{th:k-equivalence}) for $k=2$.

On the other hand, if $c\in C^2([0,T])$ we can compute
\begin{equation*}
\mathcal{E}_2'(t)=
-a_1''(t)\ul(t)\ul'(t)=
\left(
\frac{c''(t)}{c(t)^2}-2\frac{c'(t)^2}{c(t)^3}
\right)
\ul(t)\ul'(t).
\end{equation*}
Hence, by~(\ref{est:uuprime}) and the equivalence just proved, there exist constants $K_2$ and $L_2$ such that
\begin{equation*}
|\mathcal{E}_2'(t)|
\leq
\frac{K_2}{\lambda}E_\lambda(t)
\leq
\frac{L_{2}}{\lambda}\mathcal{E}_2(t)
\qquad
\forall t\in[0,T]
\end{equation*}
for every large enough $\lambda$. Integrating this differential inequality in both directions yields exactly (\ref{th:k-conservation}) with $k=2$. Now the constants involved depend on $a_1$ and its derivatives up to order~2, and hence ultimately on the first two derivatives of $c$.

\paragraph{\textmd{\textit{The case $k=3$.}}}

If $c\in C^3([0,T])$, we add a second correction and define
\begin{equation*}
\mathcal{E}_3(t):=
\mathcal{E}_2(t)
+\frac{1}{2}a_1''(t)\ul(t)^2.
\end{equation*}
Since
\begin{equation}
\ul(t)^2
\leq
\frac{1}{\lambda^2}E_\lambda(t),
\label{est:u-square}
\end{equation}
we obtain
\begin{equation*}
|\mathcal{E}_3(t)-\mathcal{E}_2(t)|
\leq
\frac{H_3}{\lambda^2}E_\lambda(t).
\end{equation*}
Since we already know that $\mathcal{E}_2$ is equivalent to $E_\lambda$, this shows that also $\mathcal{E}_3$ is equivalent to $E_\lambda$,  after increasing the frequency threshold if necessary.

A further differentiation gives
\begin{equation*}
\mathcal{E}_3'(t)
=
\frac{1}{2}a_1'''(t)\ul(t)^2.
\end{equation*}
Therefore, by~(\ref{est:u-square}) and the equivalence between
$\mathcal{E}_3$ and $E_\lambda$,
\begin{equation*}
|\mathcal{E}_3'(t)|
\leq
\frac{K_3}{\lambda^2}E_\lambda(t)
\leq
\frac{L_{3}}{\lambda^2}\mathcal{E}_3(t)
\end{equation*}
for all sufficiently large $\lambda$. At this point we can again integrate the differential inequality and obtain
(\ref{th:k-conservation}) for $k=3$. Now the constants involved depend on $a_1$ and its derivatives up to order~3, and hence ultimately on the first three derivatives of $c$.

\paragraph{\textmd{\textit{The case $k=4$.}}}

At this point the previous type of correction is no longer sufficient. In order to cancel the term containing $a_1'''$ in the derivative of $\mathcal{E}_3$, we choose a function $a_2$ satisfying
\begin{equation}
\left[c(t)a_2(t)\right]'
=
-\frac{1}{4}\frac{a_1'''(t)}{c(t)}.
\label{eqn:a2-cancellation}
\end{equation}
Such a function always exists, and is determined up to the addition of a constant multiple of $1/c$.

We now define
\begin{equation*}
\mathcal{E}_4(t):=
\mathcal{E}_3(t)
+
\frac{1}{\lambda^2}
\left[
a_2(t)
\left(
\ul'(t)^2+\lambda^2c(t)^2\ul(t)^2
\right)
-a_2'(t)\ul(t)\ul'(t)
\right].
\end{equation*}

The new correction is again small at high frequencies. Indeed, strict hyperbolicity together with~(\ref{est:uuprime}) gives
\begin{equation*}
|\mathcal{E}_4(t)-\mathcal{E}_3(t)|
\leq
\frac{H_{4,1}}{\lambda^2}E_\lambda(t)
+
\frac{H_{4,2}}{\lambda^3}E_\lambda(t),
\end{equation*}
from which in the usual way we obtain~(\ref{th:k-equivalence}) for $k=4$.

In addition, differentiating in time and exploiting~(\ref{eqn:a2-cancellation}), we obtain
\begin{equation*}
\mathcal{E}_4'(t)=
-\frac{a_2''(t)}{\lambda^2}\ul(t)\ul'(t).
\end{equation*}
Using again~(\ref{est:uuprime}) and the equivalence of the energies, we conclude that
\begin{equation*}
|\mathcal{E}_4'(t)|
\leq
\frac{K_4}{\lambda^3}E_\lambda(t)
\leq
\frac{L_4}{\lambda^3}\mathcal{E}_4(t)
\end{equation*}
for all sufficiently large $\lambda$, from which again we obtain~(\ref{th:k-conservation}) with $k=4$. We observe that the highest derivative of $a_2$ involved in these estimates is $a_2''$, which is well defined when $c$ is of class $C^4$.

\paragraph{\textmd{\textit{The general construction.}}}

The first four cases display the complete mechanism that will be iterated in the general construction, which we now describe for arbitrary $k\geq4$. Let
\[
m:=\left\lfloor\frac{k}{2}\right\rfloor .
\]
Starting from $a_1=1/c$, we choose recursively functions $a_{i+1}$ satisfying
\begin{equation}
\left[c(t)a_{i+1}(t)\right]'
=
-\frac{1}{4}\frac{a_i'''(t)}{c(t)}
\qquad
\forall i\in\{1,\ldots,m-1\}.
\label{eqn:cascade-cancellation}
\end{equation}
At each step such a function exists and is determined up to the addition of a constant multiple of $1/c$. No particular choice of these integration constants is needed in the construction.

The first cases suggest the following recursive construction of the corresponding energies. For every $i\geq2$ involved in the construction, we set
\begin{equation}
\mathcal{E}_{2i}(t):=
\mathcal{E}_{2i-1}(t)
+
\frac{1}{\lambda^{2i-2}}
\left[
a_i(t)
\left(
\ul'(t)^2+\lambda^2c(t)^2\ul(t)^2
\right)
-a_i'(t)\ul(t)\ul'(t)
\right],
\label{defn:cascade-even}
\end{equation}
while, for every $i\geq1$ involved in the construction, we set
\begin{equation}
\mathcal{E}_{2i+1}(t):=
\mathcal{E}_{2i}(t)
+
\frac{1}{2}\frac{1}{\lambda^{2i-2}}
a_i''(t)\ul(t)^2.
\label{defn:cascade-odd}
\end{equation}

Let us first check the equivalence with the standard energy. From strict hyperbolicity and~(\ref{est:uuprime}), for every $i\geq2$ we have
\begin{equation*}
|\mathcal{E}_{2i}(t)-\mathcal{E}_{2i-1}(t)|
\leq
\frac{H_{2i}}{\lambda^{2i-2}}E_\lambda(t)
\end{equation*}
for every $\lambda\geq1$. Therefore, if $\mathcal{E}_{2i-1}$ is equivalent to $E_\lambda$ for large $\lambda$, the same is true for $\mathcal{E}_{2i}$, after increasing the frequency threshold if necessary. Similarly, from~(\ref{est:u-square}),
\begin{equation*}
|\mathcal{E}_{2i+1}(t)-\mathcal{E}_{2i}(t)|
\leq
\frac{H_{2i+1}}{\lambda^{2i}}E_\lambda(t),
\end{equation*}
and hence the equivalence propagates from $\mathcal{E}_{2i}$ to $\mathcal{E}_{2i+1}$. Since the first cases have already been proved, this yields~(\ref{th:k-equivalence}) for every $k$.

It remains to estimate the time derivative. The same cancellations exhibited in the first four cases, together with~(\ref{eqn:cascade-cancellation}), give
\begin{equation*}
\mathcal{E}_{2i}'(t)
=
-\frac{a_i''(t)}{\lambda^{2i-2}}
\ul(t)\ul'(t)
\qquad\quad\text{and}\quad\qquad
\mathcal{E}_{2i+1}'(t)
=
\frac{1}{2}\frac{a_i'''(t)}{\lambda^{2i-2}}
\ul(t)^2.
\end{equation*}

If $k=2i$, using~(\ref{est:uuprime}) and the equivalence just proved, we obtain
\begin{equation*}
|\mathcal{E}_k'(t)|
\leq
\frac{K_k}{\lambda^{2i-1}}E_\lambda(t)
\leq
\frac{L_k}{\lambda^{k-1}}\mathcal{E}_k(t)
\qquad
\forall t\in[0,T].
\end{equation*}

If $k=2i+1$, using~(\ref{est:u-square}) we obtain instead
\begin{equation*}
|\mathcal{E}_k'(t)|
\leq
\frac{K_k}{\lambda^{2i}}E_\lambda(t)
\leq
\frac{L_k}{\lambda^{k-1}}\mathcal{E}_k(t)
\qquad
\forall t\in[0,T].
\end{equation*}
In both cases, integration of the differential inequality in both directions yields~(\ref{th:k-conservation}).

We finally check the regularity required by the construction. The differential relation~(\ref{eqn:cascade-cancellation}) shows that, in order to choose $a_{i+1}$ starting from $a_i$, three derivatives of $a_i$ are required, while solving the resulting first-order equation restores one derivative. An induction therefore gives
\begin{equation*}
a_i\in C^{k-2i+2}([0,T])
\qquad
\forall i\in\{1,\ldots,m\}.
\end{equation*}
If $k=2i$, then $a_i\in C^2$, which is exactly the regularity needed to control the term $a_i''$ appearing in $\mathcal{E}_{2i}'$. If $k=2i+1$, then $a_i\in C^3$, which is exactly the regularity needed to control $a_i'''$ in $\mathcal{E}_{2i+1}'$. Thus all the coefficients introduced above can be chosen with the required regularity, and the construction uses precisely the $k$ derivatives of $c$ assumed in the statement.

Finally, the recursive definitions~(\ref{defn:cascade-even}) and~(\ref{defn:cascade-odd}) show that every $\mathcal{E}_k$ is a quadratic form of the type~(\ref{defn:higher-order-energy}), and hence determine the corresponding coefficients $\alpha_k$, $\beta_k$, and $\gamma_k$. This completes the proof.
\end{proof}

\begin{rmk}[Local structure of the coefficients]
\begin{em}

Although the recursive relation~(\ref{eqn:cascade-cancellation}) determines each coefficient $a_{i+1}$ through a first-order differential equation, there is a distinguished choice of the integration constants for which all the coefficients become local differential polynomials in the reciprocal propagation speed $1/c$ and its derivatives.

More precisely, $a_i$ can be chosen as a homogeneous differential polynomial of differential weight $2i-2$. This local structure will be established in Section~\ref{sec:approx-en}, where it becomes essential in the construction and estimate of higher-order approximate energies.

\end{em}
\end{rmk}

\begin{rmk}[Exact higher-order conservation laws]
\label{rem:exact-higher-order-conservation}
\begin{em}

The cascade also reveals a somewhat exceptional phenomenon. For special propagation speeds, the remainder produced at the last step of the construction may vanish identically, so that the corresponding higher-order energy is not merely almost conserved, but exactly conserved.

Indeed, the identities obtained in the proof show that
\begin{equation*}
    \mathcal{E}_{2i}'(t)
    =
    -\frac{a_i''(t)}{\lambda^{2i-2}}
    \ul(t)\ul'(t)
    \qquad\quad\text{and}\quad\qquad
    \mathcal{E}_{2i+1}'(t)
    =
    \frac{1}{2}\frac{a_i'''(t)}{\lambda^{2i-2}}
    \ul(t)^2.
\end{equation*}
Therefore $\mathcal{E}_{2i}$ is exactly conserved whenever the corresponding coefficient $a_i$ is affine, while $\mathcal{E}_{2i+1}$ is exactly conserved whenever $a_i$ is a polynomial of degree at most two.

The first cases already give nontrivial examples. Setting
\begin{equation*}
    \gamma:=\frac1c,
\end{equation*}
the energy $\mathcal{E}_1$ is exactly conserved when $\gamma$ is constant. The energy $\mathcal{E}_2$ is exactly conserved whenever
\begin{equation*}
    \gamma(t)=At+B,
\end{equation*}
whereas $\mathcal{E}_3$ is exactly conserved whenever
\begin{equation*}
    \gamma(t)=At^2+Bt+C,
\end{equation*}
provided of course that $\gamma$ remains positive on the time interval under consideration.

Starting from the fourth-order energy, genuinely nonlinear families appear. In terms of $\gamma$, the cascade relation for $a_2$ reads
\begin{equation*}
    \left(\frac{a_2}{\gamma}\right)'
    =
    -\frac14\gamma\gamma'''.
\end{equation*}
One possible choice satisfying this relation is
\begin{equation*}
    a_2
    =
    -\frac{\gamma}{4}
    \left(
        \gamma\gamma''
        -
        \frac12(\gamma')^2
    \right).
\end{equation*}
For this choice, the fourth-order energy is exactly conserved whenever $a_2''=0$, namely whenever $a_2$ is affine.

If we write
\begin{equation*}
    \gamma=y^2,
\end{equation*}
then
\begin{equation*}
    a_2=-\frac12y^5y''.
\end{equation*}
Thus the condition
\begin{equation*}
    a_2(t)=At+B
\end{equation*}
is equivalent to the nonlinear equation
\begin{equation*}
    y''(t)
    =
    -2(At+B)y(t)^{-5}.
\end{equation*}
Any positive solution of this equation therefore generates a propagation speed
\begin{equation*}
    c(t)=\frac1{y(t)^2}
\end{equation*}
for which the corresponding fourth-order corrected energy is an exact conservation law.

Thus the hierarchy of higher-order energies also singles out distinguished classes of time-dependent propagation speeds carrying higher-order conservation laws. The first levels correspond to reciprocal speeds that are constant, affine, or quadratic, while from the fourth level on the corresponding classes are described by nonlinear differential equations.

\end{em}
\end{rmk}


\setcounter{equation}{0}
\section{Higher-order approximate energies}
\label{sec:approx-en}

The higher-order energies introduced in Theorem~\ref{thm:cascade} are constructed under the assumption that the propagation speed is smooth. The basic idea of the approximate energy method of~\cite{1979-DGCS} is to replace the actual coefficient in the smooth energy by an auxiliary smooth one, whose choice is left open and may depend on the frequency parameter $\lambda$. The auxiliary coefficient is then selected so as to balance the regularity cost of the approximation with its distance from the original propagation speed.

Our aim in this section is to extend the same idea to the whole cascade of higher-order energies. Since the first coefficient of the cascade is the reciprocal propagation speed, the natural quantity to approximate is $1/c$. We therefore introduce a positive smooth function $\gamma$, thought of as an approximation of $1/c$, and repeat the construction of Theorem~\ref{thm:cascade} replacing $1/c$ by $\gamma$ throughout.

Accordingly, we set $a_1:=\gamma$ and require the subsequent coefficients to satisfy
\begin{equation}
    \left(\frac{a_{i+1}}{\gamma}\right)'
    =
    -\frac14\gamma a_i'''
    \qquad
    \forall i\geq1.
    \label{defn:rec-gamma}
\end{equation}
The integration constants will always be chosen according to the homogeneous normalization established in Lemma~\ref{lemma:pol-representation} below. Thus, unlike in the smooth construction of Section~\ref{sec:smooth}, we now select a specific local representative of the coefficients generated by the recursion.

For every solution $u_\lambda$ to~\eqref{eqn:ODE}, we define the first three energies as
\begin{gather*}
    \mathcal E_{1,\gamma}(t)
    :=
    a_1(t)\left(
        u_\lambda'(t)^2+\lambda^2\gamma(t)^{-2}u_\lambda(t)^2
    \right),
    \\
    \mathcal E_{2,\gamma}(t)
    :=
    \mathcal E_{1,\gamma}(t)
    -
    a_1'(t)u_\lambda(t)u_\lambda'(t),
    \qquad
    \mathcal E_{3,\gamma}(t)
    :=
    \mathcal E_{2,\gamma}(t)
    +
    \frac12a_1''(t)u_\lambda(t)^2.
\end{gather*}

Then, for every $i\geq2$ we set
\begin{equation*}
    \mathcal E_{2i,\gamma}(t)
    :=
    \mathcal E_{2i-1,\gamma}(t)
    +
    \frac1{\lambda^{2i-2}}
    \left[
        a_i(t)\left(
            u_\lambda'(t)^2+\lambda^2\gamma(t)^{-2}u_\lambda(t)^2
        \right)
        -
        a_i'(t)u_\lambda(t)u_\lambda'(t)
    \right],
\end{equation*}
and
\begin{equation*}
    \mathcal E_{2i+1,\gamma}(t)
    :=
    \mathcal E_{2i,\gamma}(t)
    +
    \frac{1}{2}\frac1{\lambda^{2i-2}}a_i''(t)u_\lambda(t)^2.
\end{equation*}

If $\gamma=1/c$, these energies are consistent with the choice of the coefficients in the construction of Theorem~\ref{thm:cascade}. The equation is still driven by the actual propagation speed $c$, whereas the energies are constructed using the smooth reciprocal coefficient $\gamma$. The discrepancy between $\gamma$ and $1/c$ produces an additional error term, which plays the role of a fidelity cost.

To measure simultaneously the regularity cost of the smooth approximation and its discrepancy from the reciprocal propagation speed, for every positive integer $k$, every $\lambda>0$, and every $\gamma\in C^k([0,T])$ we introduce
\begin{equation*}
    \F_{k,\lambda}(\gamma;c)
    :=
    \frac1{\lambda^{k-1}}
    \int_0^T|\gamma^{(k)}(t)|\,dt
    +
    \lambda
    \int_0^T
    \left|
        \gamma(t)-\frac1{c(t)}
    \right|\,dt.
\end{equation*}

The first term measures the regularity cost of the smooth approximation, while the second one measures its fidelity to the reciprocal propagation speed. The next result shows that this balance is precisely the one that controls the growth of the higher-order approximate energies.

\begin{thm}[Higher-order approximate energy estimates]
\label{thm:approximate-cascade}

Let $k$ be a positive integer, and let $T$, $\nu_1$, and $\nu_2$ be positive real numbers with $\nu_1\leq\nu_2$.

Let $c:(0,T)\to\mathbb R$ be a measurable function satisfying the strict hyperbolicity assumption~(\ref{hp:sh}), and let $\gamma\in C^k([0,T])$ satisfy
\begin{equation*}
    \frac1{\nu_2}
    \leq
    \gamma(t)
    \leq
    \frac1{\nu_1}
    \qquad
    \forall t\in[0,T].
\end{equation*}

There exist four positive constants $A_k$, $B_k$, $C_k$, $\ep_k$, depending only on $k$, $T$, $\nu_1$, and $\nu_2$, such that the following property holds.

For every $\lambda>0$ such that
\begin{equation*}
    |\gamma^{(j)}(t)|
    \leq
    \ep_k\lambda^j
    \qquad
    \forall t\in[0,T],
    \qquad
    j=1,\ldots,k-1,
\end{equation*}
and for every solution $\ul$ to~(\ref{eqn:ODE}), the energy $\mathcal E_{k,\gamma}$ defined above is uniformly equivalent to the standard energy~(\ref{defn:standard-energy}) in the sense that
\begin{equation*}
    A_k E_\lambda(t)
    \leq
    \mathcal E_{k,\gamma}(t)
    \leq
    B_k E_\lambda(t)
    \qquad
    \forall t\in[0,T],
\end{equation*}
and satisfies the growth/decay estimate
\begin{multline}
    \qquad
    \mathcal E_{k,\gamma}(0)
    \exp\left\{
        -C_k\left[
            \frac1{\lambda^{k-1}}
            +
            \F_{k,\lambda}(\gamma;c)
        \right]
    \right\}
    \leq
    \mathcal E_{k,\gamma}(t)
    \\
    \leq
    \mathcal E_{k,\gamma}(0)
    \exp\left\{
        C_k\left[
            \frac1{\lambda^{k-1}}
            +
            \F_{k,\lambda}(\gamma;c)
        \right]
    \right\}
    \qquad
    \label{th:approx-energy}
\end{multline}
for every $t\in[0,T]$.

\end{thm}

\begin{rmk}
\begin{em}

For $k\geq4$, the additional term $\lambda^{-(k-1)}$ in the exponential estimate is in general necessary. The functional $\F_{k,\lambda}(\gamma;c)$ only measures the $k$-th derivative of $\gamma$ and its fidelity to $1/c$, whereas the derivative of the $k$-th order energy may also contain products of lower-order derivatives of $\gamma$.

For example, if $\gamma=1/c$ is a nonconstant polynomial of degree strictly smaller than $k$, then $\gamma^{(k)}\equiv0$ and the fidelity term vanishes, so that $\F_{k,\lambda}(\gamma;c)=0$. Nevertheless, $\mathcal E_{k,\gamma}$ need not be exactly conserved, because the differential polynomial appearing in its derivative may still contain nonzero products of lower-order derivatives of $\gamma$.

The interpolation estimate used in the proof controls these terms by a constant, and this produces precisely a contribution of order $\lambda^{-(k-1)}$ in the exponent.

Notice also that, when $\gamma=1/c$, the fidelity term vanishes and, for all sufficiently large $\lambda$, the derivative constraints are automatically satisfied. The estimate then reduces to
\begin{equation*}
    \mathcal E_{k,\gamma}(0)
    \exp\left\{
        -\frac{C_k}{\lambda^{k-1}}
    \right\}
    \leq
    \mathcal E_{k,\gamma}(t)
    \leq
    \mathcal E_{k,\gamma}(0)
    \exp\left\{
        \frac{C_k}{\lambda^{k-1}}
    \right\},
\end{equation*}
which has exactly the form of the estimate in Theorem~\ref{thm:cascade}.

\end{em}
\end{rmk}

The remainder of this section is devoted to the proof of Theorem~\ref{thm:approximate-cascade}. We first examine the low-order cases, which make the cancellation mechanism and the role of the approximation error transparent. We then isolate the differential-polynomial structure of the coefficients and establish the interpolation estimate needed to control the final remainder. Finally, we combine these ingredients in the general inductive argument.

Since the propagation speed is only assumed to be measurable, all the differential identities and inequalities involving $c$ below are understood to hold for almost every $t\in(0,T)$.


\subsection{Low order cases}

The first four orders already contain all the mechanisms that enter the general proof. Besides proving the theorem in these cases, the computations below illustrate the successive cancellations, the role of the fidelity error, and the first appearance, at order four, of the interpolation argument needed to control the final remainder.

\subsubsection{Approximate energy of order one}

The first approximate energy is
\begin{equation*}
    \mathcal E_{1,\gamma}(t)
    =
    \gamma(t)\ul'(t)^2
    +
    \frac{\lambda^2}{\gamma(t)}\ul(t)^2.
\end{equation*}

Since $\gamma(t)$ is uniformly bounded from above and below, $\mathcal E_{1,\gamma}(t)$ is uniformly equivalent to the standard energy with constants that depend only on $\nu_1$ and $\nu_2$.

Taking~(\ref{eqn:ODE}) into account, its time derivative is
\begin{equation*}
    \mathcal E_{1,\gamma}'(t)=
    \gamma'(t)
    \left(
        \ul'(t)^2
        -
        \frac{\lambda^2}{\gamma(t)^2}\ul(t)^2
    \right)
    +
    2\lambda^2
    \left(
        \frac1{\gamma(t)}-\gamma(t)c(t)^2
    \right)
    \ul(t)\ul'(t).
\end{equation*}

In order to estimate this derivative, we observe that
\begin{equation*}
    \left|\frac1{\gamma(t)}-\gamma(t)c(t)^2\right|
    =
    \left(c(t)^2+\frac{c(t)}{\gamma(t)}\right)\left|\gamma(t)-\frac1{c(t)}\right|
    \leq
    2\nu_2^2\left|\gamma(t)-\frac1{c(t)}\right|
\end{equation*}
and
\begin{equation}
    2\lambda|\ul(t)\ul'(t)|
    \leq
    \mathcal E_{1,\gamma}(t),
    \label{est:uu'-1}
\end{equation}
while
\begin{equation*}
    \left|\gamma'(t)\left(\ul'(t)^2-\frac{\lambda^2}{\gamma(t)^2}\ul(t)^2\right)\right|
    \leq
    \frac{|\gamma'(t)|}{\gamma(t)}\mathcal E_{1,\gamma}(t)
    \leq
    \nu_2|\gamma'(t)|\mathcal E_{1,\gamma}(t).
\end{equation*}
We deduce that
\begin{equation*}
    |\mathcal E_{1,\gamma}'(t)|
    \leq
    \left\{
        \nu_2|\gamma'(t)|
        +
        2\nu_2^2\lambda
        \left|\gamma(t)-\frac1{c(t)}\right|
    \right\}
    \mathcal E_{1,\gamma}(t).
\end{equation*}

Integrating this differential inequality yields
\begin{equation*}
    \mathcal E_{1,\gamma}(0)
    \exp\left\{
        -C_1\F_{1,\lambda}(\gamma;c)
    \right\}
    \leq
    \mathcal E_{1,\gamma}(t)
    \leq
    \mathcal E_{1,\gamma}(0)
    \exp\left\{C_1\F_{1,\lambda}(\gamma;c)\right\}
\end{equation*}
with $C_1:=\max\{\nu_2,2\nu_2^2\}$. This proves~(\ref{th:approx-energy}) for $k=1$, with the stronger estimate in which the additional constant term in the exponent is absent.

\subsubsection{Approximate energy of order two}

The second energy of the cascade is
\begin{equation*}
    \mathcal E_{2,\gamma}(t)
    =
    \mathcal E_{1,\gamma}(t)
    -
    \gamma'(t)\ul(t)\ul'(t).
\end{equation*}
Since
\begin{equation*}
    2\lambda|\ul(t)\ul'(t)|
    \leq
    \mathcal E_{1,\gamma}(t),
\end{equation*}
the assumption
\begin{equation*}
    |\gamma'(t)|
    \leq
    \ep_2\lambda
    \qquad
    \forall t\in[0,T],
\end{equation*}
with $\ep_2:=1$, implies
\begin{equation}
    \frac12\mathcal E_{1,\gamma}(t)
    \leq
    \mathcal E_{2,\gamma}(t)
    \leq
    \frac32\mathcal E_{1,\gamma}(t)
    \qquad
    \forall t\in[0,T].
    \label{eqn:equiv-2}
\end{equation}
In particular, $\mathcal E_{2,\gamma}$ is uniformly equivalent to $\mathcal E_{1,\gamma}$, and hence also to the standard energy.

Differentiating and using~(\ref{eqn:ODE}) gives (for shortness we do not write the dependence on $t$)
\begin{equation*}
    \mathcal E_{2,\gamma}'
    ={}
    -\gamma''\ul\ul'
    +
    \lambda^2
    \left(
        c^2-\frac1{\gamma^2}
    \right)
    \left(
        \gamma'\ul^2
        -
        2\gamma\ul\ul'
    \right).
\end{equation*}
Moreover,
\begin{equation}
    \left|
        c^2-\frac1{\gamma^2}
    \right|
    =
    \frac{c}{\gamma}
    \left(
        c+\frac1{\gamma}
    \right)
    \left|
        \gamma-\frac1{c}
    \right|
    \leq
    2\nu_2^3
    \left|
        \gamma-\frac1{c}
    \right|.
    \label{est:c^2-gamma2}
\end{equation}

From (\ref{est:uu'-1}) and (\ref{eqn:equiv-2}) we deduce that
\begin{equation}
    |\ul(t)\ul'(t)|
    \leq
    \frac{1}{2\lambda}\mathcal E_{1,\gamma}(t)
    \leq
    \frac{1}{\lambda}\mathcal E_{2,\gamma}(t)
    \label{est:uu'-2}
\end{equation}
and
\begin{equation}
    \ul(t)^2
    \leq
    \frac{\gamma(t)}{\lambda^2}\mathcal E_{1,\gamma}(t)
    \leq
    \frac{1}{\nu_1\lambda^2}\mathcal E_{1,\gamma}(t)
    \leq
    \frac{2}{\nu_1\lambda^2}
    \mathcal E_{2,\gamma}(t),
    \label{est:u2-2}
\end{equation}
and therefore
\begin{equation*}
    |\mathcal E_{2,\gamma}'(t)|
    \leq
    \left\{
        \frac{|\gamma''(t)|}{\lambda}
        +
        \frac{8\nu_2^3}{\nu_1}
        \lambda
        \left|
            \gamma(t)-\frac1{c(t)}
        \right|
    \right\}
    \mathcal E_{2,\gamma}(t).
\end{equation*}

Integrating this differential inequality yields
\begin{equation*}
    \mathcal E_{2,\gamma}(0)
    \exp\left\{
        -C_2\F_{2,\lambda}(\gamma;c)
    \right\}
    \leq
    \mathcal E_{2,\gamma}(t)
    \leq
    \mathcal E_{2,\gamma}(0)
    \exp\left\{C_2\F_{2,\lambda}(\gamma;c)\right\}
\end{equation*}
with $C_2:=\max\{1,8\nu_2^3/\nu_1\}$. This proves~(\ref{th:approx-energy}) for $k=2$, with the stronger estimate in which the additional term $\lambda^{-1}$ is absent.

\subsubsection{Approximate energy of order three}

The third energy is
\begin{equation}
    \mathcal E_{3,\gamma}(t)
    =
    \mathcal E_{2,\gamma}(t)
    +
    \frac12\gamma''(t)\ul(t)^2.
    \label{defn:E3-gamma}
\end{equation}

If we assume that
\begin{equation*}
    |\gamma'(t)|
    \leq
    \ep_3\lambda
    \qquad\text{and}\qquad
    |\gamma''(t)|
    \leq
    \ep_3\lambda^2
    \qquad
    \forall t\in[0,T],
\end{equation*}
with $\ep_3:=\min\{1,\nu_1/2\}$, then from (\ref{est:u2-2}) we obtain that
\begin{equation*}
    \frac12\mathcal E_{2,\gamma}(t)
    \leq
    \mathcal E_{3,\gamma}(t)
    \leq
    \frac32\mathcal E_{2,\gamma}(t)
    \qquad
    \forall t\in[0,T].
\end{equation*}
In particular, $\mathcal E_{3,\gamma}(t)$ is uniformly equivalent to $\mathcal E_{2,\gamma}(t)$, which in turn is equivalent to the standard energy. As a consequence, now (\ref{est:uu'-2}) and (\ref{est:u2-2}) imply
\begin{equation}
    |\ul(t)\ul'(t)|
    \leq
    \frac{2}{\lambda}\mathcal E_{3,\gamma}(t)
    \qquad\quad\text{and}\quad\qquad
    \ul(t)^2
    \leq
    \frac{4}{\nu_1\lambda^2}\mathcal E_{3,\gamma}(t).
    \label{est:uu'-u2-3}
\end{equation}

Differentiating (\ref{defn:E3-gamma}), the term containing $\gamma''(t)\ul(t)\ul'(t)$ cancels, and we obtain
\begin{equation*}
    \mathcal E_{3,\gamma}'(t)=
    \frac12\gamma'''(t)\ul(t)^2+
    \lambda^2
    \left(
        c(t)^2-\frac1{\gamma(t)^2}
    \right)
    \left(
        \gamma'(t)\ul(t)^2
        -
        2\gamma(t)\ul(t)\ul'(t)
    \right).
\end{equation*}

Now from (\ref{est:uu'-u2-3}) we obtain that
\begin{equation*}
    |\gamma'(t)|\ul(t)^2
    +
    2\gamma(t)|\ul(t)\ul'(t)|
    \leq
    \frac{8}{\nu_1\lambda}
    \mathcal E_{3,\gamma}(t),
\end{equation*}
and taking (\ref{est:c^2-gamma2}) into account we deduce that
\begin{equation*}
    |\mathcal E_{3,\gamma}'(t)|
    \leq
    \left\{
        \frac{2}{\nu_1}
        \frac{|\gamma'''(t)|}{\lambda^2}
        +
        \frac{16\nu_2^3}{\nu_1}
        \lambda
        \left|
            \gamma(t)-\frac1{c(t)}
        \right|
    \right\}
    \mathcal E_{3,\gamma}(t).
\end{equation*}

Integrating this differential inequality yields
\begin{equation*}
    \mathcal E_{3,\gamma}(0)
    \exp\left\{
        -C_3\F_{3,\lambda}(\gamma;c)
    \right\}
    \leq
    \mathcal E_{3,\gamma}(t)
    \leq
    \mathcal E_{3,\gamma}(0)
    \exp\left\{C_3\F_{3,\lambda}(\gamma;c)\right\}
\end{equation*}
with $C_3:=\max\{2/\nu_1,16\nu_2^3/\nu_1\}$. This proves~(\ref{th:approx-energy}) for $k=3$, with the stronger estimate in which the additional term $\lambda^{-2}$ is absent.

\subsubsection{Approximate energy of order four}

This is the first case in which a nontrivial coefficient appears in the higher-order cascade. Solving the recursion for $a_2$, we choose
\begin{equation*}
    a_2(t)
    =
    -\frac{\gamma(t)}4
    \left(
        \gamma(t)\gamma''(t)
        -
        \frac12\gamma'(t)^2
    \right),
\end{equation*}
and the fourth energy is
\begin{equation}
    \mathcal E_{4,\gamma}(t)
    :=
    \mathcal E_{3,\gamma}(t)+
    \frac1{\lambda^2}
    \left[
        a_2(t)
        \left(
            \ul'(t)^2
            +
            \frac{\lambda^2}{\gamma(t)^2}\ul(t)^2
        \right)
        -
        a_2'(t)\ul(t)\ul'(t)
    \right].
    \label{defn:E4-gamma}
\end{equation}

Assuming in addition that $\ep_4\leq\min\{1,\ep_3\}$, standard calculations show that
\begin{equation}
    |\gamma^{(j)}(t)|
    \leq
    \ep_4\lambda^j
    \qquad
    \forall t\in[0,T]
    \qquad
    \forall j\in\{1,2,3\}
    \label{defn:ep4}
\end{equation}
imply that
\begin{equation}
    |a_2(t)|
    \leq
    C_{4,1}\ep_4\lambda^2
    \qquad\text{and}\qquad
    |a_2'(t)|
    \leq
    C_{4,2}\ep_4\lambda^3
    \qquad
    \forall t\in[0,T]
    \label{est:a2}
\end{equation}
for suitable constants $C_{4,1}$ and $C_{4,2}$ that depend only on $\nu_1$ and $\nu_2$. By choosing $\ep_4$ possibly smaller, depending only on $\nu_1$ and $\nu_2$, we obtain
\begin{equation*}
    \frac12\mathcal E_{3,\gamma}(t)
    \leq
    \mathcal E_{4,\gamma}(t)
    \leq
    \frac32\mathcal E_{3,\gamma}(t)
    \qquad
    \forall t\in[0,T],
\end{equation*}
which means that $\mathcal E_{4,\gamma}(t)$ is uniformly equivalent to $\mathcal E_{3,\gamma}(t)$, and hence also to the standard energy. Consequently, now (\ref{est:uu'-u2-3}) implies that
\begin{equation}
    |\ul(t)\ul'(t)|
    \leq
    \frac{4}{\lambda}\mathcal E_{4,\gamma}(t)
    \qquad\quad\text{and}\quad\qquad
    \ul(t)^2
    \leq
    \frac{8}{\nu_1\lambda^2}\mathcal E_{4,\gamma}(t).
    \label{est:uu'-u2-4}
\end{equation}

Differentiating (\ref{defn:E4-gamma}), and using~(\ref{eqn:ODE}) and the recursive identity (\ref{defn:rec-gamma}) with $i=1$, we deduce that (for shortness we do not write the dependence on $t$)
\begin{equation*}
    \mathcal E_{4,\gamma}'
    =
    -\frac{a_2''}{\lambda^2}\ul\ul'
    +
    \left(c^2-\frac1{\gamma^2}\right)
    \cdot
    \left[
        \left(\lambda^2\gamma'+a_2'\right)\ul^2
        -
        2\left(\lambda^2\gamma+a_2\right)\ul\ul'
    \right].
\end{equation*}

Exploiting (\ref{defn:ep4}), (\ref{est:c^2-gamma2}), (\ref{est:a2}) and (\ref{est:uu'-u2-4}), with some patience we obtain that
\begin{equation*}
    |\mathcal E_{4,\gamma}'(t)|
    \leq
    C_{4,3}
    \left\{
        \frac{|a_2''(t)|}{\lambda^3}
        +
        \lambda
        \left|
            \gamma(t)-\frac1{c(t)}
        \right|
    \right\}
    \mathcal E_{4,\gamma}(t).
\end{equation*}

Integrating this differential inequality we obtain
\begin{equation}
    \mathcal E_{4,\gamma}(t)
    \leq
    \mathcal E_{4,\gamma}(0)
    \exp\left(
        \frac{C_{4,3}}{\lambda^3}\int_0^T|a_2''(t)|\,dt
        +
        C_{4,3}\lambda\int_0^T\left|\gamma(t)-\frac1{c(t)}\right|\,dt
    \right),
    \label{est:E4-a2}
\end{equation}
and the corresponding estimate from below.

On the other hand, a direct computation gives
\begin{equation}
    a_2''(t)
    =
    -\frac14\left(
        \gamma(t)^2\gamma^{(4)}(t)
        +
        3\gamma(t)\gamma'(t)\gamma'''(t)
        +
        \gamma(t)\gamma''(t)^2
        -
        \frac12\gamma'(t)^2\gamma''(t)
    \right),
    \label{defn:a2''}
\end{equation}
and therefore
\begin{equation*}
    |a_2''(t)|
    \leq
    C_{4,4}\left(
        |\gamma^{(4)}(t)|
        +
        |\gamma'(t)|\cdot|\gamma'''(t)|
        +
        |\gamma''(t)|^2
        +
        |\gamma'(t)|^2\cdot|\gamma''(t)|
    \right),
\end{equation*}
where $C_{4,4}$ depends only on $\nu_1$. At this point the one-dimensional interpolation inequalities (see Lemma~\ref{lemma:first-derivative-reduction} below) imply that
\begin{equation*}
    \int_0^T|a_2''(t)|\,dt
    \leq
    C_{4,5}
    \left(
        1+
        \int_0^T|\gamma^{(4)}(t)|\,dt
    \right),
\end{equation*}
where now the constant depends also on $T$.

Plugging this estimate into (\ref{est:E4-a2}), and the corresponding estimate from below, we obtain exactly (\ref{th:approx-energy}) with $k=4$. In this case the new constant $C_4$ depends on $\nu_1$, $\nu_2$, $T$, and for the first time we have the term $1/\lambda^3$ that comes from the interpolation inequalities.


\subsection{Polynomial representation}

The low-order computations suggest that the coefficients generated by the recursive construction have a precise algebraic structure. In order to formulate this structure and exploit it in the general case, we introduce the following terminology.

\begin{defn}[Differential monomials and polynomials]
\begin{em}

Let $k$ be a positive integer.
\begin{itemize}
    \item A \emph{differential monomial} of \emph{differential weight} $k$ is a monomial in the formal variables $X_0$, $X_1$, \ldots, $X_k$ of the form
    \begin{equation}
        X_0^m\prod_{\ell=1}^n X_{j_\ell},
        \label{defn:diff-mon}
    \end{equation}
    where $m$ is a nonnegative integer, $n$ is a positive integer, and $j_1,\ldots,j_n$ are positive integers such that
    \begin{equation*}
        j_1+\cdots+j_n=k.
    \end{equation*}

    \item A \emph{homogeneous differential polynomial} of \emph{differential weight} $k$ is a finite linear combination, with real coefficients, of differential monomials of differential weight~$k$.

    \item If $\gamma$ is a function of class $C^k$, and $P$ is such a polynomial, its evaluation along $\gamma$ is the function
    \begin{equation*}
        P[\gamma](t)
        :=
        P\bigl(\gamma(t),\gamma'(t),\ldots,\gamma^{(k)}(t)\bigr).
    \end{equation*}
\end{itemize}

\end{em}
\end{defn}

A differential polynomial is therefore a formal algebraic expression whose variables represent a function and its derivatives. Its differential weight is the sum of the orders of the differentiated factors, while the variable $X_0$, representing the undifferentiated function, has weight zero. For instance, the expression of $a_2''$ obtained in~(\ref{defn:a2''}) is the evaluation along $\gamma$ of the homogeneous differential polynomial
\begin{equation*}
    -\frac14
    \left(
        X_0^2X_4
        +
        3X_0X_1X_3
        +
        X_0X_2^2
        -
        \frac12X_1^2X_2
    \right),
\end{equation*}
which has differential weight $4$.

The low-order computations suggest that the coefficients generated by the recursion are evaluations of homogeneous differential polynomials, with a differential weight determined by their position in the cascade. Since the recursion determines each coefficient only up to the addition of a multiple of $\gamma$, this amounts to selecting a distinguished local normalization of the integration constants. This structure is the key point that allows the final remainder to be controlled by interpolation in terms of the highest derivative of $\gamma$.

The polynomial structure underlying this type of higher-order quadratic construction was established by R.~Manfrin~\cite{Manfrin2008}, where graded families of differential polynomials arise through repeated integrations by parts in the recursive cancellation procedure. A closely related formulation has recently been used by C.~Boiti and R.~Manfrin in~\cite{BoitiManfrin2026}. The following lemma shows that such a homogeneous normalization can always be achieved and records the precise form that will be needed here.

\begin{lemma}[Representation as homogeneous differential polynomials]
\label{lemma:pol-representation}

Let $k\geq4$ be a positive integer, let $T$ be a positive real number, and let $\gamma\in C^k([0,T])$ be a function with positive values.

Then there exist functions
\begin{equation*}
    a_i\in C^{k-2i+2}([0,T]),
    \qquad
    i\in\{1,\ldots,\lfloor k/2\rfloor\},
\end{equation*}
with the following properties:
\begin{itemize}
    \item $a_1=\gamma$ and
    \begin{equation*}
        \left(\frac{a_i}{\gamma}\right)'
        =
        -\frac14\gamma a_{i-1}'''
        \qquad
        \forall i\in\{2,\ldots,\lfloor k/2\rfloor\};
    \end{equation*}

    \item for every $i\in\{2,\ldots,\lfloor k/2\rfloor\}$, there exists a homogeneous differential polynomial $P_i$ of differential weight $2i-2$ such that
    \begin{equation*}
        a_i(t)=P_i[\gamma](t)
        \qquad
        \forall t\in[0,T].
    \end{equation*}
\end{itemize}

\end{lemma}

\begin{proof}

For the convenience of the reader, we give a self-contained proof. The key point is the transport identity (\ref{eqn:transport}) that provides a compact form of the repeated integration-by-parts mechanism described above.

We set
\begin{equation*}
    f_i:=\frac{a_i}{\gamma}.
\end{equation*}
Thus $f_1=1$, and the recursive relation becomes
\begin{equation}
    f_i'
    =
    -\frac14\gamma a_{i-1}'''
    \qquad
    \forall i\geq2.
    \label{eqn:fi-recursion}
\end{equation}

For $i=2$, an elementary calculation shows that we can choose
\begin{equation*}
    a_2
    =
    -\frac{\gamma}{4}
    \left(
        \gamma\gamma''
        -
        \frac12(\gamma')^2
    \right),
\end{equation*}
which is the evaluation along $\gamma$ of a homogeneous differential polynomial of differential weight~$2$. Moreover, also $f_2=a_2/\gamma$ has the same property.

We now assume that $a_1,\ldots,a_i$ have already been constructed and, more precisely, that for every $j\in\{2,\ldots,i\}$ the functions $f_j$ and $a_j=\gamma f_j$ are evaluations along $\gamma$ of homogeneous differential polynomials of differential weight $2j-2$.

For two expressions $F$ and $G$ which are evaluations along $\gamma$ of homogeneous differential polynomials of the same differential weight $r$, we write
\begin{equation*}
    F\equiv G
\end{equation*}
if there exists an expression $H$, which is the evaluation along $\gamma$ of a homogeneous differential polynomial of differential weight $r-1$, such that
\begin{equation*}
    F-G=H'.
\end{equation*}
Thus $F\equiv0$ means that $F$ admits a primitive which is the evaluation along $\gamma$ of a homogeneous differential polynomial, with differential weight decreased by one.

We first establish the transport identity
\begin{equation}
    f_m f_n'
    \equiv
    f_{m-1}f_{n+1}',
    \label{eqn:transport}
\end{equation}
for all pairs $(m,n)$ such that $2\leq m\leq i$ and $2\leq n\leq i-1$. Indeed, integration by parts and~(\ref{eqn:fi-recursion}) yield
\begin{equation*}
    f_m f_n'
    \equiv
    -f_m'f_n
    =
    \frac14\gamma a_{m-1}'''f_n
    =
    \frac14a_{m-1}'''a_n
    \equiv
    -\frac14a_{m-1}a_n'''
    =
    f_{m-1}f_{n+1}'.
\end{equation*}

We claim that
\begin{equation}
    f_i f_2'\equiv0.
    \label{eqn:fi-f2}
\end{equation}
If $i=2r$, repeated application of~(\ref{eqn:transport}) gives
\begin{equation*}
    f_{2r}f_2'
    \equiv
    f_{2r-1}f_3'
    \equiv
    \cdots
    \equiv
    f_{r+1}f_{r+1}'
    =
    \frac12\left(f_{r+1}^2\right)',
\end{equation*}
and therefore~(\ref{eqn:fi-f2}) follows.

If $i=2r+1$, the same argument gives
\begin{equation*}
    f_{2r+1}f_2'
    \equiv
    f_{r+1}f_{r+2}'.
\end{equation*}
On the other hand, integration by parts and~(\ref{eqn:transport}) yield
\begin{equation*}
    f_{r+1}f_{r+2}'
    \equiv
    -f_{r+2}f_{r+1}'
    \equiv
    -f_{r+1}f_{r+2}'.
\end{equation*}
Hence
\begin{equation*}
    f_{r+1}f_{r+2}'\equiv0,
\end{equation*}
and~(\ref{eqn:fi-f2}) follows also in this case.

Now from the identity
\begin{equation*}
    \left(
        \gamma a_i''
        -
        \gamma'a_i'
        +
        \gamma''a_i
    \right)'
    =
    \gamma a_i'''
    +
    \gamma'''a_i
\end{equation*}
we deduce that
\begin{equation*}
    \gamma a_i'''\equiv-\gamma'''a_i.
\end{equation*}
Recalling that $a_i=\gamma f_i$, and that
\begin{equation*}
    \gamma\gamma'''=-4f_2'
\end{equation*}
because of~(\ref{eqn:fi-recursion}) with $i=2$, we conclude that
\begin{equation*}
    \gamma a_i'''
    \equiv
    -\gamma'''a_i
    =
    -\gamma'''\gamma f_i
    =
    4f_2'f_i
    \equiv0.
\end{equation*}

Since $\gamma a_i'''$ is the evaluation along $\gamma$ of a homogeneous differential polynomial of differential weight $2i+1$, the relation $\gamma a_i'''\equiv0$ means that there exists a function $h_i$, which is the evaluation along $\gamma$ of a homogeneous differential polynomial of differential weight $2i$, such that
\begin{equation*}
    h_i'
    =
    \gamma a_i'''.
\end{equation*}

In particular, the function
\begin{equation*}
    a_{i+1}:=-\frac14\gamma h_i
\end{equation*}
satisfies the recurrence (\ref{defn:rec-gamma}). Moreover, since multiplication by the undifferentiated function $\gamma$ does not change the differential weight, $a_{i+1}$ is the evaluation along $\gamma$ of a homogeneous differential polynomial of differential weight $2i$. This completes the induction.

Finally, a homogeneous differential polynomial of differential weight $2i-2$ involves derivatives of $\gamma$ of order at most $2i-2$, and hence $a_i$ belongs to $C^{k-2i+2}([0,T])$ for every $i\in\{1,\ldots,\lfloor k/2\rfloor\}$.
\end{proof}


The previous lemma identifies the differential-polynomial structure of the coefficients generated by the recursion. In the general argument, the last remainder will have differential weight exactly $k$, and we shall need to estimate its integral in terms of the $k$-th derivative of $\gamma$. This is provided by the following interpolation lemma.

\begin{lemma}[Interpolation inequality for homogeneous differential polynomials]
\label{lem:weighted-polynomial}

Let $T$ be a positive real number, let $k$ be a positive integer, let $\gamma\in C^k([0,T])$, and let $P$ be a homogeneous differential polynomial of differential weight $k$.

Assume that there exists a constant $M$ such that
\begin{equation}
    |\gamma(t)|\leq M
    \qquad
    \forall t\in[0,T].
    \label{hp:gamma-bounded}
\end{equation}

Then there exists a constant $C$, depending only on $P$, $k$, $T$, and $M$, such that
\begin{equation*}
    \int_0^T |P[\gamma](t)|\,dt
    \leq
    C\left(
        1+
        \int_0^T|\gamma^{(k)}(t)|\,dt
    \right).
\end{equation*}

\end{lemma}

\begin{proof}

Every monomial of $P$ has the form (\ref{defn:diff-mon}). The factor $\gamma^m$ is bounded by $M^m$ and can therefore be absorbed into the constant. Setting $p_\ell:=k/j_\ell$, we observe that
\begin{equation}
    \sum_{\ell=1}^n\frac1{p_\ell}=1,
    \label{eqn:reciprocal}
\end{equation}
and therefore H\"older's inequality gives
\begin{equation}
    \int_0^T
    \prod_{\ell=1}^n
    \left|\gamma^{(j_\ell)}(t)\right|
    \,dt
    \leq
    \prod_{\ell=1}^n
    \left\|\gamma^{(j_\ell)}\right\|_{L^{k/j_\ell}(0,T)}.
    \label{post-holder}
\end{equation}

By the one-dimensional Gagliardo--Nirenberg interpolation inequality, for every integer $j\in\{1,\ldots,k\}$ we have
\begin{equation*}
    \left\|\gamma^{(j)}\right\|_{L^{k/j}(0,T)}
    \leq
    C_1
    \left\|\gamma^{(k)}\right\|_{L^1(0,T)}^{j/k}
    \left\|\gamma\right\|_{L^\infty(0,T)}^{1-j/k}
    +
    C_1\left\|\gamma\right\|_{L^\infty(0,T)},
\end{equation*}
where $C_1$ depends only on $k$, $j$, and $T$. Since $j$ ranges in the finite set $\{1,\ldots,k\}$, and taking~(\ref{hp:gamma-bounded}) into account, there exists a constant $C_2$, depending only on $k$, $T$, and $M$, such that
\begin{equation*}
    \left\|\gamma^{(j)}\right\|_{L^{k/j}(0,T)}
    \leq
    C_2\left(1+\left\|\gamma^{(k)}\right\|_{L^1(0,T)}^{j/k}\right)
    \leq
    C_2\cdot 2^{1-j/k}
    \left(1+\left\|\gamma^{(k)}\right\|_{L^1(0,T)}\right)^{j/k}
\end{equation*}
for every $j\in\{1,\ldots,k\}$, where the last inequality follows from the concavity of the function $x\mapsto x^\alpha$ when $\alpha\in(0,1]$.

Coming back to~(\ref{post-holder}), and recalling~(\ref{eqn:reciprocal}), we obtain
\begin{equation*}
    \int_0^T
    \prod_{\ell=1}^n
    \left|\gamma^{(j_\ell)}(t)\right|
    \,dt
    \leq
    C_2^n
    \prod_{\ell=1}^n
    \left(1+\left\|\gamma^{(k)}\right\|_{L^1(0,T)}^{j_\ell/k}\right)
    \leq
    C_2^n\cdot 2^{n-1}
    \left(1+\left\|\gamma^{(k)}\right\|_{L^1(0,T)}\right).
\end{equation*}

Since $n\leq k$, this proves the required estimate for each monomial. Since $P$ is a finite linear combination of such monomials, the conclusion follows.
\end{proof}


\subsection{The general case}

The previous subsection provides the structural information needed to extend the low-order computations to arbitrary order. We now combine the recursive cancellation mechanism of Theorem~\ref{thm:cascade} with the polynomial representation of the coefficients and the interpolation estimate for the final remainder.

\paragraph{\textmd{\textit{Recursive definition of energies and computation of derivatives}}}

The first four cases have already been proved explicitly. We now give the general inductive construction, which starts from $\mathcal E_{3,\gamma}$ and reproduces all the subsequent energies. For every positive integer $i$, we set
\begin{equation*}
    Q_\gamma(t):=\ul'(t)^2+\frac{\lambda^2}{\gamma(t)^2}\ul(t)^2
    \qquad\quad\text{and}\quad\qquad
    S_i(t):=\sum_{j=1}^i\frac{a_j(t)}{\lambda^{2j-2}}.
\end{equation*}

For brevity, in the sequel we do not write the explicit dependence on $t$. Let $i\geq2$, and assume that the odd energy $\mathcal E_{2i-1,\gamma}$ has already been constructed and satisfies
\begin{equation}
    \mathcal E_{2i-1,\gamma}'
    =
    \frac{1}{2}\frac{a_{i-1}'''}{\lambda^{2i-4}}\ul^2
    +
    \lambda^2\left(c^2-\frac1{\gamma^2}\right)
    \left(S_{i-1}'\ul^2-2S_{i-1}\ul\ul'\right).
    \label{eqn:induction-odd}
\end{equation}
The next two energies are defined by
\begin{equation*}
    \mathcal E_{2i,\gamma}(t)
    :=
    \mathcal E_{2i-1,\gamma}(t)
    +
    \frac1{\lambda^{2i-2}}
    \left[
        a_i(t)Q_\gamma(t)-a_i'(t)\ul(t)\ul'(t)
    \right]
\end{equation*}
and
\begin{equation*}
    \mathcal E_{2i+1,\gamma}(t)
    :=
    \mathcal E_{2i,\gamma}(t)
    +
    \frac{1}{2}\frac1{\lambda^{2i-2}}a_i''(t)\ul(t)^2.
\end{equation*}

A direct computation gives
\begin{equation*}
    \left[a_iQ_\gamma-a_i'\ul\ul'\right]'
    =
    2\lambda^2\frac1\gamma\left(\frac{a_i}{\gamma}\right)'\ul^2
    -
    a_i''\ul\ul'
    +
    \lambda^2\left(c^2-\frac1{\gamma^2}\right)
    \left(a_i'\ul^2-2a_i\ul\ul'\right).
\end{equation*}

When computing the derivative of $\mathcal E_{2i,\gamma}$, we use the recursive identity
\begin{equation*}
    \left(\frac{a_i}{\gamma}\right)'
    =
    -\frac14\gamma a_{i-1}''',
\end{equation*}
and we see that the first term produced by the new block cancels exactly the first term in~(\ref{eqn:induction-odd}). Since
\begin{equation*}
    S_i=S_{i-1}+\frac{a_i}{\lambda^{2i-2}},
\end{equation*}
we obtain
\begin{equation}
    \mathcal E_{2i,\gamma}'
    =
    -\frac{a_i''}{\lambda^{2i-2}}\ul\ul'
    +
    \lambda^2\left(c^2-\frac1{\gamma^2}\right)
    \left(S_i'\ul^2-2S_i\ul\ul'\right).
    \label{eqn:general-even-derivative}
\end{equation}

When computing the derivative of $\mathcal E_{2i+1,\gamma}$, one of the terms coming from the derivative of the correction term cancels $a_i''\ul\ul'$, and we end up with
\begin{equation}
    \mathcal E_{2i+1,\gamma}'
    =
    \frac{1}{2}\frac{a_i'''}{\lambda^{2i-2}}\ul^2
    +
    \lambda^2\left(c^2-\frac1{\gamma^2}\right)
    \left(S_i'\ul^2-2S_i\ul\ul'\right).
    \label{eqn:general-odd-derivative}
\end{equation}

This has exactly the form of~(\ref{eqn:induction-odd}) with $i$ replaced by $i+1$, and hence completes the induction.

\paragraph{\textmd{\textit{Equivalence with the standard energy}}}

The polynomial representation of Lemma~\ref{lemma:pol-representation} now becomes essential. It allows us to convert the smallness assumptions on the derivatives of $\gamma$ into pointwise bounds for all the coefficients appearing in the higher-order corrections.

In what follows, $C$ denotes a positive constant depending only on $k$, $\nu_1$, and $\nu_2$, whose value may change from line to line. When the interpolation estimate is used below, the constant will also be allowed to depend on $T$.

Indeed, for every $j\geq2$, the coefficient $a_j$ is the evaluation along $\gamma$ of a homogeneous differential polynomial of differential weight $2j-2$, and therefore $a_j^{(r)}$ is the evaluation along $\gamma$ of a homogeneous differential polynomial of differential weight $2j-2+r$. Since every monomial involved contains at least one differentiated factor, the assumptions
\begin{equation*}
    |\gamma^{(r)}(t)|
    \leq
    \ep_k\lambda^r
    \qquad
    \forall t\in[0,T],
    \qquad
    \forall r\in\{1,\ldots,k-1\},
\end{equation*}
with $\ep_k\leq1$, imply
\begin{equation}
    |a_j^{(r)}(t)|
    \leq
    C\ep_k\lambda^{2j-2+r}
    \label{est:general-aj}
\end{equation}
whenever
\begin{equation*}
    2j-2+r\leq k-1.
\end{equation*}
The same estimate holds for $a_1^{(r)}=\gamma^{(r)}$ whenever $1\leq r\leq k-1$, while $a_1=\gamma$ remains uniformly bounded from above and below. These estimates cover all the coefficients needed to prove the equivalence of $\mathcal E_{k,\gamma}$ with the standard energy. The derivatives of total differential weight $k$ which appear when differentiating the final energy will instead be estimated separately.

We prove the equivalence with the standard energy along the recursive construction of the cascade. Let $i\geq2$ be such that $2i\leq k$, and assume that $\mathcal E_{2i-1,\gamma}$ is uniformly equivalent to the standard energy. Since
\begin{equation*}
    \mathcal E_{2i,\gamma}(t)
    =
    \mathcal E_{2i-1,\gamma}(t)
    +
    \frac1{\lambda^{2i-2}}
    \left[
        a_i(t)Q_\gamma(t)
        -
        a_i'(t)\ul(t)\ul'(t)
    \right],
\end{equation*}
we have
\begin{equation*}
    |a_i(t)|
    \leq
    C\ep_k\lambda^{2i-2}
    \qquad\quad\text{and}\quad\qquad
    |a_i'(t)|
    \leq
    C\ep_k\lambda^{2i-1}.
\end{equation*}
On the other hand, the equivalence of $\mathcal E_{2i-1,\gamma}$ with the standard energy, together with the uniform bounds on $\gamma$, implies
\begin{equation*}
    Q_\gamma(t)
    \leq
    C\mathcal E_{2i-1,\gamma}(t)
    \qquad\quad\text{and}\quad\qquad
    |\ul(t)\ul'(t)|
    \leq
    \frac{C}{\lambda}
    \mathcal E_{2i-1,\gamma}(t).
\end{equation*}
Therefore
\begin{equation*}
    \left|
        \mathcal E_{2i,\gamma}(t)
        -
        \mathcal E_{2i-1,\gamma}(t)
    \right|
    \leq
    C\ep_k
    \mathcal E_{2i-1,\gamma}(t).
\end{equation*}
Choosing $\ep_k$ sufficiently small, depending only on $k$, $\nu_1$, and $\nu_2$, we obtain
\begin{equation*}
    \frac12\mathcal E_{2i-1,\gamma}(t)
    \leq
    \mathcal E_{2i,\gamma}(t)
    \leq
    \frac32\mathcal E_{2i-1,\gamma}(t)
    \qquad
    \forall t\in[0,T].
\end{equation*}

If, in addition, $2i+1\leq k$, we can pass from the even energy to the next odd one. Since
\begin{equation*}
    \mathcal E_{2i+1,\gamma}(t)
    =
    \mathcal E_{2i,\gamma}(t)
    +
    \frac1{2\lambda^{2i-2}}a_i''(t)\ul(t)^2,
\end{equation*}
the estimate
\begin{equation*}
    |a_i''(t)|
    \leq
    C\ep_k\lambda^{2i}
\end{equation*}
and the equivalence already proved for $\mathcal E_{2i,\gamma}$ imply
\begin{equation*}
    \ul(t)^2
    \leq
    \frac{C}{\lambda^2}
    \mathcal E_{2i,\gamma}(t).
\end{equation*}
Therefore
\begin{equation*}
    \left|
        \mathcal E_{2i+1,\gamma}(t)
        -
        \mathcal E_{2i,\gamma}(t)
    \right|
    \leq
    C\ep_k
    \mathcal E_{2i,\gamma}(t).
\end{equation*}
By possibly reducing $\ep_k$ once more, we obtain
\begin{equation*}
    \frac12\mathcal E_{2i,\gamma}(t)
    \leq
    \mathcal E_{2i+1,\gamma}(t)
    \leq
    \frac32\mathcal E_{2i,\gamma}(t)
    \qquad
    \forall t\in[0,T].
\end{equation*}

Starting from the equivalence already established for the first three energies, the previous two estimates show inductively that all the energies in the cascade up to order $k$ are uniformly equivalent to the standard energy, with positive equivalence constants $A_k$ and $B_k$ depending only on $k$, $\nu_1$, and $\nu_2$.

\paragraph{\textmd{\textit{Differential inequality for the energy}}}

From the equivalence between energies we know that
\begin{equation}
    |\ul(t)\ul'(t)|
    \leq
    \frac{C}{\lambda}\mathcal E_{k,\gamma}(t)
    \qquad\quad\text{and}\quad\qquad
    \ul(t)^2
    \leq
    \frac{C}{\lambda^2}\mathcal E_{k,\gamma}(t),
    \label{est:general-u}
\end{equation}
while from~(\ref{est:general-aj}) and the uniform bounds on $\gamma$ we deduce that
\begin{equation*}
    |S_i(t)|
    \leq
    C
    \qquad\quad\text{and}\quad\qquad
    |S_i'(t)|
    \leq
    C\lambda.
\end{equation*}
Finally, as in~(\ref{est:c^2-gamma2}),
\begin{equation}
    \left|c(t)^2-\frac1{\gamma(t)^2}\right|
    \leq
    2\nu_2^3
    \left|\gamma(t)-\frac1{c(t)}\right|.
    \label{est:general-defect}
\end{equation}

We distinguish the two possible parities of $k$. If $k=2i$, then~(\ref{eqn:general-even-derivative}), together with~(\ref{est:general-u})--(\ref{est:general-defect}), gives
\begin{equation}
    |\mathcal E_{k,\gamma}'(t)|
    \leq
    C
    \left\{
        \frac{|a_i''(t)|}{\lambda^{k-1}}
        +
        \lambda
        \left|\gamma(t)-\frac1{c(t)}\right|
    \right\}
    \mathcal E_{k,\gamma}(t).
    \label{est:general-even}
\end{equation}

If $k=2i+1$, then~(\ref{eqn:general-odd-derivative}) yields
\begin{equation}
    |\mathcal E_{k,\gamma}'(t)|
    \leq
    C
    \left\{
        \frac{|a_i'''(t)|}{\lambda^{k-1}}
        +
        \lambda
        \left|\gamma(t)-\frac1{c(t)}\right|
    \right\}
    \mathcal E_{k,\gamma}(t).
    \label{est:general-odd}
\end{equation}

\paragraph{\textmd{\textit{Reduction to the derivative of $\gamma$ of order $k$}}}

It remains to estimate the integral of the terms involving $a_i''$ or $a_i'''$. At this stage, the new ingredient is the interpolation estimate of Lemma~\ref{lem:weighted-polynomial}, applied to the homogeneous differential polynomials identified above. At this point the constant $C$ is also allowed to depend on $T$.

When $k=2i$ is even, Lemma~\ref{lemma:pol-representation} shows that $a_i$ is the evaluation along $\gamma$ of a homogeneous differential polynomial of differential weight $2i-2$. Therefore $a_i''$ is the evaluation along $\gamma$ of a homogeneous differential polynomial of differential weight $2i=k$, and Lemma~\ref{lem:weighted-polynomial} gives
\begin{equation*}
    \int_0^T|a_i''(t)|\,dt
    \leq
    C
    \left(
        1+
        \int_0^T|\gamma^{(k)}(t)|\,dt
    \right).
\end{equation*}

Similarly, when $k=2i+1$ is odd, $a_i'''$ is the evaluation along $\gamma$ of a homogeneous differential polynomial of differential weight $2i+1=k$. Hence Lemma~\ref{lem:weighted-polynomial} yields
\begin{equation*}
    \int_0^T|a_i'''(t)|\,dt
    \leq
    C
    \left(
        1+
        \int_0^T|\gamma^{(k)}(t)|\,dt
    \right).
\end{equation*}

Combining these estimates with~(\ref{est:general-even}) and~(\ref{est:general-odd}), and integrating the resulting differential inequality, we obtain, for a suitable constant $C_k>0$ depending only on $k$, $T$, $\nu_1$, and $\nu_2$,
\begin{equation*}
    \left|
        \log
        \frac{\mathcal E_{k,\gamma}(t)}
             {\mathcal E_{k,\gamma}(0)}
    \right|
    \leq
    C_k
    \left[
        \frac1{\lambda^{k-1}}
        +
        \F_{k,\lambda}(\gamma;c)
    \right]
    \qquad
    \forall t\in[0,T].
\end{equation*}
This is equivalent to~(\ref{th:approx-energy}) and concludes the proof of Theorem~\ref{thm:approximate-cascade}.


\setcounter{equation}{0}
\section{A family of variational problems}\label{sec:variational}

In this section we introduce a family of variational problems associated with the competition between regularity and fidelity. For every positive integer $k$, the corresponding minimum value measures how well a given function can be approximated by smooth functions, balancing the size of the $k$-th derivative of the approximation against its distance from the original function.

The first aim of the section is purely variational. We compare the minimum problems corresponding to different values of $k$, and show that a qualitative distinction occurs between the first and the higher orders. More precisely, all problems of order $k\geq2$ turn out to be equivalent, in the sense that they generate the same classes of admissible growth as the scale parameter tends to infinity, whereas the first-order problem retains a genuinely different behavior.

This naturally leads to a family of function spaces defined by prescribing the growth of the minimum values. We investigate the relations between these spaces and their connection with fractional Sobolev regularity.

The second aim is to prepare the applications to the evolution problem. The approximate-energy estimates of Section~\ref{sec:approx-en} require smooth competitors satisfying additional pointwise bounds on the function and its derivatives. We therefore introduce also constrained versions of the variational problems and compare them with the unconstrained ones to the extent needed for the subsequent applications.

Most of the analysis carried out in this section is independent of the evolution problem and can be read on its own. The connection between the variational problems and the growth of solutions will be developed in Section~\ref{sec:var2derloss}.

\subsection{Definitions and main results}

\subsubsection{Functionals and minimum values}

Let $(a,b)\subseteq\mathbb R$ be an interval, and let $g\in L^1((a,b))$. For every positive integer $k$ and every positive real number $\lambda$, we consider the functional
\begin{equation*}
    \F_k((a,b),g,\lambda,\gamma)
    :=
    \frac{1}{\lambda^{k-1}}
    \int_a^b\left|\gamma^{(k)}(x)\right|\,dx
    +
    \lambda\int_a^b|\gamma(x)-g(x)|\,dx,
\end{equation*}
defined for every $\gamma\in C^k([a,b])$, and the corresponding minimum value
\begin{equation}
    \M_k((a,b),g,\lambda)
    :=
    \inf\left\{
        \F_k((a,b),g,\lambda,\gamma):
        \gamma\in C^k([a,b])
    \right\}.
    \label{defn:Mk}
\end{equation}

The two terms in $\F_k$ play complementary roles. The first one penalizes the oscillations of the approximating function through its $k$-th derivative, while the second one is a fidelity term that penalizes its distance from $g$. The parameter $\lambda$ determines the scale at which these two effects are balanced.

We shall also use the $L^1$-oscillation of $g$ on $(a,b)$, defined by
\begin{equation}
    \operatorname{osc}_1((a,b),g)
    :=
    \min\left\{
        \int_a^b|g(x)-m|\,dx:m\in\mathbb R
    \right\}.
    \label{defn:osc-1}
\end{equation}
Thus $\operatorname{osc}_1((a,b),g)$ is the distance in $L^1((a,b))$ between $g$ and the subspace of constant functions.

\begin{rmk}\label{rmk:Mk-basic}
\begin{em}

Let us record some elementary properties of the minimum values.

\begin{enumerate}
\renewcommand{\labelenumi}{(\arabic{enumi})}

\item For every positive integer $k$, a constant function realizing the minimum in~(\ref{defn:osc-1}) can be used as a competitor in~(\ref{defn:Mk}). Hence
\begin{equation}
    \M_k((a,b),g,\lambda)
    \leq
    \lambda\,\operatorname{osc}_1((a,b),g)
    \qquad
    \forall\lambda>0.
    \label{est:Mk-osc}
\end{equation}
In particular, $\M_k((a,b),g,\lambda)\to0$ as $\lambda\to0^+$, and $\M_k((a,b),g,\lambda)$ grows at most linearly as $\lambda\to+\infty$.

\item If $g\in C^k([a,b])$, then $g$ itself is an admissible competitor, and therefore
\begin{equation*}
    \M_k((a,b),g,\lambda)
    \leq
    \frac{1}{\lambda^{k-1}}
    \int_a^b|g^{(k)}(x)|\,dx
    \qquad
    \forall\lambda>0.
\end{equation*}
In particular, for every $k\geq2$ and every $g\in C^k([a,b])$, one has
$\M_k((a,b),g,\lambda)\to0$ as $\lambda\to+\infty$.

\item In the case $k=1$, the function
$\lambda\mapsto\M_1((a,b),g,\lambda)$ is nondecreasing. Indeed, for every fixed competitor $\gamma$, the function
$\lambda\mapsto\F_1((a,b),g,\lambda,\gamma)$ is nondecreasing.

\item For $k\geq2$, the function
$\lambda\mapsto\M_k((a,b),g,\lambda)$ is in general not monotone. For instance, when $g\in C^k([a,b])$, monotonicity would be incompatible with the behavior described in the first two points.

\item For fixed $g$, the function $\lambda\mapsto\M_k((a,b),g,\lambda)$ is continuous, while for fixed $\lambda$, the map $g\mapsto\M_k((a,b),g,\lambda)$ is $\lambda$-Lipschitz continuous with respect to the $L^1((a,b))$ norm.

\item The infimum in~(\ref{defn:Mk}) is unchanged if the class of
competitors is enlarged to functions
\begin{equation*}
    \gamma\in W^{k-1,1}((a,b))
    \qquad\text{such that}\qquad
    \gamma^{(k-1)}\in BV((a,b)).
\end{equation*}
Moreover, the infimum is attained in this enlarged class.

\end{enumerate}

\end{em}
\end{rmk}

\subsubsection{Unconstrained variational problems and function spaces}

Since $\M_k$ measures the cost of approximating $g$ at scale $\lambda$, its growth as $\lambda\to+\infty$ provides a natural way of quantifying the regularity of $g$. We begin with the bounded-growth case and define
\begin{equation}
    \Sp_k((a,b))
    :=
    \left\{
        g\in L^1((a,b)):
        \sup_{\lambda>0}\M_k((a,b),g,\lambda)<+\infty
    \right\}.
    \label{defn:Sk-basic}
\end{equation}

The first result already reveals the basic dichotomy of the theory: the first-order problem is exceptional, whereas all higher-order problems generate the same space.

\begin{thm}[Inclusions between spaces -- Basic case]\label{thm:Sk-basic}

Let $(a,b)\subseteq\mathbb R$ be an interval, and let the spaces $\Sp_k((a,b))$ be defined by~(\ref{defn:Sk-basic}). Then
\begin{equation}
    BV((a,b))
    =
    \Sp_1((a,b))
    \subsetneq
    \Sp_2((a,b))
    =
    \Sp_k((a,b))
    \subsetneq
    \bigcap_{s\in(0,1)}W^{s,1}((a,b))
    \qquad
    \forall k\geq2.
    \label{th:Sk-basic}
\end{equation}

\end{thm}

\begin{ex}[A function in $\Sp_2\setminus\Sp_1$]
\label{ex:S2-not-S1}
\begin{em}

A very simple oscillatory function is enough to show that the inclusion\[
\Sp_1((a,b))
\subsetneq
\Sp_2((a,b))
\]
is strict. Indeed,
\begin{equation}
    g(x):=\sin(\log x)
    \qquad
    \forall x\in(0,1)
    \label{defn:sin-log}
\end{equation}
satisfies
\begin{equation*}
    g\in
    \Sp_2((0,1))
    \setminus
    \Sp_1((0,1)).
\end{equation*}

The reason is that $g$ has infinite total variation near the origin, whereas at the second variational order the oscillations can be removed by freezing the function at suitable maximum points comparable with the wavelength scale $1/\lambda$. A complete verification is given in Subsection~\ref{sec:proof-example}.

\end{em}
\end{ex}

More generally, we can prescribe the admissible growth of the minimum values. Let
$\varphi:(0,+\infty)\to[1,+\infty)$ be a nondecreasing function, and set
\begin{equation}
    \Sp_k^\varphi((a,b))
    :=
    \left\{
        g\in L^1((a,b)):
        \sup_{\lambda>0}
        \frac{\M_k((a,b),g,\lambda)}
             {\varphi(\lambda)}
        <+\infty
    \right\}.
    \label{defn:Sk-general}
\end{equation}

The assumption $\varphi\geq1$, together with~(\ref{est:Mk-osc}), ensures that this definition depends only on the behavior of the minimum values for large $\lambda$.

For later use, we introduce also generalized fractional Sobolev spaces. Given a function $\psi:(0,+\infty)\to(0,+\infty)$, we denote by $W^{\psi,1}((a,b))$ the space of all functions $g\in L^1((a,b))$ such that
\begin{equation*}
    \int_a^b dy
    \int_a^b
    \frac{|g(y)-g(x)|}
         {|y-x|\psi(|y-x|)}
    \,dx
    <+\infty.
\end{equation*}
When $\psi(\sigma)=\sigma^s$ with $s\in(0,1)$, this is the usual fractional Sobolev space $W^{s,1}((a,b))$.

All orders $k\geq2$ generate the same spaces $\Sp_k^\varphi((a,b))$. On the other hand, the comparison between the second- and first-order problems involves a dyadic accumulation of the prescribed growth. Depending on the growth rate, this accumulation may or may not change the resulting space.

\begin{thm}[Inclusions between spaces -- General case]\label{thm:Sk-general}

Let $(a,b)\subseteq\mathbb R$ be an interval, let
$\varphi:(0,+\infty)\to[1,+\infty)$ be nondecreasing, and let
$\Sp_k^\varphi((a,b))$ be defined by~(\ref{defn:Sk-general}).

Then the following statements hold.

\begin{enumerate}
\renewcommand{\labelenumi}{(\arabic{enumi})}

\item \emph{(Equalities for $k\geq2$).}
For every integer $k\geq2$,
\begin{equation*}
    \Sp_k^\varphi((a,b))
    =
    \Sp_2^\varphi((a,b)).
\end{equation*}

\item \emph{(Comparison with the first-order space).}
It turns out that
\begin{equation}
    \Sp_1^\varphi((a,b))
    \subseteq
    \Sp_2^\varphi((a,b))
    \subseteq
    \Sp_1^{\widehat{\varphi}}((a,b)),
    \label{th:S1-S2}
\end{equation}
where
\begin{equation}
    \widehat{\varphi}(\lambda)
    :=
    \sum_{i=0}^{\lfloor\log_2(1+\lambda)\rfloor}
    \varphi\left(\frac{\lambda}{2^i}\right)
    \qquad
    \forall\lambda>0.
    \label{defn:new-phi}
\end{equation}

\item \emph{(Generalized fractional Sobolev regularity).}
For every $g\in L^1((a,b))$ and every function $\psi:(0,+\infty)\to(0,+\infty)$, one has
\begin{equation}
    \int_a^b dy
    \int_a^b
    \frac{|g(y)-g(x)|}
         {|y-x|\psi(|y-x|)}
    \,dx
    \leq
    4\int_0^{b-a}
    \frac{\M_1((a,b),g,1/w)}
         {\psi(w)}
    \,dw.
    \label{th:fractional}
\end{equation}

Consequently, if $g\in\Sp_1^\varphi((a,b))$, then
\begin{equation}
    g\in W^{\psi,1}((a,b))
    \label{th:Sk-Ws}
\end{equation}
whenever
\begin{equation*}
    \int_0^{b-a}
    \frac{\varphi(1/w)}
         {\psi(w)}
    \,dw
    <+\infty.
\end{equation*}

\end{enumerate}

\end{thm}

Two special cases deserve to be mentioned. The first one is the case of constant weight functions. In this situation, the distinction between the first and the higher orders becomes particularly transparent: Theorem~\ref{thm:Sk-general} gives a more precise description of how the second-order space can be controlled in terms of a first-order problem.

\begin{cor}[Constant weight function]\label{cor:constant}

Let $(a,b)\subseteq\mathbb R$ be an interval, and let us consider the spaces defined in~(\ref{defn:Sk-basic}) and~(\ref{defn:Sk-general}).

Then
\begin{equation}
    \Sp_2((a,b))
    \subseteq
    \Sp_1^{\widehat{\varphi}}((a,b))
    \subseteq
    W^{\psi,1}((a,b)),
    \label{th:cor-constant}
\end{equation}
where
\begin{equation}
    \widehat{\varphi}(\lambda)
    :=
    1+\lfloor\log_2(1+\lambda)\rfloor,
    \qquad\quad\text{and}\quad\qquad
    \psi(\sigma)
    :=
    \sigma\log^3(1+1/\sigma).
    \label{defn:psi-psi-constant}
\end{equation}

\end{cor}

The situation changes for weight functions with power-like growth. In this case the dyadic accumulation appearing in the comparison with the first-order problem has the same growth as the original weight. As a consequence, the distinction between the first and the higher orders disappears, and the collapse of the spaces already occurs at order one.

\begin{cor}[Power-like weight function]\label{cor:power}

Let $(a,b)\subseteq\mathbb R$ be an interval, let $s\in(0,1)$, and consider the weight function
\begin{equation}
    \varphi(\lambda):=(1+\lambda)^s.
    \label{defn:phi-cor-power}
\end{equation}

Then
\begin{equation}
    \Sp_k^\varphi((a,b))
    =
    \Sp_1^\varphi((a,b))
    \subseteq
    W^{\psi,1}((a,b))
    \qquad
    \forall k\geq1,
    \label{th:cor-power}
\end{equation}
where
\begin{equation}
    \psi(\sigma)
    :=
    \sigma^{1-s}\log^2(1+1/\sigma).
    \label{defn:psi-cor-power}
\end{equation}
In particular,
\begin{equation}
    \Sp_k^\varphi((a,b))
    \subsetneq
    \bigcap_{r<1-s}
    W^{r,1}((a,b))
    \qquad
    \forall k\geq1.
    \label{th:cor-power-strict}
\end{equation}

\end{cor}

\subsubsection{Constrained variational problems}

The unconstrained minimum problems introduced above are the natural objects for the comparison between different differentiation orders. For the applications to the approximate energies of Section~\ref{sec:approx-en}, however, the function to be approximated is bounded from above and below, and the approximating functions have to satisfy corresponding pointwise constraints.

Let $\mu_1$ and $\mu_2$ be positive real numbers with $\mu_1\leq\mu_2$, and assume that
\begin{equation}
    \mu_1\leq g(x)\leq\mu_2
    \qquad
    \text{for almost every }x\in(a,b).
    \label{hp:Mk-constrained-range}
\end{equation}

It is then natural to restrict the competitors to the same range. For every positive integer $k$, let $\ep_k>0$, and for every $\lambda>0$ consider the class $\D_{k,\lambda}$ of all $\gamma\in C^k([a,b])$ that satisfy
\begin{equation*}
    \mu_1\leq\gamma(x)\leq\mu_2
    \qquad
    \forall x\in[a,b],
\end{equation*}
and, if $k\geq2$, also the derivative constraints
\begin{equation*}
    \left|\gamma^{(j)}(x)\right|
    \leq
    \ep_k\lambda^j
    \qquad
    \forall x\in[a,b],
    \quad
    \forall j\in\{1,\ldots,k-1\}.
\end{equation*}

The corresponding constrained minimum value is
\begin{equation*}
    \M_k^*((a,b),g,\lambda)
    :=
    \inf\left\{
        \F_k((a,b),g,\lambda,\gamma):
        \gamma\in\D_{k,\lambda}
    \right\}.
\end{equation*}

Clearly,
\begin{equation*}
    \M_k((a,b),g,\lambda)
    \leq
    \M_k^*((a,b),g,\lambda).
\end{equation*}

At first order the range constraint does not change the minimum value. At second order, the additional bound on the first derivative may change the minimum value, but only up to a multiplicative constant and a harmless additive constant.

\begin{prop}[First-order constrained problem]
\label{prop:M1-constrained}

Let $(a,b)\subseteq\mathbb R$ be an interval, let $\mu_1$, $\mu_2$ be positive real numbers with $\mu_1\leq\mu_2$, and let $g\in L^1((a,b))$ satisfy (\ref{hp:Mk-constrained-range}).

Then for every $\lambda>0$ one has
\begin{equation}
    \M_1^*((a,b),g,\lambda)
    =
    \M_1((a,b),g,\lambda).
    \label{th:M1-constrained}
\end{equation}

\end{prop}

\begin{prop}[Second-order constrained problem]
\label{prop:M2-constrained}

Let $(a,b)\subseteq\mathbb R$ be an interval, let $\ep_2$, $\mu_1$, $\mu_2$ be positive real numbers with $\mu_1\leq\mu_2$, and let $g\in L^1((a,b))$ satisfy (\ref{hp:Mk-constrained-range}).

Then there exists a constant $C>0$, depending only on $\mu_1$, $\mu_2$, $\ep_2$, and $b-a$, such that
\begin{equation*}
    \M_2^*((a,b),g,\lambda)
    \leq
    C\left(
        \M_2((a,b),g,\lambda)+1
    \right)
    \qquad
    \forall\lambda>0.
\end{equation*}

\end{prop}

Thus the constraints required by the approximate-energy method are harmless at the level needed in the sequel: at first order they do not alter the minimum at all, while at second order they do not alter its growth class.


\subsection{Technical tools for the comparison of different orders}

The comparison between variational problems of different orders relies on two approximation mechanisms. The first is based on higher-order difference quotients, while the second provides a regularization procedure that gains one derivative with quantitative control of both the fidelity error and the higher-order derivative.

We collect the corresponding tools in this subsection. Most of them are elementary consequences of the calculus of higher-order finite differences; see, for example,~\cite{DitzianTotik87,DeVoreLorentz93}. We nevertheless include the details for the convenience of the reader and in order to keep the argument self-contained. The subsection may be skipped on a first reading.

\subsubsection{Higher-order difference quotients}

Let $(a,b)\subseteq\mathbb R$ be an interval, and let $f:(a,b)\to\mathbb R$ be a function. For every positive integer $k$, we consider the set
\begin{equation*}
    D_k
    :=
    \left\{
        (x,v)\in\mathbb R^2:
        a<x<x+kv<b
    \right\}.
\end{equation*}

The higher-order difference quotients of $f$ are defined recursively by
\begin{equation*}
    \Delta_f^1(x,v)
    :=
    \frac{f(x+v)-f(x)}{v}
    \qquad
    \forall (x,v)\in D_1,
\end{equation*}
and
\begin{equation*}
    \Delta_f^{i+1}(x,v)
    :=
    \frac{\Delta_f^i(x+v,v)-\Delta_f^i(x,v)}{v}
    \qquad
    \forall (x,v)\in D_{i+1}.
\end{equation*}

It is convenient to extend the notation to order zero by setting
$\Delta_f^0(x,v):=f(x)$, so that the same recursive formula remains valid for $i=0$.

We shall use the following standard properties.

\begin{lemma}[Basic properties of higher-order difference quotients]
\label{lemma:diffq-basic}

Let $(a,b)\subseteq\mathbb R$ be an interval, let $f:(a,b)\to\mathbb R$, and let $k$ be a positive integer.

Then the following statements hold.

\begin{enumerate}
\renewcommand{\labelenumi}{(\arabic{enumi})}

\item \emph{(Binomial representation).}
For every $(x,v)\in D_k$,
\begin{equation}
    v^k\Delta_f^k(x,v)
    =
    (-1)^k
    \sum_{i=0}^{k}
    \binom{k}{i}
    (-1)^i
    f(x+iv).
    \label{eqn:Dk-newton}
\end{equation}

\item \emph{(Integral estimate).}
If $f\in L^1((a,b))$, then
\begin{equation}
    \int_a^{b-kv}
    \left|\Delta_f^k(x,v)\right|\,dx
    \leq
    \frac{2^k}{v^k}
    \int_a^b|f(x)|\,dx
    \qquad
    \forall v\in\left(0,\frac{b-a}{k}\right).
    \label{prop:delta-int}
\end{equation}

\item \emph{(Generalized mean-value theorem).}
If $f\in C^k((a,b))$, then for every $(x,v)\in D_k$ there exists
$\xi\in(x,x+kv)$ such that
\begin{equation}
    \Delta_f^k(x,v)
    =
    f^{(k)}(\xi).
    \label{prop:delta-lagrange}
\end{equation}

\end{enumerate}

\end{lemma}

The first two properties are immediate from the recursive definition and the binomial representation, while the last one is the classical generalized mean-value theorem for finite differences.

We shall also need an integral identity that relates difference quotients of consecutive orders.

\begin{lemma}[Higher-order fundamental theorem of calculus]
\label{lemma:f-k-f'}

Let $[a,b]\subseteq\mathbb R$, let $v>0$, let $k$ be a nonnegative integer, and let
$f\in C^{k+1}([a,b+(k+1)v])$.

Then
\begin{equation}
    v^{k+1}\Delta_f^{k+1}(x,v)
    =
    (k+1)
    \int_0^v
    s^k
    \Delta_{f'}^k(x+s,s)\,ds
    \qquad
    \forall x\in[a,b].
    \label{th:f-k-f'}
\end{equation}

\end{lemma}

\begin{proof}

Using~(\ref{eqn:Dk-newton}), isolating the term corresponding to $i=0$, and exploiting
\begin{equation*}
    \sum_{i=1}^{k+1}
    \binom{k+1}{i}
    (-1)^{i+1}
    =
    1,
\end{equation*}
we obtain
\begin{equation*}
    v^{k+1}\Delta_f^{k+1}(x,v)
    =
    (-1)^k
    \sum_{i=1}^{k+1}
    \binom{k+1}{i}
    (-1)^{i+1}
    [f(x+iv)-f(x)].
\end{equation*}

Since
\begin{equation*}
    f(x+iv)-f(x)
    =
    i\int_0^v f'(x+is)\,ds
\end{equation*}
and
\begin{equation*}
    i\binom{k+1}{i}
    =
    (k+1)\binom{k}{i-1},
\end{equation*}
another application of~(\ref{eqn:Dk-newton}) yields~(\ref{th:f-k-f'}).
\end{proof}

The next estimate will be used repeatedly.

\begin{lemma}[Integral estimate for higher-order difference quotients]
\label{lemma:Dk-int-fk}

Let $[a,b]\subseteq\mathbb R$, let $v>0$, let $k$ be a nonnegative integer, and let
$f\in C^k([a,b+kv])$.

Then
\begin{equation*}
    \int_a^b
    \left|\Delta_f^k(x,v)\right|\,dx
    \leq
    \int_a^{b+kv}
    \left|f^{(k)}(x)\right|\,dx.
\end{equation*}

\end{lemma}

\begin{proof}

We argue by induction on $k$. The case $k=0$ is immediate.

Assume that the estimate holds at order $k$, and let
$g\in C^{k+1}([a,b+(k+1)v])$. From Lemma~\ref{lemma:f-k-f'} we obtain
\begin{equation*}
    v^{k+1}
    \left|\Delta_g^{k+1}(x,v)\right|
    \leq
    (k+1)
    \int_0^v
    s^k
    \left|\Delta_{g'}^k(x+s,s)\right|\,ds.
\end{equation*}

Integrating with respect to $x$ and applying the inductive assumption to $g'$ on the shifted interval $[a+s,b+s]$, for every $s\in(0,v)$, we obtain
\begin{equation*}
    v^{k+1}
    \int_a^b
    \left|\Delta_g^{k+1}(x,v)\right|\,dx
    \leq
    (k+1)
    \int_0^v
    s^k\,ds
    \int_a^{b+(k+1)v}
    \left|g^{(k+1)}(x)\right|\,dx.
\end{equation*}

Since
\begin{equation*}
    (k+1)\int_0^v s^k\,ds=v^{k+1},
\end{equation*}
the conclusion follows.
\end{proof}

\subsubsection{Higher-order regularization}

We now construct an approximation procedure that gains one derivative. Given $f\in C^k([a,b])$, we seek functions $f_\ep\in C^{k+1}([a,b])$ such that the $L^1$ distance from $f$ is of order $\ep^k$, while the $L^1$ norm of the derivative of order $k+1$ is controlled by $\ep^{-1}$ times the $L^1$ norm of the $k$-th derivative. A standard mollification does not provide the required fidelity estimate in terms of the $k$-th derivative alone. The higher-order averaging operator introduced below is designed precisely to produce the cancellations needed to obtain an error of order $\ep^k$.

For functions defined on a sufficiently large interval, we introduce the higher-order averaging operator
\begin{equation}
    [C_{k,\ep}f](x)
    :=
    \frac1{\ep}
    \int_0^\ep
    \left[
        (-1)^{k+1}v^k\Delta_f^k(x,v)
        +
        f(x)
    \right]\,dv.
    \label{defn:Ckep}
\end{equation}

The relevant estimates are collected in the next lemma.

\begin{lemma}[Estimates for higher-order regularization]
\label{lemma:Ckep}

Let $[a,b]\subseteq\mathbb R$, let $\ep>0$, let $k$ be a positive integer, and let
$f\in C^k([a,b+k\ep])$.

Then $C_{k,\ep}f\in C^{k+1}([a,b])$ and
\begin{equation}
    \int_a^b
    \left|
        [C_{k,\ep}f](x)-f(x)
    \right|\,dx
    \leq
    \ep^k
    \int_a^{b+k\ep}
    |f^{(k)}(x)|\,dx,
    \label{th:Cke-fid}
\end{equation}
while
\begin{equation}
    \int_a^b
    \left|
        [C_{k,\ep}f]^{(k+1)}(x)
    \right|\,dx
    \leq
    \frac{2^{k+1}}{\ep}
    \int_a^{b+k\ep}
    |f^{(k)}(x)|\,dx.
    \label{th:Ckep-reg}
\end{equation}

\end{lemma}

\begin{proof}

The proof is divided into the fidelity and regularity estimates.

\subparagraph{\textmd{\textit{Fidelity estimate}}}

From~(\ref{defn:Ckep}) we obtain
\begin{equation*}
    \left|
        [C_{k,\ep}f](x)-f(x)
    \right|
    \leq
    \frac1{\ep}
    \int_0^\ep
    v^k
    \left|\Delta_f^k(x,v)\right|
    \,dv
    \qquad
    \forall x\in[a,b].
\end{equation*}
Integrating with respect to $x$ and reversing the order of integration, Lemma~\ref{lemma:Dk-int-fk} yields
\begin{align*}
    \int_a^b
    \left|
        [C_{k,\ep}f](x)-f(x)
    \right|\,dx
    &\leq
    \frac1{\ep}
    \int_0^\ep
    v^k
    \left(
        \int_a^b
        \left|\Delta_f^k(x,v)\right|\,dx
    \right)
    dv
    \\
    &\leq
    \frac1{\ep}
    \int_0^\ep
    v^k\,dv
    \int_a^{b+k\ep}
    |f^{(k)}(x)|\,dx
    \\
    &\leq
    \ep^k
    \int_a^{b+k\ep}
    |f^{(k)}(x)|\,dx.
\end{align*}
This proves~(\ref{th:Cke-fid}).

\subparagraph{\textmd{\textit{Regularity estimate}}}

Using the binomial representation~(\ref{eqn:Dk-newton}), the term $f(x)$ in (\ref{defn:Ckep}) cancels the term corresponding to $i=0$, and therefore
\begin{equation*}
    (-1)^{k+1}v^k\Delta_f^k(x,v)+f(x)
    =
    \sum_{i=1}^k
    \binom{k}{i}
    (-1)^{i+1}
    f(x+iv).
\end{equation*}
Since
\begin{equation*}
    \int_0^\ep f(x+iv)\,dv
    =
    \frac1i
    \int_x^{x+i\ep}f(y)\,dy,
\end{equation*}
we obtain
\begin{equation*}
    [C_{k,\ep}f](x)
    =
    \sum_{i=1}^k
    \binom{k}{i}
    \frac{(-1)^{i+1}}{\ep i}
    \int_x^{x+i\ep}f(y)\,dy.
\end{equation*}
In particular, $C_{k,\ep}f$ is of class $C^{k+1}$ on $[a,b]$, and
\begin{equation*}
    [C_{k,\ep}f]^{(k+1)}(x)
    =
    \sum_{i=1}^k
    \binom{k}{i}
    \frac{(-1)^{i+1}}{\ep i}
    \left(
        f^{(k)}(x+i\ep)-f^{(k)}(x)
    \right).
\end{equation*}
Consequently,
\begin{align*}
    \int_a^b
    \left|
        [C_{k,\ep}f]^{(k+1)}(x)
    \right|\,dx
    &\leq
    \frac1{\ep}
    \sum_{i=1}^k
    \binom{k}{i}\frac1i
    \int_a^b
    \left(
        |f^{(k)}(x+i\ep)|
        +
        |f^{(k)}(x)|
    \right)
    dx
    \\
    &\leq
    \frac2{\ep}
    \left(
        \sum_{i=1}^k\binom{k}{i}
    \right)
    \int_a^{b+k\ep}
    |f^{(k)}(x)|\,dx
    \\
    &\leq
    \frac{2^{k+1}}{\ep}
    \int_a^{b+k\ep}
    |f^{(k)}(x)|\,dx.
\end{align*}
This proves~(\ref{th:Ckep-reg}) and completes the proof.
\end{proof}

The previous lemma applies to functions defined on an interval larger than the one where the approximation is needed. We therefore complement it with a simple extension result.

\begin{lemma}[Half-line extension of order $k$]\label{lemma:extender}

Let $[a,b]\subseteq\mathbb R$ be an interval, let $k$ be a positive integer, and let $f\in C^k([a,b])$.

Then there exists a function $f_*\in C^k([a,+\infty))$ such that
\begin{equation*}
    f_*(x)=f(x)
    \qquad
    \forall x\in[a,b],
\end{equation*}
and
\begin{equation}
    \int_a^{+\infty}
    \left|f_*^{(k)}(x)\right|\,dx
    \leq
    2\int_a^b
    \left|f^{(k)}(x)\right|\,dx.
    \label{est:Ek-reg}
\end{equation}

\end{lemma}

\begin{proof}

The argument is standard. We first extend $f^{(k)}$ from $[a,b]$ to $[a,2b-a]$ by reflection, setting
\begin{equation*}
    g(x)
    :=
    \begin{cases}
        f^{(k)}(x) & \text{if }x\in[a,b],\\
        f^{(k)}(2b-x) & \text{if }x\in[b,2b-a].
    \end{cases}
\end{equation*}

Let $\theta:\mathbb R\to[0,1]$ be a smooth cut-off function such that
$\theta(x)=1$ for $x\leq b$ and $\theta(x)=0$ for $x\geq2b-a$. We define
\begin{equation*}
    \widehat g(x)
    :=
    \begin{cases}
        g(x)\theta(x) & \text{if }x\in[a,2b-a],\\
        0 & \text{if }x\geq2b-a.
    \end{cases}
\end{equation*}
Then $\widehat g$ is continuous on $[a,+\infty)$ and
\begin{equation}
    \int_a^{+\infty}
    |\widehat g(x)|\,dx
    \leq
    2\int_a^b
    |f^{(k)}(x)|\,dx.
    \label{est:gk}
\end{equation}

We finally define $f_*$ as the unique function satisfying
\begin{equation*}
    f_*^{(k)}=\widehat g
\end{equation*}
on $[a,+\infty)$ together with
\begin{equation*}
    f_*^{(j)}(b)=f^{(j)}(b)
    \qquad
    \forall j=0,\ldots,k-1.
\end{equation*}
Then $f_*=f$ on $[a,b]$, and~(\ref{est:Ek-reg}) follows from~(\ref{est:gk}).
\end{proof}

We can now combine the extension and regularization procedures into the approximation result that will be used in the comparison of consecutive orders.

\begin{prop}[Approximation of order $k$]\label{prop:Ckep}

Let $[a,b]\subseteq\mathbb R$ be an interval, let $k$ be a positive integer, and let $f\in C^k([a,b])$.

Then for every $\ep>0$ there exists a function
$f_\ep\in C^{k+1}([a,b])$ such that
\begin{equation}
    \int_a^b
    |f_\ep(x)-f(x)|\,dx
    \leq
    2\ep^k
    \int_a^b
    |f^{(k)}(x)|\,dx,
    \label{th:fep-fid}
\end{equation}
and
\begin{equation}
    \int_a^b
    |f_\ep^{(k+1)}(x)|\,dx
    \leq
    \frac{2^{k+2}}{\ep}
    \int_a^b
    |f^{(k)}(x)|\,dx.
    \label{th:fep-reg}
\end{equation}

\end{prop}

\begin{proof}

Let $f_*\in C^k([a,+\infty))$ be the extension provided by Lemma~\ref{lemma:extender}. For every $\ep>0$, we set
\begin{equation*}
    f_\ep
    :=
    C_{k,\ep}f_*.
\end{equation*}

Since $f_*$ is defined on the whole half-line, Lemma~\ref{lemma:Ckep} applies on $[a,b]$. The fidelity estimate~(\ref{th:fep-fid}) follows from~(\ref{th:Cke-fid}) and~(\ref{est:Ek-reg}), while the regularity estimate~(\ref{th:fep-reg}) follows from~(\ref{th:Ckep-reg}) and~(\ref{est:Ek-reg}).
\end{proof}

The technical preparation is now complete. We turn to the comparison between the variational problems corresponding to different differentiation orders.


\subsection{Comparison of different orders}

We now compare the variational problems corresponding to consecutive differentiation orders. The two directions are of a different nature. Passing from order $k$ to order $k+1$ is based on the approximation procedure developed above, while the converse direction relies on an interpolation estimate and leads naturally to a dyadic decomposition.

\subsubsection{Increasing the differentiation order}

We begin with the easier direction. Good competitors for the minimum problem of order $k+1$ can be obtained by regularizing competitors for the problem of order $k$. Proposition~\ref{prop:Ckep} provides precisely the required approximation.

\begin{prop}[From order $k$ to order $k+1$]\label{prop:k+1<k}

Let $(a,b)\subseteq\re$ be an interval, let $k$ be a positive integer, and let $g\in L^1((a,b))$ be a function.

Then it turns out that
\begin{equation}
    \M_{k+1}((a,b),g,\lambda)
    \leq
    2(2^{k+1}+1)\M_k((a,b),g,\lambda)
    \qquad
    \forall\lambda>0,
    \label{th:k+1<k}
\end{equation}
and in particular
$\Sp_k^\varphi((a,b))\subseteq\Sp_{k+1}^\varphi((a,b))$ for every nondecreasing weight function $\varphi:(0,+\infty)\to[1,+\infty)$.
\end{prop}

\begin{proof}

For every function $\gamma\in C^k([a,b])$, and for every $\ep>0$, we consider the function $\gamma_\ep\in C^{k+1}([a,b])$ provided by Proposition~\ref{prop:Ckep}. Since
\begin{multline*}
    \F_{k+1}((a,b),g,\lambda,\gamma_\ep)
    =
    \frac{1}{\lambda^k}
    \int_a^b
    \bigl|\gamma_\ep^{(k+1)}(x)\bigr|\,dx
    +
    \lambda
    \int_a^b
    |\gamma_\ep(x)-g(x)|\,dx
    \\
    \leq
    \frac{1}{\lambda^k}
    \int_a^b
    \bigl|\gamma_\ep^{(k+1)}(x)\bigr|\,dx
    +
    \lambda
    \int_a^b
    |\gamma_\ep(x)-\gamma(x)|\,dx
    +
    \lambda
    \int_a^b
    |\gamma(x)-g(x)|\,dx,
\end{multline*}
from (\ref{th:fep-reg}) and (\ref{th:fep-fid}) we obtain that
\begin{eqnarray*}
    \F_{k+1}((a,b),g,\lambda,\gamma_\ep)
    & \leq &
    \frac{1}{\lambda^{k}}
    \cdot
    \frac{2^{k+2}}{\ep}
    \int_a^b
    \bigl|\gamma^{(k)}(x)\bigr|\,dx
    \\
    & &
    \mbox{}+
    \lambda\cdot 2\ep^k
    \int_a^b
    \bigl|\gamma^{(k)}(x)\bigr|\,dx
    +
    \lambda
    \int_a^b
    |\gamma(x)-g(x)|\,dx.
\end{eqnarray*}

Setting $\ep:=1/\lambda$, we conclude that
\begin{eqnarray*}
    \F_{k+1}((a,b),g,\lambda,\gamma_{1/\lambda})
    & \leq &
    2(2^{k+1}+1)
    \cdot
    \frac{1}{\lambda^{k-1}}
    \int_a^b
    \bigl|\gamma^{(k)}(x)\bigr|\,dx
    \\
    & &
    \mbox{}+
    \lambda
    \int_a^b
    |\gamma(x)-g(x)|\,dx
    \\
    & \leq &
    2(2^{k+1}+1)
    \cdot
    \F_k((a,b),g,\lambda,\gamma).
\end{eqnarray*}

If we consider the infimum over all admissible functions $\gamma$, we obtain (\ref{th:k+1<k}).
\end{proof}


\subsubsection{Decreasing the differentiation order}

We now turn to the converse comparison. A competitor for the variational problem of order $k+1$ is of course also admissible for the problem of order $k$, but the functional of order $k+1$ does not directly control its $k$-th derivative.

The first step is therefore an interpolation estimate which bounds the $L^1$ norm of the $k$-th derivative in terms of the $L^1$ norm of the derivative of order $k+1$ and the oscillation of the function. When applied to competitors at different scales, this estimate leads naturally to a dyadic decomposition.

The resulting multiscale estimate is precisely where the distinction between the first and the higher orders appears. For $k\geq2$, the dyadic contributions carry a summable geometric weight, whereas for $k=1$ no such decay is available and the contributions accumulate over the dyadic scales.

The interpolation estimate needed for the first step is in the spirit of classical Glaeser and Gagliardo--Nirenberg inequalities.

\begin{lemma}[Interpolation inequality]\label{lemma:Dk}

Let $(a,b)\subseteq\re$ be an interval, and let $k$ be a positive integer.

Then for every $f\in C^{k+1}([a,b])$, and every real number $v\in(0,(b-a)/(2k)]$, one has
\begin{equation*}
    \int_a^b
    \left|f^{(k)}(x)\right|\,dx
    \leq
    \frac{2^{k+1}}{v^k}
    \osc_1((a,b),f)
    +
    2kv
    \int_a^b
    \left|f^{(k+1)}(x)\right|\,dx,
\end{equation*}
where $\osc_1((a,b),f)$ is defined according to (\ref{defn:osc-1}).

\end{lemma}

\begin{proof}

Let us consider the $k$-th order difference quotient $\Delta_f^k(x,v)$ introduced above and defined whenever $a\leq x<x+kv\leq b$.

For every admissible value of $x$ and $v$, let us consider the number $\xi(x,v)$ for which (\ref{prop:delta-lagrange}) holds. Since $|\xi(x,v)-x|\leq kv$, from the standard equality
\begin{equation*}
    f^{(k)}(x)
    =
    f^{(k)}(\xi(x,v))
    +
    \int_{\xi(x,v)}^{x}
    f^{(k+1)}(s)\,ds
    =
    \Delta_f^k(x,v)
    +
    \int_{\xi(x,v)}^{x}
    f^{(k+1)}(s)\,ds
\end{equation*}
we obtain that
\begin{equation*}
    \left|f^{(k)}(x)\right|
    \leq
    \left|\Delta_f^k(x,v)\right|
    +
    \int_x^{x+kv}
    \left|f^{(k+1)}(s)\right|\,ds
    \qquad
    \forall x\in[a,b-kv],
\end{equation*}
and hence
\begin{equation}
    \int_a^{b-kv}
    \left|f^{(k)}(x)\right|\,dx
    \leq
    \int_a^{b-kv}
    \left|\Delta_f^k(x,v)\right|\,dx
    +
    \int_a^{b-kv}
    dx
    \int_x^{x+kv}
    \left|f^{(k+1)}(s)\right|\,ds.
    \label{est:delta-k}
\end{equation}

Let us estimate the two integrals in the right-hand side. As for the integral of the difference quotient, we can apply directly (\ref{prop:delta-int}). As for the double integral, after reversing the order of integration, we observe that for every fixed $s\in[a,b]$ the corresponding section of the integration region in the $x$-variable has length at most $kv$, and therefore
\begin{equation}
    \int_a^{b-kv}
    dx
    \int_x^{x+kv}
    \left|f^{(k+1)}(s)\right|\,ds
    \leq
    kv
    \int_a^b
    \left|f^{(k+1)}(s)\right|\,ds.
    \label{est:delta-k-2}
\end{equation}

Plugging (\ref{prop:delta-int}) and (\ref{est:delta-k-2}) into (\ref{est:delta-k}) we conclude that
\begin{equation}
    \int_a^{b-kv}
    \left|f^{(k)}(x)\right|\,dx
    \leq
    \frac{2^k}{v^k}
    \int_a^b
    |f(x)|\,dx
    +
    kv
    \int_a^b
    \left|f^{(k+1)}(x)\right|\,dx.
    \label{est:delta-k-dir}
\end{equation}

By applying the same estimate (\ref{est:delta-k-dir}) to the function $x\mapsto f(b+a-x)$ we obtain that
\begin{equation}
    \int_{a+kv}^{b}
    \left|f^{(k)}(x)\right|\,dx
    \leq
    \frac{2^k}{v^k}
    \int_a^b
    |f(x)|\,dx
    +
    kv
    \int_a^b
    \left|f^{(k+1)}(x)\right|\,dx.
    \label{est:delta-k-rev}
\end{equation}

When $kv\leq(b-a)/2$, as we assumed in the statement, the two intervals $[a,b-kv]$ and $[a+kv,b]$ cover the whole interval $[a,b]$. Therefore, by adding (\ref{est:delta-k-dir}) and (\ref{est:delta-k-rev}) we deduce that
\begin{equation*}
    \int_a^b
    \left|f^{(k)}(x)\right|\,dx
    \leq
    \frac{2^{k+1}}{v^k}
    \int_a^b
    |f(x)|\,dx
    +
    2kv
    \int_a^b
    \left|f^{(k+1)}(x)\right|\,dx.
\end{equation*}

Now we observe that the function $f(x)-m$ satisfies an analogous estimate for every $m\in\re$, and we conclude by considering the infimum over all possible choices of $m$.
\end{proof}


We now apply the interpolation inequality to differences of competitors chosen at consecutive dyadic scales. This yields the quantitative comparison between minimum values that lies at the core of the downward argument.

\begin{prop}[From order $k+1$ to order $k$ -- Estimates for minima]\label{prop:k<k+1}

Let $(a,b)\subseteq\re$ be an interval, let $k$ be a positive integer, and let $g\in L^1((a,b))$ be a function.

Then, for every $\lambda\geq 2k/(b-a)$, one has
\begin{eqnarray}
    \M_k((a,b),g,\lambda)
    & \leq &
    2^{2k+2}
    \sum_{j=0}^{N(\lambda)}
    \frac{1}{2^{(k-1)j}}
    \M_{k+1}
    \left(
        (a,b),
        g,
        \frac{\lambda}{2^j}
    \right)
    \nonumber
    \\
    & &
    \mbox{}+
    \frac{k2^{k+3}}{b-a}
    \osc_1((a,b),g),
    \label{th:k<k+1}
\end{eqnarray}
where
\begin{equation*}
    N(\lambda)
    :=
    \max
    \left\{
        i\geq0:
        \frac{\lambda}{2^i}
        \geq
        \frac{2k}{b-a}
    \right\}
    =
    \left\lfloor
        \log_2
        \frac{\lambda(b-a)}{2k}
    \right\rfloor.
\end{equation*}

\end{prop}

\begin{proof}
Since $(a,b)$ and $g$ are fixed throughout the proof, we omit them from the notation for functionals and oscillations whenever no ambiguity can arise.

For every $j=0,\ldots,N(\lambda)$, let $h_j\in C^{k+1}([a,b])$ be arbitrary, and set
\begin{equation*}
    h_{N(\lambda)+1}(x):=m
    \qquad
    \forall x\in[a,b],
\end{equation*}
where $m$ is the constant that minimizes (\ref{defn:osc-1}). Finally, for every admissible value of $j$ we define
\begin{equation*}
    \lambda_j:=\frac{\lambda}{2^j}
    \qquad\quad\text{and}\quad\qquad
    f_j(x)
    :=
    h_j(x)-h_{j+1}(x)
    \quad
    \forall x\in[a,b].
\end{equation*}

Since
\begin{equation*}
    h_0(x)
    =
    h_{N(\lambda)+1}(x)
    +
    \sum_{j=0}^{N(\lambda)}f_j(x),
\end{equation*}
and $h_{N(\lambda)+1}$ is constant, we obtain that
\begin{equation}
    \int_a^b
    \left|h_0^{(k)}(x)\right|\,dx
    \leq
    \sum_{j=0}^{N(\lambda)}
    \int_a^b
    \left|f_j^{(k)}(x)\right|\,dx.
    \label{est:h0k-basic}
\end{equation}

For every $j=0,\ldots,N(\lambda)$, we apply Lemma~\ref{lemma:Dk} with
\begin{equation*}
    f:=f_j,
    \qquad\quad\text{and}\quad\qquad
    v
    :=
    \frac{1}{\lambda_j}
    =
    \frac{2^j}{\lambda}.
\end{equation*}
Since $\lambda_j\geq 2k/(b-a)$, the choice of $v$ is admissible, and hence
\begin{equation}
    \int_a^b
    \left|f_j^{(k)}(x)\right|\,dx
    \leq
    2^{k+1}\lambda_j^k\osc_1(f_j)
    +
    \frac{2k}{\lambda_j}
    \int_a^b
    \left|f_j^{(k+1)}(x)\right|\,dx.
    \label{est:fj}
\end{equation}

Now we observe that
\begin{eqnarray*}
    \osc_1(f_j)
    & \leq &
    \int_a^b
    |h_j(x)-h_{j+1}(x)|\,dx
    \\
    & \leq &
    \int_a^b
    |h_j(x)-g(x)|\,dx
    +
    \int_a^b
    |h_{j+1}(x)-g(x)|\,dx,
\end{eqnarray*}
while
\begin{equation*}
    \int_a^b
    \left|f_j^{(k+1)}(x)\right|\,dx
    \leq
    \int_a^b
    \left|h_j^{(k+1)}(x)\right|\,dx
    +
    \int_a^b
    \left|h_{j+1}^{(k+1)}(x)\right|\,dx.
\end{equation*}

Since $\lambda_{j+1}=\lambda_j/2$ and $k\leq2^k$, from these estimates and~(\ref{est:fj}) we deduce that
\begin{equation}
    \int_a^b
    \left|f_j^{(k)}(x)\right|\,dx
    \leq
    \frac{\lambda^{k-1}}{2^{(k-1)j}}
    2^{k+2}
    \left\{
        \F_{k+1}(\lambda_j,h_j)
        +
        \F_{k+1}(\lambda_{j+1},h_{j+1})
    \right\}.
    \label{est:fj-F}
\end{equation} 

Combining (\ref{est:h0k-basic}) and (\ref{est:fj-F}) we obtain
\begin{eqnarray}
    \frac{1}{\lambda^{k-1}}
    \int_a^b
    \left|h_0^{(k)}(x)\right|\,dx
    & \leq &
    2^{k+2}\F_{k+1}(\lambda_0,h_0)
    \nonumber
    \\
    & &
    \mbox{}+
    2^{k+2}
    \sum_{j=1}^{N(\lambda)}
    \frac{1+2^{k-1}}{2^{(k-1)j}}
    \F_{k+1}(\lambda_j,h_j)
    \nonumber  
    \\
    & &
    \mbox{}+
    \frac{2^{k+2}}{2^{(k-1)N(\lambda)}}
    \F_{k+1}
    (\lambda_{N(\lambda)+1},h_{N(\lambda)+1}).
    \label{est:h0k-dyadic}
\end{eqnarray}

On the other hand,
\begin{equation*}
    \lambda
    \int_a^b
    |h_0(x)-g(x)|\,dx
    \leq
    \F_{k+1}(\lambda_0,h_0),
\end{equation*}
and since
\begin{equation*}
    2^{k+2}+1
    \leq
    2^{2k+2}
    \qquad\quad\text{and}\quad\qquad
    2^{k+2}(1+2^{k-1})
    \leq
    2^{2k+2},
\end{equation*}
from (\ref{est:h0k-dyadic}) we conclude that
\begin{eqnarray}
    \F_k(\lambda,h_0)
    & \leq &
    2^{2k+2}
    \sum_{j=0}^{N(\lambda)}
    \frac{1}{2^{(k-1)j}}
    \F_{k+1}(\lambda_j,h_j)
    \nonumber
    \\
    & &
    \mbox{}+
    \frac{2^{k+2}}{2^{(k-1)N(\lambda)}}
    \F_{k+1}
    (\lambda_{N(\lambda)+1},h_{N(\lambda)+1}).
    \label{est:Fk-dyadic}
\end{eqnarray}

By the definition of $N(\lambda)$, we have
\begin{equation*}
    \lambda_{N(\lambda)+1}
    <
    \frac{2k}{b-a}.
\end{equation*}
Since $h_{N(\lambda)+1}\equiv m$, it follows that
\begin{equation*}
    \F_{k+1}
    (\lambda_{N(\lambda)+1},h_{N(\lambda)+1})
    =
    \lambda_{N(\lambda)+1}\osc_1(g)
    \leq
    \frac{2k}{b-a}\osc_1(g).
\end{equation*}
Therefore (\ref{est:Fk-dyadic}) yields
\begin{equation}
    \M_k((a,b),g,\lambda)
    \leq
    \F_k(\lambda,h_0)
    \leq
    2^{2k+2}
    \sum_{j=0}^{N(\lambda)}
    \frac{1}{2^{(k-1)j}}
    \F_{k+1}(\lambda_j,h_j)+
    \frac{k2^{k+3}}{b-a}
    \osc_1(g).
    \label{est:Fk-dyadic-final}
\end{equation}

The functions $h_0,\ldots,h_{N(\lambda)}$ are arbitrary and can be chosen independently. Therefore, taking the infimum of the right-hand side of (\ref{est:Fk-dyadic-final}) with respect to each $h_j$ yields the corresponding sum of the minimum values $\M_{k+1}$. This proves (\ref{th:k<k+1}).
\end{proof}


Finally, we are in a position to prove the main result of this subsection,
namely the inclusion from spaces of order $k+1$ to spaces of order $k$.

\begin{prop}[From order $k+1$ to order $k$ -- Inclusions between spaces]

Let $(a,b)\subseteq\re$ be an interval, and let $k$ be a positive integer.

Then for every nondecreasing weight function
$\varphi:(0,+\infty)\to[1,+\infty)$ one has
\begin{equation*}
    \Sp_{k+1}^\varphi((a,b))
    \subseteq
    \Sp_k^{\varphi_*}((a,b)),
\end{equation*}
where
\begin{equation}
    \varphi_*(\lambda)
    :=
    \sum_{i=0}^{\lfloor\log_2(1+\lambda)\rfloor}
    \frac{1}{2^{(k-1)i}}
    \varphi\left(\frac{\lambda}{2^i}\right)
    \qquad
    \forall\lambda>0.
    \label{defn:phi*}
\end{equation}

\end{prop}

\begin{proof}

Let $g\in\Sp_{k+1}^\varphi((a,b))$, so that by definition there exists a constant $C_0>0$ such that
\begin{equation}
    \M_{k+1}((a,b),g,\lambda)
    \leq
    C_0\varphi(\lambda)
    \qquad
    \forall\lambda>0.
    \label{hp:C0}
\end{equation}

Let us set
\begin{equation*}
    \alpha
    :=
    \frac{2k}{b-a}.
\end{equation*}
For every $\lambda\geq\alpha$, Proposition~\ref{prop:k<k+1} and (\ref{hp:C0}) give
\begin{equation}
    \M_k((a,b),g,\lambda)
    \leq
    C_0 2^{2k+2}
    \sum_{j=0}^{N(\lambda)}
    \frac{1}{2^{(k-1)j}}
    \varphi\left(\frac{\lambda}{2^j}\right)+
    \frac{k2^{k+3}}{b-a}
    \osc_1((a,b),g),
    \label{est:Mk-phi-dyadic}
\end{equation}
where
\begin{equation*}
    N(\lambda)
    =
    \left\lfloor
        \log_2\frac{\lambda}{\alpha}
    \right\rfloor.
\end{equation*}

On the other hand, let
\begin{equation*}
    N_*(\lambda)
    :=
    \lfloor\log_2(1+\lambda)\rfloor.
\end{equation*}
The excess of $N(\lambda)$ over $N_*(\lambda)$ is bounded independently
of $\lambda$. More precisely,
\begin{equation}
    N(\lambda)-N_*(\lambda)
    \leq
    |\log_2\alpha|+1.
    \label{est:N-N*}
\end{equation}

If $N(\lambda)\leq N_*(\lambda)$, then
\begin{equation*}
    \sum_{j=0}^{N(\lambda)}
    \frac{1}{2^{(k-1)j}}
    \varphi\left(\frac{\lambda}{2^j}\right)
    \leq
    \varphi_*(\lambda).
\end{equation*}

If $N(\lambda)>N_*(\lambda)$, then from (\ref{est:N-N*}), the fact that
$\varphi$ is nondecreasing, and the inequality
$2^{-(k-1)j}\leq1$, we obtain
\begin{equation*}
    \sum_{j=0}^{N(\lambda)}
    \frac{1}{2^{(k-1)j}}
    \varphi\left(\frac{\lambda}{2^j}\right)
    \leq
    \varphi_*(\lambda)
    +
    \sum_{j=N_*(\lambda)+1}^{N(\lambda)}
    \varphi(\lambda)
    \leq
    \left(
        |\log_2\alpha|+2
    \right)
    \varphi_*(\lambda),
\end{equation*}
because the term with $j=0$ in~(\ref{defn:phi*}) implies $\varphi(\lambda)\leq\varphi_*(\lambda)$. Therefore, in both cases there exists a constant $C_1>0$, depending
only on $\alpha$, such that
\begin{equation*}
    \sum_{j=0}^{N(\lambda)}
    \frac{1}{2^{(k-1)j}}
    \varphi\left(\frac{\lambda}{2^j}\right)
    \leq
    C_1\varphi_*(\lambda)
    \qquad
    \forall\lambda\geq\alpha.
\end{equation*}

Since $\varphi_*(\lambda)\geq\varphi(\lambda)\geq1$, from
(\ref{est:Mk-phi-dyadic}) we conclude that
\begin{equation*}
    \M_k((a,b),g,\lambda)
    \leq
    C_2\varphi_*(\lambda)
    \qquad
    \forall\lambda\geq\alpha
\end{equation*}
for a suitable constant $C_2>0$.

Finally, when $0<\lambda\leq\alpha$, estimate~(\ref{est:Mk-osc}) yields
\begin{equation*}
    \M_k((a,b),g,\lambda)
    \leq
    \lambda\osc_1((a,b),g)
    \leq
    \alpha\osc_1((a,b),g)\varphi_*(\lambda),
\end{equation*}
again because $\varphi_*(\lambda)\geq1$.

We have therefore proved that
\begin{equation*}
    \sup_{\lambda>0}
    \frac{\M_k((a,b),g,\lambda)}
         {\varphi_*(\lambda)}
    <+\infty,
\end{equation*}
and hence $g\in\Sp_k^{\varphi_*}((a,b))$.
\end{proof}


\subsection{Proofs of the remaining main results}

We complete the proofs of the results stated at the beginning of the section. We first establish the fractional Sobolev estimate, and then collect the consequences of the comparison results obtained in the previous subsection.

\subsubsection{Generalized fractional Sobolev regularity}

We begin with statement~(3) of Theorem~\ref{thm:Sk-general}.

To begin with, thanks to a symmetry argument and a change of variable, we rewrite the left-hand side of~(\ref{th:fractional}) in the form
\begin{equation}
    2\int_a^b dx
    \int_x^b
    \frac{|g(y)-g(x)|}
         {|y-x|\psi(|y-x|)}
    \,dy
    =
    2\int_0^{b-a}
    \frac{G(w)}
         {w\psi(w)}
    \,dw,
    \label{eqn:psi-w}
\end{equation}
where
\begin{equation*}
    G(w)
    :=
    \int_a^{b-w}
    |g(x+w)-g(x)|\,dx
    \qquad
    \forall w\in(0,b-a).
\end{equation*}

For every function $\gamma\in C^1([a,b])$ we can write
\begin{eqnarray}
    G(w)
    & \leq &
    \int_a^{b-w}
    |g(x+w)-\gamma(x+w)|\,dx
    +
    \int_a^{b-w}
    |\gamma(x)-g(x)|\,dx
    \nonumber
    \\
    & &
    \mbox{}+
    \int_a^{b-w}
    |\gamma(x+w)-\gamma(x)|\,dx
    \nonumber
    \\
    & \leq &
    2\int_a^b
    |\gamma(x)-g(x)|\,dx
    +
    \int_a^{b-w}
    |\gamma(x+w)-\gamma(x)|\,dx.
    \label{eqn:triangular}
\end{eqnarray}

As for the last integral, we observe that
\begin{equation*}
    \int_a^{b-w}
    |\gamma(x+w)-\gamma(x)|\,dx
    \leq
    \int_a^{b-w}
    dx
    \int_x^{x+w}
    |\gamma'(y)|\,dy.
\end{equation*}

In the last double integral, the sections of the integration region with fixed $y\in[a,b]$ have length less than or equal to $w$. Therefore, by reversing the integration order, we obtain
\begin{equation*}
    \int_a^{b-w}
    |\gamma(x+w)-\gamma(x)|\,dx
    \leq
    w\int_a^b|\gamma'(y)|\,dy.
\end{equation*}

Plugging these estimates into~(\ref{eqn:triangular}), we deduce that
\begin{equation*}
    G(w)
    \leq
    2\int_a^b|\gamma(x)-g(x)|\,dx
    +
    w\int_a^b|\gamma'(y)|\,dy
    \leq
    2w\F_1((a,b),g,1/w,\gamma).
\end{equation*}

Taking the infimum over all functions $\gamma\in C^1([a,b])$, we conclude that
\begin{equation*}
    G(w)
    \leq
    2w\M_1((a,b),g,1/w)
    \qquad
    \forall w\in(0,b-a).
\end{equation*}

Plugging this estimate into~(\ref{eqn:psi-w}), we obtain~(\ref{th:fractional}).

If now $g\in\Sp_1^\varphi((a,b))$, then
\begin{equation*}
    \M_1((a,b),g,\lambda)
    \leq
    C\varphi(\lambda)
    \qquad
    \forall\lambda>0
\end{equation*}
for some constant $C>0$. Therefore~(\ref{th:fractional}) yields
\begin{equation*}
    \int_a^b dy
    \int_a^b
    \frac{|g(y)-g(x)|}
         {|y-x|\psi(|y-x|)}
    \,dx
    \leq
    4C
    \int_0^{b-a}
    \frac{\varphi(1/w)}
         {\psi(w)}
    \,dw.
\end{equation*}
This proves~(\ref{th:Sk-Ws}) under the integrability assumption stated in Theorem~\ref{thm:Sk-general}.
\qed

\subsubsection{Proof of Theorem~\ref{thm:Sk-general}}

\paragraph{\textmd{\textit{Statement~(1)}}}

We need to prove that
\begin{equation*}
    \Sp_k^\varphi((a,b))
    =
    \Sp_{k+1}^\varphi((a,b))
    \qquad
    \forall k\geq2.
\end{equation*}

The inclusion
\begin{equation*}
    \Sp_k^\varphi((a,b))
    \subseteq
    \Sp_{k+1}^\varphi((a,b))
\end{equation*}
is exactly Proposition~\ref{prop:k+1<k}.

As for the opposite inclusion, we know from the comparison obtained in the previous subsection that
\begin{equation*}
    \Sp_{k+1}^\varphi((a,b))
    \subseteq
    \Sp_k^{\varphi_*}((a,b)),
\end{equation*}
where $\varphi_*$ is given by~(\ref{defn:phi*}). Since $\varphi$ is nondecreasing,
\begin{equation*}
    \varphi_*(\lambda)
    \leq
    \varphi(\lambda)
    \sum_{i=0}^{\lfloor\log_2(1+\lambda)\rfloor}
    \frac{1}{2^{(k-1)i}}.
\end{equation*}

If $k\geq2$, the sum on the right-hand side is bounded by~2. Moreover, the term corresponding to $i=0$ in the definition of $\varphi_*$ shows that
\begin{equation*}
    \varphi(\lambda)
    \leq
    \varphi_*(\lambda)
    \leq
    2\varphi(\lambda)
    \qquad
    \forall\lambda>0.
\end{equation*}
Therefore the two weight functions are equivalent, and hence
\begin{equation*}
    \Sp_k^{\varphi_*}((a,b))
    =
    \Sp_k^\varphi((a,b)).
\end{equation*}
This proves statement~(1).

\paragraph{\textmd{\textit{Statement~(2)}}}

The first inclusion in~(\ref{th:S1-S2}) follows again from Proposition~\ref{prop:k+1<k}, which is valid also for $k=1$.

For the opposite direction, we apply the comparison from order $2$ to order $1$. When $k=1$, all the factors
\begin{equation*}
    \frac{1}{2^{(k-1)i}}
\end{equation*}
in~(\ref{defn:phi*}) are equal to~1. Therefore $\varphi_*$ coincides with $\widehat{\varphi}$ defined in~(\ref{defn:new-phi}), and we obtain
\begin{equation*}
    \Sp_2^\varphi((a,b))
    \subseteq
    \Sp_1^{\widehat{\varphi}}((a,b)).
\end{equation*}

\paragraph{\textmd{\textit{Statement~(3)}}}

This is the generalized fractional Sobolev estimate proved above.
\qed

\subsubsection{Proof of Corollary~\ref{cor:constant}}

If $\varphi(\lambda)\equiv1$, then~(\ref{defn:new-phi}) gives
\begin{equation*}
    \widehat{\varphi}(\lambda)
    =
    1+\lfloor\log_2(1+\lambda)\rfloor.
\end{equation*}
The first inclusion in~(\ref{th:cor-constant}) therefore follows from statement~(2) of Theorem~\ref{thm:Sk-general}.

For the second inclusion, we apply statement~(3) of Theorem~\ref{thm:Sk-general}. It is enough to verify that
\begin{equation*}
    \int_0^{b-a}\frac{\widehat{\varphi}(1/w)}{\psi(w)}\,dw<+\infty.
\end{equation*}

From (\ref{defn:psi-psi-constant}) it follows that the integrand is bounded, for $w$ close to zero, by a constant multiple of
\begin{equation*}
    \frac{1}{w\log^2(1/w)},
\end{equation*}
which is integrable at the origin. Away from the origin there is nothing to prove. Therefore statement~(3) of Theorem~\ref{thm:Sk-general} yields the required inclusion.
\qed

\subsubsection{Proof of Corollary~\ref{cor:power}}

The inclusion
\begin{equation*}
    \Sp_1^\varphi((a,b))
    \subseteq
    \Sp_2^\varphi((a,b))
\end{equation*}
is exactly the first inclusion in~(\ref{th:S1-S2}).

As for the opposite inclusion, we know from~(\ref{th:S1-S2}) that
\begin{equation*}
    \Sp_2^\varphi((a,b))
    \subseteq
    \Sp_1^{\widehat{\varphi}}((a,b)).
\end{equation*}
Therefore, when
\begin{equation*}
    \varphi(\lambda)=(1+\lambda)^s,
    \qquad s\in(0,1),
\end{equation*}
it is enough to prove that $\widehat{\varphi}$ is bounded from above by a constant multiple of $\varphi$.

For bounded $\lambda$ the estimate is immediate, while for $\lambda\geq1$ a geometric estimate yields 
\begin{equation*}
    \widehat{\varphi}(\lambda)
    \leq
    C_s\varphi(\lambda)
    \qquad
    \forall\lambda>0
\end{equation*}
for a suitable constant $C_s>0$. Combining the two cases,
\begin{equation*}
    \Sp_2^\varphi((a,b))
    =
    \Sp_1^\varphi((a,b)).
\end{equation*}

Together with statement~(1) of Theorem~\ref{thm:Sk-general}, this also shows that all spaces $\Sp_k^\varphi((a,b))$ coincide in this case.

The inclusion in~(\ref{th:cor-power}) follows from statement~(3) of Theorem~\ref{thm:Sk-general}. Indeed, when $\varphi$ and $\psi$ are defined by (\ref{defn:phi-cor-power}) and (\ref{defn:psi-cor-power}), one has
\begin{equation*}
    \frac{\varphi(1/w)}{\psi(w)}
    \sim
    \frac{1}{w\log^2(1/w)}
    \qquad
    \text{as }w\to0^+,
\end{equation*}
and the latter function is integrable at the origin, and hence $\Sp_1^\varphi((a,b))\subseteq W^{\psi,1}((a,b))$.

Since
\[
\psi(\sigma)
=
\sigma^{1-s}\log^2(1+1/\sigma)
=
o(\sigma^r)
\qquad
\text{as }\sigma\to0^+
\]
for every $r<1-s$, one has
\[
W^{\psi,1}((a,b))
\subseteq
\bigcap_{r<1-s}W^{r,1}((a,b)).
\]
The strictness of this last inclusion is classical, which completes the proof of (\ref{th:cor-power-strict}). 
\qed

\subsubsection{Proof of Theorem~\ref{thm:Sk-basic}}

\paragraph{\textmd{\textit{Proof that
$BV((a,b))=\Sp_1((a,b))$}}}

The inclusion
\begin{equation*}
    BV((a,b))
    \subseteq
    \Sp_1((a,b))
\end{equation*}
follows by approximation of $BV$ functions by smooth functions with uniformly bounded total variation.

Conversely, let $g\in\Sp_1((a,b))$, so that there exists a constant $C_0$ such that
\begin{equation*}
    \M_1((a,b),g,\lambda)
    \leq
    C_0
    \qquad
    \forall\lambda>0.
\end{equation*}
For every $\lambda>0$, we can therefore find
$\gamma_\lambda\in C^1([a,b])$ such that
\begin{equation*}
    \F_1
    ((a,b),g,\lambda,\gamma_\lambda)
    \leq
    2C_0.
\end{equation*}
In particular,
\begin{equation*}
    \int_a^b
    |\gamma_\lambda'(x)|\,dx
    \leq
    2C_0
    \qquad\quad\text{and}\quad\qquad
    \int_a^b
    |\gamma_\lambda(x)-g(x)|\,dx
    \leq
    \frac{2C_0}{\lambda}.
\end{equation*}

Hence $\gamma_\lambda\to g$ in $L^1((a,b))$ as $\lambda\to+\infty$, while their total variations remain uniformly bounded. By lower semicontinuity of the total variation, we conclude that $g\in BV((a,b))$, as required.

\paragraph{\textmd{\textit{Inclusions and saturation}}}

The inclusion
\[
\Sp_1((a,b))
\subseteq
\Sp_2((a,b))
\]
is a special case of~(\ref{th:S1-S2}), while its strictness follows from Example~\ref{ex:S2-not-S1}.

Moreover,
\[
\Sp_2((a,b))
=
\Sp_k((a,b))
\qquad
\forall k\geq2
\]
is a special case of statement~(1) of Theorem~\ref{thm:Sk-general}.

\paragraph{\textmd{\textit{Fractional Sobolev regularity}}}

From Corollary~\ref{cor:constant}, we know that
\begin{equation*}
    \Sp_2((a,b))
    \subseteq
    W^{\psi,1}((a,b))
\end{equation*}
with
\begin{equation*}
    \psi(\sigma)
    :=
    \sigma\log^3(1+1/\sigma).
\end{equation*}
Since $\psi(\sigma)=o(\sigma^s)$ as $\sigma\to0^+$ for every $s\in(0,1)$, one has
\begin{equation*}
    W^{\psi,1}((a,b))
    \subseteq
    \bigcap_{s\in(0,1)}
    W^{s,1}((a,b)).
\end{equation*}
The strictness of this last inclusion is classical, and this completes the proof of (\ref{th:Sk-basic}).
\qed

\subsubsection{Proof of Example~\ref{ex:S2-not-S1}}
\label{sec:proof-example}

Let $g$ be the function defined by (\ref{defn:sin-log}). It is rather standard to verify that
\[
\int_0^1|g'(x)|\,dx=+\infty,
\]
and therefore
\[
g\notin BV((0,1))
=
\Sp_1((0,1)).
\]

We now prove that $g\in\Sp_2((0,1))$. The positive maximum points of $g$ form a geometric sequence converging to zero. Hence there exists a constant $c_0>0$ such that, for every sufficiently large $\lambda$, one can choose a maximum point $\delta_\lambda$ satisfying
\[
\frac{c_0}{\lambda}
\leq
\delta_\lambda
\leq
\frac1\lambda
\]
so that in particular $g(\delta_\lambda)=1$ and $g'(\delta_\lambda)=0$. If we define
\[
\gamma_\lambda(x)
:=
\begin{cases}
1 & \text{if }x\in[0,\delta_\lambda],\\
g(x) & \text{if }x\in[\delta_\lambda,1],
\end{cases}
\]
then $\gamma_\lambda\in C^1([0,1])$, and
\[
\lambda\int_0^1|\gamma_\lambda-g|
\leq
2\lambda\delta_\lambda
\leq
2.
\]

Moreover, since $|g''(x)|\leq 2/x^2$ for every $x\in(0,1)$, we deduce that
\[
\frac1\lambda\int_0^1|\gamma_\lambda''(x)|\,dx
=
\frac1\lambda\int_{\delta_\lambda}^1|g''(x)|\,dx
\leq
\frac{2}{\lambda\delta_\lambda}
\leq
\frac{2}{c_0}.
\]

After an arbitrarily small mollification around $\delta_\lambda$, we obtain competitors in $C^2([0,1])$ with the same bounds up to an arbitrarily small error. Hence
\[
\sup_{\lambda\geq\lambda_0}
\M_2((0,1),g,\lambda)
<+\infty.
\]
For bounded values of $\lambda$, the minimum values are bounded by~(\ref{est:Mk-osc}), and therefore $g\in\Sp_2((0,1))$.
\qed


\subsection{Proofs for the constrained variational problems}

\subsubsection{First order: proof of Proposition~\ref{prop:M1-constrained}}

The inequality
\begin{equation*}
    \M_1((a,b),g,\lambda)
    \leq
    \M_1^*((a,b),g,\lambda)
\end{equation*}
is immediate from the definitions.

For the converse inequality, let $\gamma\in C^1([a,b])$, and define its truncation by
\begin{equation*}
    \gamma_*(x)
    :=
    \min\left\{
        \mu_2,
        \max\{\mu_1,\gamma(x)\}
    \right\}
    \qquad
    \forall x\in[a,b].
\end{equation*}
By~(\ref{hp:Mk-constrained-range}),
\begin{equation*}
    |\gamma_*(x)-g(x)|
    \leq
    |\gamma(x)-g(x)|
\end{equation*}
for almost every $x\in(a,b)$, while
\begin{equation*}
    |\gamma_*'(x)|
    \leq
    |\gamma'(x)|
\end{equation*}
for almost every $x\in(a,b)$. Hence
\begin{equation*}
    \int_a^b|\gamma_*'(x)|\,dx
    +
    \lambda\int_a^b|\gamma_*(x)-g(x)|\,dx
    \leq
    \F_1((a,b),g,\lambda,\gamma).
\end{equation*}

The function $\gamma_*$ is Lipschitz continuous and takes values in $[\mu_1,\mu_2]$. Extend it constantly outside $[a,b]$ and mollify with a standard nonnegative mollifier. The resulting functions $(\gamma_*)_\delta$ belong to $C^1([a,b])$, still take values in $[\mu_1,\mu_2]$, converge to $\gamma_*$ in $L^1((a,b))$, and satisfy
\begin{equation*}
    \limsup_{\delta\to0^+}
    \int_a^b|(\gamma_*)_\delta'(x)|\,dx
    \leq
    \int_a^b|\gamma_*'(x)|\,dx.
\end{equation*}
Therefore
\begin{equation*}
    \M_1^*((a,b),g,\lambda)
    \leq
    \F_1((a,b),g,\lambda,\gamma).
\end{equation*}
Taking the infimum over all $\gamma\in C^1([a,b])$ proves~(\ref{th:M1-constrained}).
\qed


\subsubsection{Second order: reduction to the admissible range}

We begin by showing that the competitors can be restricted to the admissible range. The key observation is that truncation on interior components does not increase the second-order regularity cost; any additional contribution can only come from components touching the endpoints.

\begin{lemma}[Reduction to the admissible range]
\label{lemma:range-reduction}

Let $(a,b)\subseteq\mathbb R$ be an interval, let $\mu_1$ and $\mu_2$ be positive real numbers with $\mu_1\leq\mu_2$, and let $g\in L^1((a,b))$ satisfy the range constraint (\ref{hp:Mk-constrained-range}).

Then there exists a constant $C>0$, depending only on $\mu_1$, $\mu_2$, and $b-a$, such that, for every $\lambda\geq1$ and every $\gamma\in C^2([a,b])$, there exists $\widehat\gamma\in C^2([a,b])$ satisfying
\begin{equation*}
    \mu_1\leq\widehat\gamma(x)\leq\mu_2
    \qquad
    \forall x\in[a,b],
\end{equation*}
and
\begin{equation*}
    \F_2((a,b),g,\lambda,\widehat\gamma)
    \leq
    C\left(
        \F_2((a,b),g,\lambda,\gamma)
        +
        \frac1\lambda
    \right).
\end{equation*}

\end{lemma}

\begin{proof}

Set
\begin{equation*}
    T:=b-a,
    \qquad
    R:=\int_a^b|\gamma''(x)|\,dx,
    \qquad
    F:=\int_a^b|\gamma(x)-g(x)|\,dx,
\end{equation*}
and define
\begin{equation*}
    \gamma_*(x)
    :=
    \min\left\{
        \mu_2,
        \max\{\mu_1,\gamma(x)\}
    \right\}
    \qquad
    \forall x\in[a,b].
\end{equation*}
Clearly,
\begin{equation}
    \mu_1\leq\gamma_*\leq\mu_2,
    \qquad\quad\text{and}\quad\qquad
    \int_a^b|\gamma_*-g|
    \leq F.
    \label{est:range-clipping}
\end{equation}

The function $\gamma_*$ is Lipschitz continuous, but in general it is not of class $C^1$. We estimate the total variation of its weak derivative by analyzing the connected components of the set where the truncation is active. To this end, we consider the open set
\begin{equation*}
    \Omega
    :=
    \left\{
        x\in(a,b):
        \gamma(x)\notin[\mu_1,\mu_2]
    \right\}.
\end{equation*}

If $\Omega=(a,b)$, then $\gamma$ lies entirely above $\mu_2$ or entirely below $\mu_1$. In this case $\gamma_*$ is constant, so that
\begin{equation*}
    \operatorname{Var}(\gamma_*';(a,b))=0,
\end{equation*}
and the regularity term is trivially improved. We may therefore assume that $\Omega\neq(a,b)$.

We first show that truncation on every connected component of $\Omega$ which is compactly contained in $(a,b)$ does not increase the total variation of the first derivative.

Indeed, let $(\alpha,\beta)$ be any such a component. Then $\gamma$ lies either above $\mu_2$ or below $\mu_1$ on $(\alpha,\beta)$, and in both cases
\begin{equation*}
    \gamma(\alpha)=\gamma(\beta).
\end{equation*}
By Rolle's theorem there exists $\xi\in(\alpha,\beta)$ such that $\gamma'(\xi)=0$. Hence
\begin{equation}
    |\gamma'(\alpha)|+|\gamma'(\beta)|
    \leq
    \int_\alpha^\xi|\gamma''(x)|\,dx
    +
    \int_\xi^\beta|\gamma''(x)|\,dx
    =
    \int_\alpha^\beta|\gamma''(x)|\,dx.
    \label{est:range-interior-component}
\end{equation}
After truncation, $\gamma_*'$ vanishes on $(\alpha,\beta)$, and the left-hand side of~(\ref{est:range-interior-component}) is exactly the contribution of the two new junctions. Thus the variation created at the endpoints of the component is paid for by the variation removed inside it.

Consequently, if all components of $\Omega$ are compactly contained in $(a,b)$, then
\begin{equation*}
    \operatorname{Var}(\gamma_*';(a,b))
    \leq R.
\end{equation*}
The only possible additional contributions therefore come from components touching $a$ or $b$, and there are at most two of them.

We now estimate the slope at the interior endpoint of such a boundary component. More generally, we show that if
\begin{equation*}
    \gamma(x)\in[\mu_1,\mu_2],
\end{equation*}
then
\begin{equation}
    |\gamma'(x)|
    \leq
    R
    +
    \frac{4(\mu_2-\mu_1)}{T}
    +
    \frac{16F}{T^2}.
    \label{est:range-slope}
\end{equation}

Indeed, among the two intervals
\begin{equation*}
    \left(a,a+\frac{T}{4}\right)
    \qquad\mbox{and}\qquad
    \left(b-\frac{T}{4},b\right)
\end{equation*}
we may choose one whose distance from $x$ is at least $T/4$. Since its length is $T/4$, there exists a point $y$ in this interval such that
\begin{equation*}
    |\gamma(y)-g(y)|
    \leq
    \frac{4F}{T}.
\end{equation*}
Since both $\gamma(x)$ and $g(y)$ belong to $[\mu_1,\mu_2]$, we have
\begin{equation*}
    |\gamma(y)-\gamma(x)|
    \leq
    |\gamma(y)-g(y)|
    +
    |g(y)-\gamma(x)|
    \leq
    \frac{4F}{T}
    +
    \mu_2-\mu_1.
\end{equation*}
The mean value theorem gives a point $z$ between $x$ and $y$ such that
\begin{equation*}
    |\gamma'(z)|
    \leq
    \frac{4(\mu_2-\mu_1)}{T}
    +
    \frac{16F}{T^2}.
\end{equation*}
Since
\begin{equation*}
    |\gamma'(x)-\gamma'(z)|
    \leq R,
\end{equation*}
we obtain~(\ref{est:range-slope}).

Applying~(\ref{est:range-slope}) to the interior endpoints of the at most two boundary components, and combining this with~(\ref{est:range-interior-component}), we obtain
\begin{equation*}
    \operatorname{Var}(\gamma_*';(a,b))
    \leq
    3R
    +
    \frac{8(\mu_2-\mu_1)}{T}
    +
    \frac{32F}{T^2}.
\end{equation*}

It remains to recover smoothness. We extend $\gamma_*$ constantly outside $[a,b]$. This extension may create a jump of the first derivative at each endpoint. Whenever such a jump is nonzero, the corresponding endpoint value of $\gamma$ belongs to $[\mu_1,\mu_2]$, and therefore~(\ref{est:range-slope}) applies. Since there are at most two such jumps, we obtain
\begin{equation}
    \operatorname{Var}_{\mathbb R}(\gamma_*')
    \leq
    5R
    +
    \frac{16(\mu_2-\mu_1)}{T}
    +
    \frac{64F}{T^2}.
    \label{est:range-extended-TV}
\end{equation}

Let $\rho_\delta$ be a standard nonnegative mollifier and set
\begin{equation*}
    (\gamma_*)_\delta
    :=
    \rho_\delta*\gamma_*.
\end{equation*}
Since $[\mu_1,\mu_2]$ is convex,
\begin{equation*}
    \mu_1\leq(\gamma_*)_\delta\leq\mu_2.
\end{equation*}
Moreover, the standard BV estimate gives
\begin{equation*}
    \int_a^b|(\gamma_*)_\delta''(x)|\,dx
    \leq
    \operatorname{Var}_{\mathbb R}(\gamma_*').
\end{equation*}

Since $(\gamma_*)_\delta\to\gamma_*$ in $L^1((a,b))$ as $\delta\to0^+$, we may choose $\delta$ sufficiently small to ensure
\begin{equation*}
    \lambda
    \int_a^b
    |(\gamma_*)_\delta-\gamma_*|\,dx
    \leq
    \frac1\lambda.
\end{equation*}

Using~(\ref{est:range-clipping}) and~(\ref{est:range-extended-TV}), we obtain
\begin{equation*}
    \F_2((a,b),g,\lambda,(\gamma_*)_\delta)
    \leq
    \frac{5R}{\lambda}
    +
    \left(
        \lambda+\frac{64}{T^2\lambda}
    \right)F
    +
    \left(
        1+\frac{16(\mu_2-\mu_1)}{T}
    \right)\frac1\lambda.
\end{equation*}
Since $\lambda\geq1$, it follows that
\begin{equation*}
    \F_2((a,b),g,\lambda,(\gamma_*)_\delta)
    \leq
    C\left(
        \F_2((a,b),g,\lambda,\gamma)
        +
        \frac1\lambda
    \right),
\end{equation*}
where $C$ depends only on $\mu_1$, $\mu_2$, and $T=b-a$.

The conclusion follows by setting $\widehat\gamma:=(\gamma_*)_\delta$.
\end{proof}


\subsubsection{Second order: reduction of the first derivative}

The reduction of the first derivative is based on a classical inf-convolution construction. Given a function $f\in C([a,b])$ and a constant $L>0$, we denote by
\begin{equation*}
    \mathcal I_Lf(x)
    :=
    \inf_{y\in[a,b]}
    \left\{
        f(y)+L|x-y|
    \right\}
\end{equation*}
the inf-convolution of $f$ with the function $x\mapsto L|x|$. The function $\mathcal I_Lf$ is the largest $L$-Lipschitz function lying below $f$.

In the present application we choose $L=\ep\lambda$. The one-dimensional geometry of the inf-convolution is particularly simple on the set where $\mathcal I_Lf$ detaches from $f$, and makes it possible to control both the total variation of its derivative and its $L^1$ distance from $f$ in terms of the second derivative of $f$. A final mollification then provides the required smooth competitor.

\begin{lemma}[Reduction of the first derivative]
\label{lemma:first-derivative-reduction}

Let $(a,b)\subseteq\mathbb R$ be an interval, and let $\ep$, $\mu_1$, $\mu_2$, $\lambda$ be positive real numbers with $\mu_1\leq\mu_2$ and
\begin{equation}
    \lambda>
    \frac{2(\mu_2-\mu_1)}
         {\ep(b-a)}.
    \label{hp:first-derivative-lambda}
\end{equation}
Let $\gamma\in C^2([a,b])$ satisfy
\begin{equation}
    \mu_1\leq\gamma(x)\leq\mu_2
    \qquad
    \forall x\in[a,b].
    \label{hp:first-derivative-range}
\end{equation}

Then there exists $\widehat\gamma\in C^2([a,b])$ that satisfies the pointwise estimates
\begin{equation*}
    \mu_1\leq\widehat\gamma(x)\leq\mu_2
    \qquad\quad\text{and}\quad\qquad
    |\widehat\gamma\,'(x)|\leq\ep\lambda
\end{equation*}
for every $x\in[a,b]$, and the integral estimates
\begin{equation}
    \int_a^b|\widehat\gamma\,''(x)|\,dx
    \leq
    \int_a^b|\gamma''(x)|\,dx
    +
    2\ep\lambda,
    \label{est:first-derivative-second}
\end{equation}
and
\begin{equation}
    \int_a^b|\widehat\gamma(x)-\gamma(x)|\,dx
    \leq
    \left(
        1+\frac{(\mu_2-\mu_1)^2}{\ep^2}
    \right)
    \frac1{\lambda^2}
    \int_a^b|\gamma''(x)|\,dx.
    \label{est:first-derivative-fidelity}
\end{equation}

\end{lemma}

\begin{proof}

Set
\begin{equation*}
    L:=\ep\lambda
    \qquad\quad\text{and}\quad\qquad
    \gamma_*:=\mathcal I_L\gamma.
\end{equation*}

\paragraph{\textmd{\textit{Basic properties}}}

Taking $y=x$ in the definition of the inf-convolution gives $\gamma_*\leq\gamma$, while~(\ref{hp:first-derivative-range}) immediately yields $\gamma_*\geq\mu_1$. Therefore
\begin{equation}
    \mu_1\leq\gamma_*(x)\leq\gamma(x)\leq\mu_2
    \qquad
    \forall x\in[a,b].
    \label{est:inf-convolution-range}
\end{equation}
Moreover, $\gamma_*$ is $L$-Lipschitz, and hence
\begin{equation}
    |\gamma_*'(x)|\leq L
    \qquad
    \text{for almost every }x\in(a,b).
    \label{est:inf-convolution-slope}
\end{equation}

\paragraph{\textmd{\textit{Structure of the detachment set}}}
Let
\begin{equation*}
    \Omega:=
    \left\{
        x\in(a,b):
        \gamma_*(x)<\gamma(x)
    \right\}
\end{equation*}
be the set where $\gamma_*$ and $\gamma$ do not coincide. Since $\Omega$ is open, it is the union of at most countably many pairwise disjoint open intervals.

Let $(c,d)$ be any such connected component. We show that there exists $z\in[c,d]$ with the following properties.

\begin{itemize}

    \item If $z>c$, then
    \begin{equation}
        \gamma'(c)=L
        \quad\text{if $c>a$}
        \qquad\quad\text{or}\quad\qquad
        \gamma'(c)\geq L
        \quad\text{if $c=a$},
        \label{th:gamma'c}
    \end{equation}
    and
    \begin{equation*}
        \gamma_*(x)=\gamma(c)+L(x-c)
        \qquad
        \forall x\in[c,z].
    \end{equation*}

    \item If $z<d$, then
    \begin{equation}
        \gamma'(d)=-L
        \quad\text{if $d<b$}
        \qquad\quad\text{or}\quad\qquad
        \gamma'(d)\leq -L
        \quad\text{if $d=b$},
        \label{th:gamma'd}
    \end{equation}
    and
    \begin{equation*}
        \gamma_*(x)=\gamma(d)+L(d-x)
        \qquad
        \forall x\in[z,d].
    \end{equation*}

\end{itemize}

In words, this means that on every connected component of the detachment set the function $\gamma_*$ is either affine with slope $L$, affine with slope $-L$, or consists of two affine pieces with slopes $L$ and $-L$, in this order, with switch at $z$.

In addition, we show that
\begin{equation}
    z-c\leq\frac{\mu_2-\mu_1}{L}
    \qquad\quad\text{and}\quad\qquad
    d-z\leq\frac{\mu_2-\mu_1}{L}.
    \label{est:detachment-branch-length}
\end{equation}
In particular, (\ref{hp:first-derivative-lambda}) excludes the possibility that the whole interval $(a,b)$ is a single component of $\Omega$.

To this end, we take any point $x\in(c,d)$. Since the minimum in the definition of $\gamma_*(x)$ is attained, there exists a point $y\in[a,b]$ such that
\begin{equation*}
    \gamma_*(x)=\gamma(y)+L|x-y|.
\end{equation*}
When this is the case, we say that $y$ supports $x$.

To begin with, we show that we can assume that $x$ is supported by either $c$ or $d$. Indeed, since $\gamma_*$ is Lipschitz continuous and less than or equal to $\gamma$, one has
\begin{equation*}
    \gamma_*(y)+L|x-y|
    \leq
    \gamma(y)+L|x-y|
    =
    \gamma_*(x)
    \leq
    \gamma_*(y)+L|x-y|,
\end{equation*}
which proves that $\gamma_*(y)=\gamma(y)$. Hence either $y\leq c$ or $y\geq d$.

Let us assume that $y\leq c$. Notice that $\gamma_*(c)=\gamma(c)$. Indeed, this follows from the definition of the connected component if $c>a$, while if $c=a$, then $y\leq c$ implies $y=c$, and the equality just proved gives the same conclusion. Since
\begin{equation*}
    \gamma_*(x)
    \leq
    \gamma(c)+L(x-c)
\end{equation*}
and
\begin{equation*}
    \gamma(c)
    \leq
    \gamma(y)+L(c-y),
\end{equation*}
we obtain
\begin{equation*}
    \gamma(y)+L(x-y)
    =
    \gamma_*(x)
    \leq
    \gamma(c)+L(x-c)
    \leq
    \gamma(y)+L(c-y)+L(x-c)
    =
    \gamma(y)+L(x-y).
\end{equation*}
Hence all inequalities are equalities, and therefore
\begin{equation*}
    \gamma_*(x)=\gamma(c)+L(x-c),
\end{equation*}
which means that $c$ supports $x$ as well. A symmetric argument applies if $y\geq d$.

Once we know that $c$ supports some $x\in(c,d)$, the derivative information in~(\ref{th:gamma'c}) follows immediately. Indeed, $c$ minimizes
\begin{equation*}
    y\longmapsto\gamma(y)+L(x-y).
\end{equation*}
If $c>a$, the minimum is attained at an interior point, and therefore $\gamma'(c)=L$. If $c=a$, the corresponding one-sided condition yields $\gamma'(a)\geq L$. The proof of~(\ref{th:gamma'd}) is symmetric.

We now know that every point $x\in(c,d)$ is supported by either $c$ or $d$. Assume that $x$ is supported by $c$. Then every $x_1\in(c,x)$ is supported by $c$ as well. Indeed, from the Lipschitz continuity of $\gamma_*$ we have
\begin{equation*}
    \gamma(c)+L(x_1-c)
    \geq
    \gamma_*(x_1)
    \geq
    \gamma_*(x)-L(x-x_1)
    =
    \gamma(c)+L(x-c)-L(x-x_1).
\end{equation*}
The left-hand side and the right-hand side coincide, and therefore all inequalities are equalities. Hence $c$ supports $x_1$. Symmetrically, if $x$ is supported by $d$, then every point of $(x,d)$ is supported by $d$.

Let $C$ and $D$ denote the set of points in $(c,d)$ that are supported, respectively, by $c$ and $d$. We have proved that $C\cup D=(c,d)$, that $C$ is an initial portion of $(c,d)$, and that $D$ is a final portion of $(c,d)$. Moreover, $C\cap D$ contains at most one point, because otherwise $\gamma_*$ would be affine with slopes $L$ and $-L$ on the same nontrivial interval.

Let $z:=\sup C$, with the convention $z=c$ if $C$ is empty. Then
\begin{equation*}
    \gamma_*(x)=\gamma(c)+L(x-c)
    \qquad
    \forall x\in(c,z),
\end{equation*}
and
\begin{equation*}
    \gamma_*(x)=\gamma(d)+L(d-x)
    \qquad
    \forall x\in(z,d).
\end{equation*}
By continuity, the corresponding identities hold also at $z$ whenever the relevant branch is nonempty. This proves the announced structure.

Finally, (\ref{est:detachment-branch-length}) follows because $\gamma_*$ takes values in $[\mu_1,\mu_2]$ and is affine with slope of absolute value $L$ on each nonempty branch.

\paragraph{\textmd{\textit{Variation estimate}}}

We now prove that
\begin{equation}
    \operatorname{Var}(\gamma_*';(a,b))
    \leq
    \int_a^b|\gamma''(x)|\,dx.
    \label{est:inf-convolution-variation}
\end{equation}

Let $\{J_j\}$ be an enumeration of the connected components of $\Omega$. For every positive integer $N$, define
\begin{equation*}
    \gamma_N(x)
    :=
    \begin{cases}
        \gamma_*(x),
        & x\in J_1\cup\cdots\cup J_N,\\[1ex]
        \gamma(x),
        & \text{otherwise}.
    \end{cases}
\end{equation*}
At every interior endpoint of one of the components $J_j$ the functions $\gamma_*$ and $\gamma$ coincide, and therefore $\gamma_N$ is continuous in $(a,b)$. Moreover, $\gamma_N'\in BV((a,b))$, since only finitely many components are modified.

Consider first a component $(c,d)$ with $a<c<d<b$. If $\gamma_*$ is affine with slope $L$ on $(c,d)$, then~(\ref{th:gamma'c}) gives $\gamma'(c)=L$. Hence no jump is produced at $c$, while
\begin{equation*}
    |L-\gamma'(d)|
    =
    |\gamma'(c)-\gamma'(d)|
    \leq
    \int_c^d|\gamma''(x)|\,dx.
\end{equation*}
If $\gamma_*$ is affine with slope $-L$, then $\gamma'(d)=-L$, and symmetrically
\begin{equation*}
    |\gamma'(c)+L|
    =
    |\gamma'(c)-\gamma'(d)|
    \leq
    \int_c^d|\gamma''(x)|\,dx.
\end{equation*}
Finally, if $\gamma_*$ has a genuine corner in some $z\in(c,d)$, then
\begin{equation*}
    \gamma'(c)=L
    \qquad\quad\text{and}\quad\qquad
    \gamma'(d)=-L.
\end{equation*}
No jump is produced at the endpoints, while the jump at the corner has size
\begin{equation*}
    2L
    =
    |\gamma'(c)-\gamma'(d)|
    \leq
    \int_c^d|\gamma''(x)|\,dx.
\end{equation*}

The same estimate holds for a component touching the boundary. For example, let $(a,d)$ be this component. If only the right branch is present, then $\gamma_*'=-L$ on $(a,d)$ and $\gamma'(d)=-L$, so there is no contribution to the variation. If only the left branch is present, then $\gamma'(d)\leq L$, while~(\ref{th:gamma'c}) gives $\gamma'(a)\geq L$. Hence
\begin{equation*}
    |L-\gamma'(d)|
    =
    L-\gamma'(d)
    \leq
    \gamma'(a)-\gamma'(d)
    \leq
    \int_a^d|\gamma''(x)|\,dx.
\end{equation*}
If both branches are present, then $\gamma'(d)=-L$ and~(\ref{th:gamma'c}) gives
\begin{equation*}
    2L
    \leq
    \gamma'(a)-\gamma'(d)
    \leq
    \int_a^d|\gamma''(x)|\,dx.
\end{equation*}
The case of a component touching $b$ is symmetric.

It follows that
\begin{equation*}
    \operatorname{Var}(\gamma_N';(a,b))
    \leq
    \int_{(a,b)\setminus\bigcup_{j=1}^N J_j}
    |\gamma''(x)|\,dx
    +
    \sum_{j=1}^N
    \int_{J_j}|\gamma''(x)|\,dx
    =
    \int_a^b|\gamma''(x)|\,dx.
\end{equation*}

For almost every $x\in\Omega$, the point $x$ belongs to one of the components $J_j$, and therefore
\begin{equation*}
    \gamma_N'(x)=\gamma_*'(x)
\end{equation*}
for all sufficiently large $N$.

On the complement of $\Omega$, the function $\gamma-\gamma_*$ vanishes. Since the isolated points of the contact set are at most countably many, almost every point of the contact set is a limit of other contact points. Therefore, at every such point where $\gamma_*$ is differentiable,
\begin{equation*}
    \gamma_*'(x)=\gamma'(x),
\end{equation*}
and therefore
\begin{equation*}
    \gamma_N'(x)\longrightarrow\gamma_*'(x)
    \qquad
    \text{for almost every }x\in(a,b).
\end{equation*}
Moreover, they satisfy a uniform bound because
\begin{equation*}
    |\gamma_N'(x)|
    \leq
    \max\{|\gamma'(x)|,L\}
    \qquad
    \text{for almost every }x\in(a,b),
\end{equation*}
and therefore the dominated convergence theorem yields
\begin{equation*}
    \gamma_N'\longrightarrow\gamma_*'
    \qquad
    \text{in }L^1((a,b)).
\end{equation*}
By the lower semicontinuity of the total variation with respect to $L^1$ convergence, we conclude that
\begin{equation*}
    \operatorname{Var}(\gamma_*';(a,b))
    \leq
    \liminf_{N\to\infty}
    \operatorname{Var}(\gamma_N';(a,b))
    \leq
    \int_a^b|\gamma''(x)|\,dx,
\end{equation*}
which proves~(\ref{est:inf-convolution-variation}).

\paragraph{\textmd{\textit{Fidelity estimate}}}

We now prove that
\begin{equation}
    \int_a^b|\gamma(x)-\gamma_*(x)|\,dx
    \leq
    \frac{(\mu_2-\mu_1)^2}{L^2}
    \int_a^b|\gamma''(x)|\,dx.
    \label{est:inf-convolution-fidelity}
\end{equation}

To this end, we consider the difference
\begin{equation*}
    w(x):=\gamma(x)-\gamma_*(x)
    \qquad
    \forall x\in[a,b].
\end{equation*}
Let $(c,d)$ be a connected component of $\Omega$, and let $z\in[c,d]$ be the point provided by the structure above.

Suppose first that $c>a$ and that the left branch is nonempty. Then
\begin{equation*}
    w(c)=0
    \qquad\quad\text{and}\quad\qquad
    w'(c)=0.
\end{equation*}
Since $\gamma_*$ is affine on $(c,z)$, one has
\begin{equation*}
    w(x)
    =
    \int_c^x(x-t)\gamma''(t)\,dt
    \qquad
    \forall x\in[c,z].
\end{equation*}
Consequently
\begin{equation*}
    \int_c^z|w(x)|\,dx
    \leq
    \frac12
    \int_c^z(z-t)^2|\gamma''(t)|\,dt
    \leq
    \frac{(\mu_2-\mu_1)^2}{2L^2}
    \int_c^z|\gamma''(t)|\,dt,
\end{equation*}
where we used~(\ref{est:detachment-branch-length}). Similarly, if $d<b$ and the right branch is nonempty, then
\begin{equation}
    \int_z^d|w(x)|\,dx
    \leq
    \frac{(\mu_2-\mu_1)^2}{2L^2}
    \int_z^d|\gamma''(t)|\,dt.
    \label{est:fidelity-right}
\end{equation}

Thus every component compactly contained in $(a,b)$ satisfies
\begin{equation*}
    \int_c^d|\gamma(x)-\gamma_*(x)|\,dx
    \leq
    \frac{(\mu_2-\mu_1)^2}{2L^2}
    \int_c^d|\gamma''(x)|\,dx.
\end{equation*}

We now consider a component touching the boundary, say of the form $(a,d)$. The right branch, if nonempty, is estimated exactly as in~(\ref{est:fidelity-right}). If the left branch is nonempty, then
\begin{equation*}
    w(a)=0,
    \qquad
    w'(a)=\gamma'(a)-L\geq0.
\end{equation*}
If only the left branch is present, the contact condition at $d$ gives $\gamma'(d)\leq L$, while if both branches are present one has $\gamma'(d)=-L$. In either case
\begin{equation}
    0\leq w'(a)
    \leq
    \int_a^d|\gamma''(x)|\,dx.
    \label{est:boundary-initial-slope}
\end{equation}
Hence, on the left branch,
\begin{equation*}
    w(x)
    =
    w'(a)(x-a)
    +
    \int_a^x(x-t)\gamma''(t)\,dt,
\end{equation*}
and therefore
\begin{equation*}
    \int_a^z|w(x)|\,dx
    \leq
    \frac{(z-a)^2}{2}
    \left(
        w'(a)
        +
        \int_a^z|\gamma''(t)|\,dt
    \right).
\end{equation*}
Using~(\ref{est:detachment-branch-length}) and~(\ref{est:boundary-initial-slope}), we obtain
\begin{equation*}
    \int_a^z|w(x)|\,dx
    \leq
    \frac{(\mu_2-\mu_1)^2}{2L^2}
    \left(
        \int_a^d|\gamma''(t)|\,dt
        +
        \int_a^z|\gamma''(t)|\,dt
    \right).
\end{equation*}
Combining this estimate with~(\ref{est:fidelity-right}) when the right branch is nonempty, we conclude that
\begin{equation*}
    \int_a^d|\gamma(x)-\gamma_*(x)|\,dx
    \leq
    \frac{(\mu_2-\mu_1)^2}{L^2}
    \int_a^d|\gamma''(x)|\,dx.
\end{equation*}
The case of a component touching $b$ is symmetric.

Summing the corresponding estimates over the at most countably many connected components of $\Omega$ proves~(\ref{est:inf-convolution-fidelity}).

\paragraph{\textmd{\textit{Final regularization}}}

It remains to regularize $\gamma_*$ without losing the range and first-derivative bounds, while increasing the second-order cost by at most $2L$.

If
\begin{equation*}
    \int_a^b|\gamma''(x)|\,dx=0,
\end{equation*}
then $\gamma$ is affine and~(\ref{hp:first-derivative-range}) gives
\begin{equation*}
    |\gamma'|
    \leq
    \frac{\mu_2-\mu_1}{b-a}
    <
    L.
\end{equation*}
In this case we simply take $\widehat\gamma:=\gamma$.

Assume therefore that
\begin{equation*}
    \int_a^b|\gamma''(x)|\,dx>0.
\end{equation*}
Extend $\gamma_*$ constantly outside $[a,b]$, and denote the resulting function by $\widetilde\gamma_*$. Let
\begin{equation*}
    (\gamma_*)_\delta:=\rho_\delta*\widetilde\gamma_*,
\end{equation*}
where $\rho_\delta$ is a standard nonnegative mollifier. By~(\ref{est:inf-convolution-range}) and~(\ref{est:inf-convolution-slope}),
\begin{equation*}
    \mu_1\leq(\gamma_*)_\delta(x)\leq\mu_2
    \qquad\quad\text{and}\quad\qquad
    |(\gamma_*)_\delta'(x)|\leq L
\end{equation*}
for every $x\in[a,b]$.

The constant extension adds the two endpoint jumps of the first derivative, with sizes $|\gamma_*'(a+)|$ and $|\gamma_*'(b-)|$. Therefore, using~(\ref{est:inf-convolution-variation}) and the fact that convolution does not increase total variation, we obtain
\begin{equation*}
    \int_a^b|(\gamma_*)_\delta''(x)|\,dx
    \leq
    \operatorname{Var}(\gamma_*';(a,b))
    +
    |\gamma_*'(a+)|
    +
    |\gamma_*'(b-)|
    \leq
    \int_a^b|\gamma''(x)|\,dx
    +
    2L.
\end{equation*}

Since $(\gamma_*)_\delta\to\gamma_*$ in $L^1((a,b))$ as $\delta\to0^+$, we may choose $\delta$ sufficiently small that
\begin{equation*}
    \int_a^b|(\gamma_*)_\delta(x)-\gamma_*(x)|\,dx
    \leq
    \frac1{\lambda^2}
    \int_a^b|\gamma''(x)|\,dx.
\end{equation*}
Combining this estimate with~(\ref{est:inf-convolution-fidelity}) and recalling that $L=\ep\lambda$, we obtain
\begin{equation*}
    \int_a^b|(\gamma_*)_\delta(x)-\gamma(x)|\,dx
    \leq
    \left(
        1+\frac{(\mu_2-\mu_1)^2}{\ep^2}
    \right)
    \frac1{\lambda^2}
    \int_a^b|\gamma''(x)|\,dx.
\end{equation*}

The conclusion follows by setting $\widehat\gamma:=(\gamma_*)_\delta$.
\end{proof}


\subsubsection{Second order: proof of Proposition~\ref{prop:M2-constrained}}

We first consider $\lambda$ sufficiently large. Choose $\gamma\in C^2([a,b])$ such that
\begin{equation*}
    \F_2((a,b),g,\lambda,\gamma)
    \leq
    \M_2((a,b),g,\lambda)
    +
    \frac1\lambda.
\end{equation*}

By Lemma~\ref{lemma:range-reduction}, after replacing $\gamma$ with another competitor, still denoted by $\gamma$, we may assume that
\begin{equation*}
    \mu_1\leq\gamma(x)\leq\mu_2
    \qquad
    \forall x\in[a,b],
\end{equation*}
and
\begin{equation}
    \F_2((a,b),g,\lambda,\gamma)
    \leq
    C\left(
        \M_2((a,b),g,\lambda)
        +
        \frac1\lambda
    \right).
    \label{est:M2-after-range}
\end{equation}

We now apply Lemma~\ref{lemma:first-derivative-reduction} with $\ep=\ep_2$. For $\lambda$ sufficiently large, condition~(\ref{hp:first-derivative-lambda}) is satisfied, and hence there exists $\widehat\gamma\in C^2([a,b])$ such that
\begin{equation*}
    \mu_1\leq\widehat\gamma(x)\leq\mu_2
    \qquad\quad\text{and}\quad\qquad
    |\widehat\gamma\,'(x)|\leq\ep_2\lambda
\end{equation*}
for every $x\in[a,b]$. Moreover, by~(\ref{est:first-derivative-second}) and~(\ref{est:first-derivative-fidelity}),
\begin{equation*}
    \int_a^b|\widehat\gamma\,''(x)|\,dx
    \leq
    \int_a^b|\gamma''(x)|\,dx
    +
    2\ep_2\lambda,
\end{equation*}
and
\begin{equation*}
    \int_a^b|\widehat\gamma(x)-\gamma(x)|\,dx
    \leq
    \left(
        1+\frac{(\mu_2-\mu_1)^2}{\ep_2^2}
    \right)
    \frac1{\lambda^2}
    \int_a^b|\gamma''(x)|\,dx.
\end{equation*}

Therefore $\widehat\gamma$ is admissible for the constrained problem. By the triangle inequality,
\begin{eqnarray*}
    \F_2((a,b),g,\lambda,\widehat\gamma)
    & \leq &
    \F_2((a,b),g,\lambda,\gamma)
    +
    2\ep_2
    +
    \left(
        1+\frac{(\mu_2-\mu_1)^2}{\ep_2^2}
    \right)
    \frac1\lambda
    \int_a^b|\gamma''(x)|\,dx
    \\
    & \leq &
    C\left(
        \F_2((a,b),g,\lambda,\gamma)
        +
        1
    \right).
\end{eqnarray*}
Together with~(\ref{est:M2-after-range}), this yields
\begin{equation*}
    \M_2^*((a,b),g,\lambda)
    \leq
    C\left(
        \M_2((a,b),g,\lambda)
        +
        1
    \right)
\end{equation*}
for all sufficiently large $\lambda$.

For the remaining bounded values of $\lambda$, the same estimate follows by using any constant function with value in $[\mu_1,\mu_2]$ as an admissible competitor and enlarging $C$ if necessary. This completes the proof.
\qed


\setcounter{equation}{0}
\section{Existence through absolute integrability}
\label{sec:var2derloss}

We now connect the variational problems introduced in the previous section with the growth function $\G(c,\lambda,T)$ introduced in Section~\ref{sec:background}. Since this function is the logarithm of the maximal amplification of the standard energy, the exponential estimates obtained through the approximate energies translate directly into upper bounds in terms of the corresponding variational minimum values.

\begin{thm}[Variational bounds for the growth function]
\label{thm:G-vs-Mk}

Let $T$, $\nu_1$ and $\nu_2$ be positive real numbers with $\nu_1\leq\nu_2$. Let $c:(0,T)\to\mathbb R$ be a measurable function satisfying the strict hyperbolicity assumption~(\ref{hp:sh}).

Then, for every positive integer $k$, there exists a constant $C_k>0$ such that
\begin{equation}
    \G(c,\lambda,T)
    \leq
    C_k\left(
        1+\M_k^*\left((0,T),\frac1c,\lambda\right)
    \right)
    \qquad
    \forall\lambda>0.
    \label{est:G-vs-Mk-star}
\end{equation}
The constant $C_k$ depends only on $k$, $T$, $\nu_1$ and $\nu_2$.

\end{thm}

\begin{proof}

For $0<\lambda<1$, differentiating the standard energy gives
\begin{equation*}
    |E_\lambda'(t)|
    \leq
    \lambda|1-c(t)^2|E_\lambda(t)
    \leq
    (1+\nu_2^2)E_\lambda(t),
\end{equation*}
and hence
\begin{equation*}
    \G(c,\lambda,T)
    \leq
    (1+\nu_2^2)T.
\end{equation*}
Thus~(\ref{est:G-vs-Mk-star}) holds in this range after increasing $C_k$ if necessary.

Let now $\lambda\geq1$, and let $\gamma$ be any admissible competitor for $\M_k^*((0,T),1/c,\lambda)$. By Theorem~\ref{thm:approximate-cascade}, the corresponding approximate energy is uniformly equivalent to the standard energy, and its growth is bounded by
\begin{equation*}
    E_\lambda(t)
    \leq
    C E_\lambda(0)
    \exp\left(
        C\left[
            \frac1{\lambda^{k-1}}
            +
            \F_k\left((0,T),\frac1c,\lambda,\gamma\right)
        \right]
    \right)
\end{equation*}
for every $t\in[0,T]$, where $C$ depends only on $k$, $T$, $\nu_1$ and $\nu_2$. Since $\lambda\geq1$, it follows that
\begin{equation*}
    \G(c,\lambda,T)
    \leq
    C\left(
        1+\F_k\left((0,T),\frac1c,\lambda,\gamma\right)
    \right).
\end{equation*}
Taking the infimum over all admissible competitors $\gamma$ gives~(\ref{est:G-vs-Mk-star}).
\end{proof}

The main consequence of the variational theory developed in the previous section is that, although the approximate-energy construction is naturally formulated in terms of constrained problems of arbitrary order, the resulting growth estimate already saturates at second order. In particular, all the higher-order machinery can be reduced, at the level of general growth bounds, to the unconstrained second-order variational problem.

\begin{cor}[Saturation at second order]
\label{cor:second-order-saturation}

Under the assumptions of Theorem~\ref{thm:G-vs-Mk}, there exists a constant $C>0$ such that
\begin{equation*}
    \G(c,\lambda,T)
    \leq
    C\left(
        1+\M_2\left((0,T),\frac1c,\lambda\right)
    \right)
    \qquad
    \forall\lambda>0.
\end{equation*}

\end{cor}

\begin{proof}

Apply Theorem~\ref{thm:G-vs-Mk} with $k=2$ and then Proposition~\ref{prop:M2-constrained}.
\end{proof}

In terms of the spaces introduced in the previous section, Corollary~\ref{cor:second-order-saturation} immediately implies that, for every nondecreasing function
$\varphi:(0,+\infty)\to[1,+\infty)$,
\begin{equation*}
    \frac1c\in\Sp_2^\varphi((0,T))
    \qquad\Longrightarrow\qquad
    \G(c,\lambda,T)=O(\varphi(\lambda))
    \quad\text{as }\lambda\to+\infty.
\end{equation*}
In particular,
\begin{equation*}
    \frac1c\in\Sp_2((0,T))
    \qquad\Longrightarrow\qquad
    \G(c,\lambda,T)=O(1).
\end{equation*}

Since $\Sp_k^\varphi((0,T))=\Sp_2^\varphi((0,T))$ for every $k\geq2$, all higher-order unconstrained variational problems yield the same general growth bounds for the evolution problem. Thus second order is the natural saturation level for growth estimates obtained through absolute integrability, and membership of $1/c$ in $\Sp_2((0,T))$ provides a purely variational sufficient condition for well-posedness without derivative loss.

We now apply these variational bounds to several regularity assumptions on the propagation speed. In all the examples below we set
\begin{equation*}
    g:=\frac1c.
\end{equation*}
Under strict hyperbolicity, the regularity properties considered below are transferred from $c$ to $g$, up to harmless multiplicative constants. The general strategy is to construct a suitable competitor for either the first- or the second-order variational problem and then convert the resulting variational bound into an estimate for the growth function.  

In the constructions below we shall repeatedly use, without further comment, the relaxed formulations of the variational problems described in the last item of Remark~\ref{rmk:Mk-basic}. Since only the high-frequency behavior is relevant for derivative loss, throughout the examples below $\lambda$ is understood to be sufficiently large whenever this is required by the construction. Finally, the letter $C$ denotes a positive constant whose value may change from line to line and which is independent of $\lambda$ and of the auxiliary parameters introduced in the constructions, unless otherwise specified.

\subsection{Examples based on the first-order variational problem}

\begin{ex}[Modulus of continuity]
\label{ex:modulus-continuity}
\begin{em}

Assume that $c$ is continuous on $[0,T]$ with modulus of continuity $\omega$, namely
\begin{equation*}
    |c(t)-c(s)|
    \leq
    \omega(|t-s|)
    \qquad
    \forall s,t\in[0,T].
\end{equation*}
Then $g=1/c$ has the same modulus of continuity up to a multiplicative constant depending only on the strict hyperbolicity constants.

Let $\rho$ be a standard nonnegative mollifier and, after extending $g$ constantly outside $[0,T]$, set
\begin{equation*}
    \gamma_\lambda:=\rho_{1/\lambda}*g.
\end{equation*}
The standard estimates for mollification give
\begin{equation*}
    |\gamma_\lambda-g|
    \leq
    C\omega(1/\lambda)
    \qquad\quad\text{and}\quad\qquad
    |\gamma_\lambda'|
    \leq
    C\lambda\,\omega(1/\lambda).
\end{equation*}
Therefore
\begin{equation*}
    \F_1((0,T),g,\lambda,\gamma_\lambda)
    \leq
    C\lambda\,\omega(1/\lambda).
\end{equation*}
Since convolution with a nonnegative kernel preserves the range of $g$, the function $\gamma_\lambda$ is admissible for the first-order constrained problem. Applying Theorem~\ref{thm:G-vs-Mk} with $k=1$, we obtain
\begin{equation}
    \G(c,\lambda,T)
    \leq
    C\left(
        1+\lambda\,\omega(1/\lambda)
    \right).
    \label{est:G-modulus}
\end{equation}

This estimate is the natural general form of the classical regularization argument originating in~\cite{1979-DGCS}. The early literature was mostly formulated in terms of specific moduli of continuity and the corresponding regularity regimes. The quantity $\lambda\,\omega(1/\lambda)$ appears explicitly in the general formulation of~\cite{GhisiGobbino09}; see also~\cite{CicognaniColombini06} for subsequent developments concerning general moduli of continuity and derivative loss.

Two classical cases are immediately recovered from~(\ref{est:G-modulus}). In the Lipschitz case, when $\omega(\sigma)\leq C\sigma$, one has $\lambda\,\omega(1/\lambda)=O(1)$, and therefore there is no derivative loss. In the log-Lipschitz case, when $\omega(\sigma)\leq C\sigma(1+|\log\sigma|)$ for $\sigma$ close to zero, one has $\lambda\,\omega(1/\lambda)=O(\log\lambda)$, and therefore the derivative loss is finite.

Thus the variational formulation packages the classical regularization argument into the single quantity $\lambda\,\omega(1/\lambda)$, independently of the particular modulus under consideration.

\end{em}
\end{ex}

\begin{ex}[Fractional Sobolev regularity]
\label{ex:fractional-regularity}
\begin{em}

Assume that $c\in W^{s,1}((0,T))$ for some $s\in(0,1)$. Then $g=1/c$ belongs to the same space and
\begin{equation*}
    [g]_{W^{s,1}}
    \leq
    C[c]_{W^{s,1}},
\end{equation*}
where $C$ depends only on the strict hyperbolicity constants.

After extending $g$ outside $(0,T)$ by means of a bounded extension operator for $W^{s,1}$, let
\begin{equation*}
    \gamma_\lambda:=\rho_{1/\lambda}*g.
\end{equation*}
The standard fractional estimates for mollification yield
\begin{equation*}
    \|\gamma_\lambda-g\|_{L^1((0,T))}
    \leq
    C\lambda^{-s}[g]_{W^{s,1}}
    \qquad\quad\text{and}\quad\qquad
    \|\gamma_\lambda'\|_{L^1((0,T))}
    \leq
    C\lambda^{1-s}[g]_{W^{s,1}}.
\end{equation*}
Consequently,
\begin{equation*}
    \M_1((0,T),g,\lambda)
    \leq
    C\lambda^{1-s}[g]_{W^{s,1}}.
\end{equation*}
By Proposition~\ref{prop:M1-constrained} and Theorem~\ref{thm:G-vs-Mk} with $k=1$, we obtain
\begin{equation*}
    \G(c,\lambda,T)
    \leq
    C\left(
        1+\lambda^{1-s}[c]_{W^{s,1}}
    \right).
\end{equation*}

Thus fractional Sobolev regularity of order $s$ leads naturally to a growth bound of order $\lambda^{1-s}$. To the best of our knowledge, this fractional Sobolev criterion does not seem to have appeared explicitly in the classical literature for time-dependent propagation speeds.

\end{em}
\end{ex}

\begin{ex}[Control of the first derivative away from the origin]
\label{ex:first-derivative-away-zero}
\begin{em}

Assume that $c\in C^1((0,T])$, with no uniform control on $c'$ near the origin. Let
\begin{equation*}
    \delta:=\frac1\lambda,
\end{equation*}
and consider the function
\begin{equation*}
    \gamma_\lambda(t):=
    \begin{cases}
        g(\delta), & 0\leq t\leq\delta,\\
        g(t), & \delta<t\leq T.
    \end{cases}
\end{equation*}
This function is continuous and belongs to $BV((0,T))$.

Its total variation is
\begin{equation*}
    \operatorname{Var}(\gamma_\lambda;(0,T))
    =
    \int_{1/\lambda}^T|g'(t)|\,dt.
\end{equation*}
On the other hand, the fidelity cost is confined to $(0,1/\lambda)$. Since strict hyperbolicity gives
\begin{equation*}
    \frac1{\nu_2}\leq g(t)\leq\frac1{\nu_1},
\end{equation*}
we have
\begin{equation*}
    \lambda\int_0^{1/\lambda}
    |\gamma_\lambda(t)-g(t)|\,dt
    \leq
    \frac1{\nu_1}-\frac1{\nu_2}.
\end{equation*}
It follows that
\begin{equation*}
    \M_1((0,T),g,\lambda)
    \leq
    \frac1{\nu_1}-\frac1{\nu_2}
    +
    \int_{1/\lambda}^T|g'(t)|\,dt.
\end{equation*}
Since
\begin{equation*}
    |g'(t)|
    =
    \frac{|c'(t)|}{c(t)^2}
    \leq
    \frac1{\nu_1^2}|c'(t)|,
\end{equation*}
Proposition~\ref{prop:M1-constrained} and Theorem~\ref{thm:G-vs-Mk} with $k=1$ yield
\begin{equation}
    \G(c,\lambda,T)
    \leq
    C\left(
        1+\int_{1/\lambda}^T|c'(t)|\,dt
    \right).
    \label{est:G-first-derivative}
\end{equation}

This estimate is the variational counterpart of the classical frequency-dependent splitting used for coefficients with a singular first derivative near the origin; see, in particular,~\cite{CDSK2002}. In the classical argument the transition occurs at the scale $t\sim1/\lambda$: an approximate energy is used before this scale, while the standard hyperbolic energy is used afterwards. The competitor above encodes the same two-region mechanism in a single variational object.

Two familiar consequences are immediate from~(\ref{est:G-first-derivative}). If $c'\in L^1((0,T))$, then $\G(c,\lambda,T)=O(1)$ and there is no derivative loss. If instead $|c'(t)|\leq C/t$ near the origin, then $\G(c,\lambda,T)=O(\log\lambda)$ and the derivative loss is finite.

\end{em}
\end{ex}

\begin{ex}[Modulus of continuity and first derivative]
\label{ex:modulus-first-derivative}
\begin{em}

Let us assume that $c$ is $\omega$-continuous on $[0,T]$ and belongs to $C^1((0,T])$. This combines the assumptions of Examples~\ref{ex:modulus-continuity} and~\ref{ex:first-derivative-away-zero}.

Let $\rho$ be a standard mollifier and extend $g$ constantly outside $[0,T]$. Set
\begin{equation*}
    g_\lambda:=\rho_{1/\lambda}*g.
\end{equation*}
Fix $s\in(0,T)$, and define
\begin{equation*}
    \gamma_{\lambda,s}(t)
    :=
    \begin{cases}
        g_\lambda(t)-g_\lambda(s)+g(s),
        & 0\leq t\leq s,\\
        g(t),
        & s<t\leq T.
    \end{cases}
\end{equation*}
The function $\gamma_{\lambda,s}$ is continuous and piecewise of class $C^1$.

The standard estimates for mollification give
\begin{equation*}
    |g_\lambda(t)-g(t)|
    \leq
    C\omega(1/\lambda)
    \qquad
    \forall t\in[0,T],
\end{equation*}
and
\begin{equation*}
    |g_\lambda'(t)|
    \leq
    C\lambda\,\omega(1/\lambda)
    \qquad
    \forall t\in[0,T].
\end{equation*}
Therefore, for every $t\in[0,s]$,
\begin{equation*}
    |\gamma_{\lambda,s}(t)-g(t)|
    \leq
    |g_\lambda(t)-g(t)|
    +
    |g_\lambda(s)-g(s)|
    \leq
    C\omega(1/\lambda),
\end{equation*}
and hence
\begin{equation*}
    \lambda\int_0^T
    |\gamma_{\lambda,s}(t)-g(t)|\,dt
    \leq
    C\lambda\,\omega(1/\lambda)\,s.
\end{equation*}

On the other hand,
\begin{equation*}
    \operatorname{Var}(\gamma_{\lambda,s};(0,T))
    \leq
    C\lambda\,\omega(1/\lambda)\,s
    +
    \int_s^T|g'(t)|\,dt.
\end{equation*}
It follows that
\begin{equation*}
    \M_1((0,T),g,\lambda)
    \leq
    C\left(
        \lambda\,\omega(1/\lambda)\,s
        +
        \int_s^T|g'(t)|\,dt
    \right).
\end{equation*}
Taking the infimum with respect to $s$, and using
\begin{equation*}
    |g'(t)|
    \leq
    \frac1{\nu_1^2}|c'(t)|,
\end{equation*}
Proposition~\ref{prop:M1-constrained} and Theorem~\ref{thm:G-vs-Mk} with $k=1$ yield
\begin{equation*}
    \G(c,\lambda,T)
    \leq
    C\left[
        1+
        \inf_{0<s<T}
        \left\{
            \lambda\,\omega(1/\lambda)\,s
            +
            \int_s^T|c'(t)|\,dt
        \right\}
    \right].
\end{equation*}

This estimate belongs to a line of results in which a modulus of continuity is combined with information on a possibly singular first derivative. Several specific regimes were treated in~\cite{CicognaniColombini2003,CDSK2002,CDSR2003,KinoshitaReissig2005,DelSantoKinoshitaReissig2007}. A general formulation of the compensation between the two assumptions, based on the optimization of the transition scale, was obtained in~\cite{gg:OptDerLoss}. In the present variational framework, the same mechanism is encoded directly by the family of competitors $\gamma_{\lambda,s}$.

\end{em}
\end{ex}

\subsection{Examples based on the second-order variational problem}

We now consider examples in which the second-order variational problem provides information that is not visible at first order.

\begin{ex}[Control of the second derivative away from the origin]
\label{ex:second-derivative-away-zero}
\begin{em}

Assume that $c\in C^2((0,T])$. Fix $s\in(0,T)$, and consider the continuous function
\begin{equation*}
    \gamma_s(t):=
    \begin{cases}
        g(s), & 0\leq t\leq s,\\
        g(t), & s<t\leq T.
    \end{cases}
\end{equation*}
ts weak derivative is equal to $0$ on $(0,s)$ and coincides with $g'$ on $(s,T)$, while globally it belongs to $BV((0,T))$. More precisely,
\begin{equation*}
    \operatorname{Var}(\gamma_s';(0,T))
    =
    |g'(s)|
    +
    \int_s^T|g''(t)|\,dt.
\end{equation*}

Using $\gamma_s$ as a competitor in the relaxed second-order variational problem, we obtain
\begin{equation*}
    \M_2((0,T),g,\lambda)
    \leq
    \frac1\lambda
    \left(
        |g'(s)|
        +
        \int_s^T|g''(t)|\,dt
    \right)
    +
    \lambda\int_0^s|g(s)-g(t)|\,dt.
\end{equation*}
Strict hyperbolicity gives
\begin{equation*}
    \lambda\int_0^s|g(s)-g(t)|\,dt
    \leq
    \left(
        \frac1{\nu_1}-\frac1{\nu_2}
    \right)\lambda s.
\end{equation*}
Moreover,
\begin{equation*}
    |g'(s)|
    \leq
    |g'(T)|
    +
    \int_s^T|g''(t)|\,dt.
\end{equation*}
Therefore
\begin{equation*}
    \M_2((0,T),g,\lambda)
    \leq
    C\left[
        1+
        \lambda s
        +
        \frac1\lambda
        \int_s^T|g''(t)|\,dt
    \right],
\end{equation*}
where $C$ may also depend on $|g'(T)|$.

Taking the infimum with respect to $s$ and applying Corollary~\ref{cor:second-order-saturation}, we obtain
\begin{equation*}
    \G(c,\lambda,T)
    \leq
    C\left[
        1+
        \inf_{0<s<T}
        \left\{
            \lambda s
            +
            \frac1\lambda
            \int_s^T|g''(t)|\,dt
        \right\}
    \right].
\end{equation*}

Since
\begin{equation*}
    g''=-\frac{c''}{c^2}+2\frac{(c')^2}{c^3},
\end{equation*}
strict hyperbolicity yields
\begin{equation}
    \G(c,\lambda,T)
    \leq
    C\left[
        1+
        \inf_{0<s<T}
        \left\{
            \lambda s
            +
            \frac1\lambda
            \int_s^T
            \left(
                |c''(t)|+|c'(t)|^2
            \right)\,dt
        \right\}
    \right].
    \label{est:G-second-derivative}
\end{equation}

This estimate provides a variational formulation of a classical line of results based on the simultaneous control of the first and second derivatives of the propagation speed. A first result of this type was obtained by T.~Yamazaki~\cite{Yamazaki1990}, who proved well-posedness without derivative loss under the critical assumptions
\begin{equation*}
    |c'(t)|\leq\frac{C}{t}
    \qquad\quad\text{and}\quad\qquad
    |c''(t)|\leq\frac{C}{t^2}.
\end{equation*}
Indeed, choosing $s=1/\lambda$ in~(\ref{est:G-second-derivative}) gives $\G(c,\lambda,T)=O(1)$.

Logarithmically stronger singularities were considered in~\cite{CDSR2003}; see also~\cite{Hirosawa2003Loss,Hirosawa2003Cauchy,DelSantoKinoshitaReissig2007}. For instance, under assumptions of the form
\begin{equation*}
    |c'(t)|
    \leq
    C\frac{|\log t|}{t}
    \qquad\quad\text{and}\quad\qquad
    |c''(t)|
    \leq
    C\left(
        \frac{|\log t|}{t}
    \right)^2
\end{equation*}
near the origin, the choice
\begin{equation*}
    s\sim\frac{\log\lambda}{\lambda}
\end{equation*}
in~(\ref{est:G-second-derivative}) yields $\G(c,\lambda,T)=O(\log\lambda)$ and hence a finite derivative loss.

Thus, in the present variational framework, the choice of the transition scale is encoded directly by the minimization with respect to $s$.

A substantially more flexible mechanism, based on a three-region decomposition with two transition scales and involving the simultaneous control of the first two derivatives, was developed in~\cite{GhisiGobbino2021}. We shall return to this construction below.

\end{em}
\end{ex}

\begin{ex}[Modulus of continuity and second derivative]
\label{ex:modulus-second-derivative}
\begin{em}

Assume that $c$ is $\omega$-continuous on $[0,T]$ and belongs to $C^2((0,T])$. We combine the regularization argument of Example~\ref{ex:modulus-continuity} with the second-order construction of Example~\ref{ex:second-derivative-away-zero}.

Let $\rho$ be a standard mollifier, extend $g$ constantly outside $[0,T]$, set
\begin{equation*}
    g_\lambda:=\rho_{1/\lambda}*g,
\end{equation*}
and, for $s\in(0,T)$, define
\begin{equation*}
    \gamma_{\lambda,s}(t)
    :=
    \begin{cases}
        g_\lambda(t)-g_\lambda(s)+g(s),
        & 0\leq t\leq s,\\
        g(t),
        & s<t\leq T.
    \end{cases}
\end{equation*}
Arguing as in the previous examples, and using the standard mollification estimates
\begin{equation*}
    |g_\lambda-g|
    \leq
    C\omega(1/\lambda),
    \qquad\quad\text{and}\quad\qquad
    |g_\lambda^{(j)}|
    \leq
    C\lambda^j\omega(1/\lambda)
    \quad (j=1,2),
\end{equation*}
we obtain
\begin{equation*}
    \M_2((0,T),g,\lambda)
    \leq
    C\left[
        1+
        \lambda\,\omega(1/\lambda)\,s
        +
        \frac1\lambda
        \int_s^T|g''(t)|\,dt
    \right].
\end{equation*}
Taking the infimum with respect to $s$ and applying Corollary~\ref{cor:second-order-saturation}, we conclude that
\begin{equation*}
    \G(c,\lambda,T)
    \leq
    C\left[
        1+
        \inf_{0<s<T}
        \left\{
            \lambda\,\omega(1/\lambda)\,s
            +
            \frac1\lambda
            \int_s^T
            \left(
                |c''(t)|+|c'(t)|^2
            \right)\,dt
        \right\}
    \right].
\end{equation*}

Thus the modulus of continuity improves the cost of the initial region from $\lambda s$ to $\lambda\,\omega(1/\lambda)\,s$, while the second-order cost on the remaining interval is unchanged.

\end{em}
\end{ex}

\subsection{The three-region mechanism}
\label{subsec:three-region}

The previous examples use only two regions: a neighborhood of the origin where the coefficient is replaced or regularized, and a region where its actual derivatives are exploited. A more flexible construction is obtained by inserting an intermediate region in which the coefficient is approximated at the wavelength scale. This leads to a variational counterpart of the classical three-region strategy.

\begin{prop}[A three-region estimate]
\label{prop:three-region-estimate}

Let $T>0$ and let
\begin{equation*}
    g\in L^\infty((0,T))\cap C^2((0,T]).
\end{equation*}
Then there exists a constant $C>0$, depending only on $T$, $\|g\|_{L^\infty((0,T))}$ and $|g'(T)|$, such that, for every $\lambda>0$, one has
\begin{equation}
    \M_2((0,T),g,\lambda)
    \leq
    C\left[
        1+
        \inf_{0\leq a\leq b\leq T}
        \left\{
            \lambda a
            +
            \int_a^b|g'(t)|\,dt
            +
            \frac1\lambda\int_b^T|g''(t)|\,dt
        \right\}
    \right].
    \label{est:three-region}
\end{equation}

\end{prop}

\begin{proof}

We prove that, for every $\lambda>0$ and every $0\leq a\leq b\leq T$,
\begin{equation}
    \M_2((0,T),g,\lambda)
    \leq
    C\left[
        1+\lambda a
        +
        \int_a^b|g'(t)|\,dt
        +
        \frac1\lambda\int_b^T|g''(t)|\,dt
    \right].
    \label{est:three-region-fixed-ab}
\end{equation}
If one of the integrals on the right-hand side is infinite, there is nothing to prove. In particular, when an endpoint is equal to zero, the corresponding finite integral guarantees the existence of the limits needed in the constructions below.

We distinguish three cases.

\paragraph{\textmd{\textit{The case $0<\lambda<1$.}}}
The constant competitor $v\equiv0$ gives
\begin{equation*}
    \F_2((0,T),g,\lambda,v)
    =
    \lambda\int_0^T|g(t)|\,dt
    \leq
    T\|g\|_{L^\infty((0,T))},
\end{equation*}
and therefore~(\ref{est:three-region-fixed-ab}) follows.

\paragraph{\textmd{\textit{The case $\lambda\geq1$ and $\lambda(b-a)<1$.}}}
We use the same freezing construction as in Example~\ref{ex:second-derivative-away-zero}, now with $s=b$. Namely, we take $v$ equal to $g(b)$ on $[0,b]$ and to $g$ on $[b,T]$. Arguing exactly as there, and using the boundedness of $g$, the relaxed formulation gives
\begin{equation*}
    \M_2((0,T),g,\lambda)
    \leq
    2\|g\|_{L^\infty((0,T))}\lambda b
    +
    \frac{|g'(T)|}{\lambda}
    +
    \frac2\lambda\int_b^T|g''(t)|\,dt.
\end{equation*}
Since
\begin{equation*}
    \lambda b
    =
    \lambda a+\lambda(b-a)
    <
    1+\lambda a,
\end{equation*}
and $\lambda\geq1$, estimate~(\ref{est:three-region-fixed-ab}) follows.

\paragraph{\textmd{\textit{The case $\lambda\geq1$ and $\lambda(b-a)\geq1$.}}}
Set
\begin{equation*}
    N:=\left\lceil\lambda(b-a)\right\rceil,
    \qquad\qquad
    h:=\frac{b-a}{N},
\end{equation*}
and then
\begin{equation*}
    x_j:=a+jh
    \quad
    (j=0,\ldots,N),
    \qquad\qquad
    m_j:=\frac{g(x_j)-g(x_{j-1})}{h}
    \qquad
    (j=1,\ldots,N).
\end{equation*}

We choose $v$ to be equal to $g(a)$ on $[0,a]$, to the piecewise affine interpolant of $g$ at the nodes $x_0,\ldots,x_N$ on $[a,b]$, and to $g$ on $[b,T]$.

Let us first estimate the fidelity term. A standard estimate for affine interpolation gives
\begin{equation*}
    \int_{x_{j-1}}^{x_j}|v(t)-g(t)|\,dt
    \leq
    \frac h2
    \int_{x_{j-1}}^{x_j}|g'(t)|\,dt.
\end{equation*}
Since $\lambda h\leq1$, it follows that
\begin{equation*}
    \lambda\int_0^T|v(t)-g(t)|\,dt
    \leq
    2\|g\|_{L^\infty((0,T))}\lambda a
    +
    \frac12\int_a^b|g'(t)|\,dt.
\end{equation*}

We now estimate the second-order regularity cost. The function $v'$ has bounded variation, and more precisely vanishes on $(0,a)$, is equal to $m_j$ on $(x_{j-1},x_j)$, and coincides with $g'$ on $(b,T)$. Its total variation therefore satisfies
\begin{eqnarray}
    \operatorname{Var}(v';(0,T))
    & \leq &
    |m_1|+\sum_{j=1}^{N-1}|m_{j+1}-m_j|
    +|g'(b)-m_N|+\int_b^T|g''(t)|\,dt
    \nonumber
    \\
    & \leq &
    2\sum_{j=1}^{N}|m_j|
    +
    |g'(b)|
    +
    \int_b^T|g''(t)|\,dt.
    \label{est:var-v'}
\end{eqnarray}
We observe that the first inequality is an equality whenever $0<a<b<T$; when $a=0$ or $b=T$, the corresponding boundary jump is simply absent. Now we use as usual that
\begin{equation*}
    |g'(b)|
    \leq
    |g'(T)|+\int_b^T|g''(t)|\,dt,
\end{equation*}
and we observe that
\begin{equation*}
    \sum_{j=1}^N|m_j|
    \leq
    \frac1h\int_a^b|g'(t)|\,dt.    
\end{equation*}

Finally, from $\lambda(b-a)\geq 1$ we deduce that $1/(\lambda h)\leq 2$. Plugging all these estimates into (\ref{est:var-v'}) we conclude that
\begin{equation*}
    \frac1\lambda\operatorname{Var}(v';(0,T))
    \leq
    4\int_a^b|g'(t)|\,dt
    +
    \frac{|g'(T)|}{\lambda}
    +
    \frac2\lambda\int_b^T|g''(t)|\,dt.
\end{equation*}
Recalling that $\lambda\geq1$, we obtain~(\ref{est:three-region-fixed-ab}) also in this case.

Thus~(\ref{est:three-region-fixed-ab}) holds for every $\lambda>0$ and every $0\leq a\leq b\leq T$. Taking the infimum with respect to $a$ and $b$ yields~(\ref{est:three-region}).
\end{proof}

The three contributions in Proposition~\ref{prop:three-region-estimate} reflect the three regimes of the classical argument. The initial cost $\lambda a$ corresponds to the region where no regularity is exploited, the term with the integral of $|g'|$ to the intermediate region where first-order information is used, and the term with the integral of $|g''|$ to the genuinely second-order regime. In the classical energy method these roles are played, respectively, by the Kovalevskayan, hyperbolic, and Tarama energies. In the present framework the same mechanism is encoded by a single variational competitor.

\begin{ex}[Mixed first- and second-order singularities]
\label{ex:mixed-first-second-singularities}
\begin{em}

Assume that $c\in C^2((0,T])$ satisfies the strict hyperbolicity assumption, and let us assume that there exist two nonincreasing functions $\omega:(0,T]\to(0,+\infty)$ and $\psi:(0,T]\to[0,+\infty)$ such that
\begin{equation*}
    |c'(t)|
    \leq
    \frac{\omega(t)}{t}
    \qquad\text{and}\qquad
    |c''(t)|
    \leq
    \frac{\omega(t)^2}{t^2}\exp(\psi(t))
    \qquad
    \forall t\in(0,T].
\end{equation*}
Set
\begin{equation*}
    H(t):=\omega(t)(1+\psi(t)).
\end{equation*}
Then 
\begin{equation}
    \G(c,\lambda,T)
    \leq
    C\left(1+H(1/\lambda)\right).
    \label{est:mixed-singularities-growth}
\end{equation}

Indeed, setting $g:=1/c$, strict hyperbolicity gives
\begin{equation}
    |g'(t)|
    \leq
    C\frac{\omega(t)}{t},
    \qquad\quad\text{and}\quad\qquad
    |g''(t)|
    \leq
    C\frac{\omega(t)^2}{t^2}\exp(\psi(t)).
    \label{est:g'-g''}
\end{equation}
Let us set
\begin{equation*}
    \omega_\lambda:=\omega(1/\lambda)
    \qquad\quad\text{and}\quad\qquad
    \psi_\lambda:=\psi(1/\lambda),
\end{equation*}
and choose
\begin{equation}
    a_\lambda:=\min\left\{T,\frac{1+\omega_\lambda}{\lambda}\right\}
    \qquad\quad\text{and}\quad\qquad
    b_\lambda:=
    \min\left\{
        T,\,
        a_\lambda\exp(1+\psi_\lambda)
    \right\}.
    \label{defn:al-bl}
\end{equation}

Let us estimate the contribution of the three terms in Proposition~\ref{prop:three-region-estimate}. The first term satisfies
\begin{equation*}
    \lambda a_\lambda
    \leq
    1+\omega_\lambda.
\end{equation*}

As for the second term, we can assume that $a_\lambda<T$, because otherwise there is nothing to estimate. In this case
\begin{equation*}
    a_\lambda
    =
    \frac{1+\omega_\lambda}{\lambda}
    \geq
    \frac1\lambda,
\end{equation*}
and therefore, by~(\ref{est:g'-g''}) and the monotonicity of $\omega$, we deduce that $\omega(a_\lambda)\leq\omega_\lambda$, so that
\begin{equation*}
    \int_{a_\lambda}^{b_\lambda}|g'(t)|\,dt
    \leq
    C\omega_\lambda
    \log\left(\frac{b_\lambda}{a_\lambda}\right)
    \leq
    C\omega_\lambda(1+\psi_\lambda).
\end{equation*}

As for the third term, we can assume that $b_\lambda<T$, because otherwise there is nothing to estimate. In this case also $a_\lambda<T$, and hence~(\ref{defn:al-bl}) gives
\begin{equation*}
    \lambda b_\lambda
    =
    (1+\omega_\lambda)\exp(1+\psi_\lambda).
\end{equation*}
Moreover, since $b_\lambda\geq a_\lambda\geq1/\lambda$, from~(\ref{est:g'-g''}) and the monotonicity of $\omega$ and $\psi$ we obtain
\begin{equation*}
    \frac1\lambda\int_{b_\lambda}^T|g''(t)|\,dt
    \leq
    C\frac{\omega_\lambda^2\exp(\psi_\lambda)}{\lambda}
    \int_{b_\lambda}^T\frac{dt}{t^2}
    \leq
    C\frac{\omega_\lambda^2\exp(\psi_\lambda)}
    {\lambda b_\lambda}
    \leq
    C\frac{\omega_\lambda^2}{1+\omega_\lambda}
    \leq
    C\omega_\lambda.
\end{equation*}

Thus, applying Proposition~\ref{prop:three-region-estimate} with $a=a_\lambda$ and $b=b_\lambda$, we obtain
\begin{equation*}
    \M_2((0,T),g,\lambda)
    \leq
    C\left(1+H(1/\lambda)\right),
\end{equation*}
so that now Corollary~\ref{cor:second-order-saturation} yields~(\ref{est:mixed-singularities-growth}).

For the class of singular propagation speeds considered in~\cite{GhisiGobbino2021}, this variational estimate recovers the three-region mechanism used there and provides directly a quantitative bound for the growth function. In particular, bounded $H$ yields no derivative loss, $H(t)=o(|\log t|)$ yields an arbitrarily small derivative loss, and $H(t)=O(|\log t|)$ yields at most a finite derivative loss.

\end{em}
\end{ex}


\subsection{Zygmund-type assumptions}
\label{subsec:zygmund-difference-quotients}

The second-order variational problem is also naturally related to Zygmund-type regularity. Indeed, the quantities appearing in the definition of Zygmund classes are precisely second-order differences, and therefore fit directly into the framework introduced in Section~\ref{sec:variational}.

Zygmund-type assumptions have a long history in the theory of hyperbolic equations with low-regularity coefficients. In the time-dependent case, a fundamental contribution is due to S.~Tarama~\cite{Tarama2007}, who obtained energy estimates under integral conditions on second-order differences by combining a frequency-dependent regularization with a corrected energy. Related Zygmund and log-Zygmund assumptions were subsequently treated in more general settings in~\cite{ColombiniDelSantoFanelliMetivier2013LogZygmund,ColombiniDelSantoFanelliMetivier2013Zygmund,CDFM2015}.

For a function $g\in L^\infty((0,T))$ and $\sigma\in(0,T/2)$, let us set
\begin{equation*}
    Z_2(g,\sigma)
    :=
    \frac1{\sigma}
    \int_\sigma^{T-\sigma}
    |g(x+\sigma)+g(x-\sigma)-2g(x)|\,dx.
\end{equation*}
Thus an integral generalized Zygmund condition can be written in the form
\begin{equation}
    Z_2(g,\sigma)
    \leq
    K\psi(\sigma)
    \qquad
    \forall \sigma\in(0,\sigma_0]
    \label{hp:psi-zygmund}
\end{equation}
for some positive function $\psi$. At the level of~(\ref{hp:psi-zygmund}), the Zygmund and log-Zygmund scales correspond, respectively, to bounded $\psi$ and to $\psi(\sigma)$ of order $1+|\log\sigma|$.

The connection with the second-order variational problem can be seen by means of a particularly simple regularization. For $\ep\in(0,T/2)$, let
\begin{equation*}
    K_\ep(r)
    :=
    \frac{\max\{\ep-|r|,0\}}{\ep^2},
\end{equation*}
and define $\gamma_\ep$ on $[\ep,T-\ep]$ by
\begin{equation*}
    \gamma_\ep(x)
    :=
    \int_{-\ep}^{\ep}
    K_\ep(r)g(x-r)\,dr.
\end{equation*}

The kernel $K_\ep$ is even, nonnegative, and has integral one. The function $\gamma_\ep$ belongs to $C^1([\ep,T-\ep])$, with absolutely continuous first derivative, and a standard calculation shows that
\begin{equation*}
    \gamma_\ep''(x)
    =
    \frac{
        g(x+\ep)-2g(x)+g(x-\ep)
    }{\ep^2}
    \qquad
    \text{for almost every }x\in(\ep,T-\ep),
\end{equation*}
and in particular
\begin{equation*}
    \int_\ep^{T-\ep}
    |\gamma_\ep''(x)|\,dx
    =
    \frac1\ep Z_2(g,\ep).
\end{equation*}

This is the basic mechanism behind the classical use of Zygmund-type assumptions: second-order differences control the second derivative of a frequency-dependent regularization of the coefficient; see, in particular,~\cite{Tarama2007}.

We now extend $\gamma_\ep$ to the whole interval $[0,T]$ by affine continuation at the endpoints, namely by keeping the same definition on $[\ep,T-\ep]$ and setting
\begin{equation*}
    \gamma_\ep(x)
    :=
    \gamma_\ep(\ep)
    +(x-\ep)\gamma_\ep'(\ep)
    \qquad
    \forall x\in[0,\ep],
\end{equation*}
and
\begin{equation*}
    \gamma_\ep(x)
    :=
    \gamma_\ep(T-\ep)
    +(x-T+\ep)\gamma_\ep'(T-\ep)
    \qquad
    \forall x\in[T-\ep,T].
\end{equation*}
Then $\gamma_\ep\in C^1([0,T])$, its first derivative belongs to $BV((0,T))$, and no additional variation is created at the two junctions. Hence
\begin{equation}
    \operatorname{Var}(\gamma_\ep';(0,T))
    =
    \frac1\ep Z_2(g,\ep).
    \label{est:triangular-variation}
\end{equation}

The same second-order differences also control the fidelity in the interior region. Indeed, by symmetry of the kernel,
\begin{equation*}
    \gamma_\ep(x)-g(x)
    =
    \frac12
    \int_{-\ep}^{\ep}
    K_\ep(r)
    \left(
        g(x+r)+g(x-r)-2g(x)
    \right)\,dr
\end{equation*}
for almost every $x\in[\ep,T-\ep]$. Therefore, after integration with respect to $x$ and reversing the order of integration, the inclusion $[\ep,T-\ep]\subseteq[r,T-r]$ for $0<r<\ep$, together with the definition of $Z_2(g,r)$, gives
\begin{equation}
    \int_\ep^{T-\ep}
    |\gamma_\ep(x)-g(x)|\,dx
    \leq
    \frac1{\ep^2}
    \int_0^\ep
    (\ep-r)r Z_2(g,r)\,dr.
    \label{est:triangular-fidelity}
\end{equation}

It remains to estimate the fidelity on the two boundary intervals. Set
\begin{equation*}
    M:=\|g\|_{L^\infty((0,T))}.
\end{equation*}
Since $K_\ep$ is nonnegative and has integral one,
\begin{equation*}
    |\gamma_\ep(x)|
    \leq
    M
    \qquad
    \forall x\in[\ep,T-\ep],
\end{equation*}
and the derivative satisfies
\begin{equation*}
    |\gamma_\ep'(x)|
    \leq
    \frac{2M}{\ep}
    \qquad
    \forall x\in[\ep,T-\ep].
\end{equation*}
It follows from the affine definition of $\gamma_\ep$ on the two boundary intervals that in these intervals $|\gamma_\ep(x)|\leq 3M$, and therefore
\begin{equation}
    \int_0^\ep
    |\gamma_\ep(x)-g(x)|\,dx
    +
    \int_{T-\ep}^T
    |\gamma_\ep(x)-g(x)|\,dx
    \leq
    8M\ep.
    \label{est:triangular-boundary-fidelity}
\end{equation}

Using the relaxed second-order variational problem, together with estimates~(\ref{est:triangular-variation}), (\ref{est:triangular-fidelity}), and~(\ref{est:triangular-boundary-fidelity}), we obtain
\begin{equation*}
    \M_2((0,T),g,\lambda)
    \leq
    \frac{1}{\lambda\ep}Z_2(g,\ep)
    +
    \frac{\lambda}{\ep^2}
    \int_0^\ep
    (\ep-\sigma)\sigma Z_2(g,\sigma)\,d\sigma
    +
    8M\lambda\ep.
\end{equation*}

Choosing $\ep=1/\lambda$, we obtain
\begin{equation}
    \M_2((0,T),g,\lambda)
    \leq
    8M
    +
    Z_2(g,1/\lambda)
    +
    \lambda^3
    \int_0^{1/\lambda}
    \left(
        \frac1\lambda-\sigma
    \right)
    \sigma Z_2(g,\sigma)\,d\sigma.
    \label{est:M2-triangular}
\end{equation}

This suggests introducing the averaged modulus
\begin{equation}
    \widehat\psi(\sigma)
    :=
    \psi(\sigma)
    +
    \frac1{\sigma^3}
    \int_0^\sigma
    (\sigma-r)r\psi(r)\,dr.
    \label{defn:psi-hat}
\end{equation}
If~(\ref{hp:psi-zygmund}) holds, then~(\ref{est:M2-triangular}) gives
\begin{equation*}
    \M_2((0,T),g,\lambda)
    \leq
    8\|g\|_{L^\infty((0,T))}
    +
    K\widehat\psi(1/\lambda).
\end{equation*}
Consequently, if $g=1/c$, strict hyperbolicity and Corollary~\ref{cor:second-order-saturation} yield
\begin{equation*}
    \G(c,\lambda,T)
    \leq
    C\left(
        1+K\widehat\psi(1/\lambda)
    \right).
\end{equation*}

In the classical energy approach, this second-order information is used to construct suitable corrections of the energy or of the symmetrizer; see, for example,~\cite{Tarama2007,ColombiniDelSantoFanelliMetivier2013LogZygmund,ColombiniDelSantoFanelliMetivier2013Zygmund,CDFM2015}. In the present framework, the same information is converted directly into a competitor for the second-order variational problem.

For the usual Zygmund scales, the averaging in~(\ref{defn:psi-hat}) does not change the order of magnitude. In particular, 
\begin{equation*}
    \psi(\sigma)\equiv1
    \qquad
    \leadsto
    \qquad
    \widehat\psi\text{ bounded}
    \qquad
    \leadsto
    \qquad
    \G(c,\lambda,T)\leq C,
\end{equation*}
which gives well-posedness without derivative loss.

Analogously,
\begin{equation*}
    \psi(\sigma):=1+|\log\sigma|
    \quad
    \leadsto
    \quad
    \widehat\psi(\sigma)
    \leq
    C(1+|\log\sigma|)
    \quad
    \leadsto
    \quad
    \G(c,\lambda,T)
    \leq
    C(1+\log\lambda),
\end{equation*}
which corresponds to a finite derivative loss.

More generally,
\begin{equation*}
    \psi(\sigma)
    :=
    (1+|\log\sigma|)^\alpha
\end{equation*}
yields
\begin{equation*}
    \G(c,\lambda,T)
    \leq
    C(1+\log\lambda)^\alpha.
\end{equation*}
For $0<\alpha<1$ this growth is $o(\log\lambda)$, and therefore corresponds to an arbitrarily small derivative loss. The case $\alpha=1$ gives a finite derivative loss, while $\alpha=0$ gives no derivative loss.

The argument also clarifies the role of the higher-order difference quotients introduced earlier. The classical Zygmund condition is not an additional regularity mechanism external to the variational theory: it is precisely a scale-invariant bound on second-order differences. The triangular regularization above shows that both the fidelity cost and the second-order regularity cost can be controlled by the same family of second differences. In this sense, Zygmund regularity is intrinsically adapted to the second-order variational problem.


\setcounter{equation}{0}
\section{Existence beyond absolute integrability}\label{sec:construction}

The results developed so far fit naturally into the classical philosophy of the theory: the regularity of the propagation speed is used to control suitable energy functionals, and the relevant error terms are estimated through their absolute values. This leads to sufficient conditions that remain closely connected with absolute integrability.

There is, however, another mechanism that is largely invisible to estimates of this kind. The terms appearing in the energy identities are oscillatory, and oscillations may produce cancellations. As in Dirichlet-type arguments, boundedness may persist even when the corresponding absolute values are no longer integrable. From this point of view, absolute integrability marks a limit of these methods rather than necessarily a limit of the equation.

In this section we exploit this mechanism to construct propagation speeds whose regularity lies beyond the scope of the previous criteria, while the associated wave equation remains well posed in Sobolev spaces. The key point is a suitable non-resonance between the oscillations of the propagation speed and those of the solutions.

The following result shows how far this mechanism can go. We construct a smooth propagation speed on $(0,+\infty)$ which does not belong to $W^{s,1}((0,T))$ for any $T>0$ and any $s>1/2$, while the corresponding abstract wave equation remains well posed in the natural energy space for every nonnegative multiplication operator.

\begin{thm}[Well-posedness in Sobolev spaces for an irregular propagation speed]\label{thm:main-irregular}

There exists a function $c:(0,+\infty)\to[1/2,3/2]$ of class $C^\infty$ with the following two properties:
\begin{enumerate}
\renewcommand{\labelenumi}{(\arabic{enumi})}
\item $c\notin W^{s,1}((0,T))$ for every $T>0$ and every $s>1/2$;

\item for every Hilbert space $H$ and every nonnegative self-adjoint operator $A$ on $H$, the abstract wave equation~(\ref{eqn:basic}) is well posed in $D(A^{1/2})\times H$.
\end{enumerate}

\end{thm}

The proof of Theorem~\ref{thm:main-irregular} will be completed in the final subsection, where we introduce a refined exponential construction whose critical threshold is $s=1/2$. The underlying idea is to introduce a propagation speed of the form
\begin{equation}
    c(t):=1+g(t)\sin(\Phi(t))
    \qquad
    \forall t>0,
    \label{defn:c(t)-example}
\end{equation}
where $g:(0,+\infty)\to[0,1/2]$ is a function of class $C^\infty$ that vanishes as $t\to0^+$, and $\Phi:(0,+\infty)\to\re$ is a function of class $C^\infty$ such that $\Phi(t)\to-\infty$ as $t\to0^+$. The factor $\sin(\Phi(t))$ produces increasingly rapid oscillations near the origin. The crucial point is to balance the increasing oscillation rate encoded by $\Phi$ with the decay of $g$ so that $c$ is irregular enough to fail the stated fractional Sobolev regularity, while a suitable non-resonance mechanism still guarantees well-posedness in Sobolev spaces for the associated wave equation.

The remainder of this section is devoted to implementing this strategy. We first identify conditions that guarantee the required non-resonance uniformly with respect to the frequency parameter, and then construct explicit examples satisfying them and verify their fractional Sobolev irregularity.

\subsection{Non-resonance and well-posedness in Sobolev spaces}

The main issue is to control uniformly with respect to $\lambda$ the interaction between the oscillations of the propagation speed and the natural oscillations of the solutions at frequency $\lambda$. The following assumption separates a central region, where the two frequencies may be close, from two non-resonant regions, where oscillatory cancellations can be exploited.

\begin{defn}[Non-resonance assumption]\label{defn:non-res}
\begin{em}

Let $t_0$ be a positive real number, and let $g:(0,t_0]\to[0,+\infty)$ and $\varphi:(0,t_0]\to[0,+\infty)$ be two functions of class $C^1$ with $\varphi(t)\to +\infty$ as $t\to 0^+$.

We say that $g$ and $\varphi$ satisfy the \emph{non-resonance assumption} in $(0,t_0]$ if there exist four nonnegative real numbers $M_1$, $M_2$, $M_3$, $M_4$, and a positive real number $\lambda_0$ such that, for every $\lambda\geq\lambda_0$, there exist two real numbers $\al$ and $\bl$ that satisfy the following properties.
\begin{enumerate}
\renewcommand{\labelenumi}{(\roman{enumi})}

\item The points $\al$ and $\bl$ satisfy 
\begin{equation}
    \frac{1}{\lambda}<\al<\bl<t_0.
    \label{hp:al-bl-position}
\end{equation}

\item The function $\varphi$ satisfies 
\begin{equation}
    \varphi(t)>2\lambda
    \quad
    \forall t\in[1/\lambda,\al]
    \qquad\quad\text{and}\quad\qquad
    \varphi(t)<2\lambda
    \quad
    \forall t\in[\bl,t_0].
    \label{hp:phi-2lambda}
\end{equation}

\item In the central interval we have
\begin{equation}
    \int_{\al}^{\bl}g(t)\varphi(t)\,dt\leq M_1.
    \label{hp:al-bl}
\end{equation}

\item In the left and right interval, the function defined by
\begin{equation*}
    \Hl(t):=\frac{g(t)\varphi(t)}{\varphi(t)-2\lambda}
\end{equation*}
satisfies
\begin{equation}
    |\Hl(t)|\leq M_2
    \qquad
    \forall t\in[1/\lambda,\al]\cup[\bl,t_0],
    \label{hp:Hl-infty}
\end{equation}
\begin{equation}
    \int_{1/\lambda}^{\al}|\Hl'(t)|\,dt
    +
    \int_{\bl}^{t_0}|\Hl'(t)|\,dt
    \leq M_3,
    \label{hp:Hl-BV}
\end{equation}
and
\begin{equation}
    \int_{1/\lambda}^{\al}g(t)\varphi(t)|\Hl(t)|\,dt
    +
    \int_{\bl}^{t_0}g(t)\varphi(t)|\Hl(t)|\,dt
    \leq M_4.
    \label{hp:Hl-g-phi}
\end{equation}

\end{enumerate}

\end{em}    
\end{defn}

The relevance of these conditions is that they control both the boundary terms and the differentiated amplitudes arising from integration by parts in the two non-resonant regions, and therefore provide a frequency-independent energy estimate for the corresponding oscillatory propagation speed.

\begin{prop}[Non-resonance implies well-posedness in Sobolev spaces]\label{prop:non-res}

    Let $g:[0,+\infty)\to [0,1/2]$ be a function of class $C^1$, let $\varphi:(0,+\infty)\to(0,+\infty)$ be a function of class $C^1$ such that $\varphi(t)\to +\infty$ as $t\to 0^+$, and let $t_0>0$ be a positive real number.

    We set
    \begin{equation}
        \Phi(t):=\int_{t_0}^t \varphi(s)\,ds
        \qquad
        \forall t>0,
        \label{defn:Phi}
    \end{equation}
    and we consider the propagation speed given by (\ref{defn:c(t)-example}).

    Let us assume that the restrictions of $g$ and $\varphi$ to $(0,t_0]$ satisfy the non-resonance conditions of Definition~\ref{defn:non-res}, and the integrability conditions
    \begin{equation}
        \int_0^{t_0}\frac{g(t)|\varphi'(t)|}{\varphi(t)}\,dt<+\infty,
        \label{hp:g-phi-int-1}
    \end{equation}
    and
    \begin{equation}
        \int_0^{t_0} g(t)^2\varphi(t)\,dt<+\infty.
        \label{hp:g-phi-int-2}
    \end{equation}

    Then for every $T>0$, there exists a constant $C_T$ such that, for all $\lambda\geq 0$, all solutions to (\ref{eqn:ODE}) satisfy
    \begin{equation*}
        \ul'(t)^2+\lambda^2\ul(t)^2\leq
        C_T\left(\ul'(0)^2+\lambda^2 \ul(0)^2\right)
        \qquad
        \forall t\in[0,T].
    \end{equation*}
    
\end{prop}

\begin{proof}

To begin with, we observe that it is enough to establish this estimate on the fixed interval $[0,t_0]$. Indeed, if $T\leq t_0$, the corresponding estimate follows by restriction, while if $T>t_0$, the estimate on $[0,t_0]$ can be extended to $[0,T]$ by means of the standard hyperbolic energy estimate, since $c$ is of class $C^1$ on $[t_0,T]$.

Our goal is therefore to prove that
\begin{equation}
    \ul'(t)^2+\lambda^2 \ul(t)^2\leq
    C_0\left(\ul'(0)^2+\lambda^2 \ul(0)^2\right)
    \qquad
    \forall t\in[0,t_0]
    \label{th:C1}
\end{equation}
for all solutions to (\ref{eqn:ODE}), where $C_0$ is independent of both $\lambda$ and the solution $\ul$.

We observe also that the case $\lambda=0$ is trivial, since in this case $\ul''=0$ and the energy reduces to $\ul'(t)^2=\ul'(0)^2$. We therefore assume that $\lambda>0$.

\paragraph{\textmd{\textit{The Kovalevskian phase}}}

We first consider the Kovalevskian energy
\begin{equation}
    \Ekov(t):=\ul'(t)^2+\lambda^2\ul(t)^2.
    \label{defn:Ekov}
\end{equation}

By computing the time-derivative of (\ref{defn:Ekov}), and exploiting (\ref{eqn:ODE}), we deduce that
\begin{equation*}
    \Ekov'(t)=2\lambda^2\left(1-c(t)^2\right)\ul(t)\ul'(t)
    \leq
    \lambda|1-c(t)^2|\cdot\Ekov(t).
\end{equation*}

Since $c(t)\in[1/2,3/2]$, integrating this differential inequality yields
\begin{equation}
    \ul'(t)^2+\lambda^2 \ul(t)^2\leq
    \exp\left(\frac{5}{4}\lambda t\right)
    \left(\ul'(0)^2+\lambda^2 \ul(0)^2\right)
    \qquad
    \forall t>0.
    \label{est:Ekov}
\end{equation}

In particular, (\ref{est:Ekov}) provides a uniform estimate on $[0,1/\lambda]$. If $1/\lambda\geq t_0$, this already proves (\ref{th:C1}) by restriction. Moreover, if $0<\lambda<\lambda_0$, the same estimate applied on the whole interval $[0,t_0]$ gives (\ref{th:C1}) with a constant bounded by $\exp(5\lambda_0t_0/4)$. Therefore, in the remainder of the proof we may assume that
\[
    \lambda\geq\lambda_0
    \qquad\quad\text{and}\quad\qquad
    \frac{1}{\lambda}<t_0,
\]
and we only need to extend the estimate from time $1/\lambda$ up to time $t_0$.

\paragraph{\textmd{\textit{Polar coordinates in the hyperbolic phase}}}

In the interval $[1/\lambda, t_0]$, we consider the hyperbolic energy
\begin{equation*}
    \Ehyp(t) := \ul'(t)^2 + \lambda^2 c(t)^2\ul(t)^2,
\end{equation*}
and we observe that it is equivalent to $\Ekov(t)$ since $c(t)$ is bounded between two positive constants. We claim that
\begin{equation}
    \Ehyp(t) \leq C_1 \Ehyp(1/\lambda)
    \qquad
    \forall t\in[1/\lambda,t_0]
    \label{th:Ehyp}
\end{equation}
for a suitable constant $C_1$ independent of both $\lambda$ and the solution $\ul$.  
Once this claim is established, (\ref{th:C1}) follows immediately from (\ref{est:Ekov}) with $t=1/\lambda$ and the equivalence between $\Ekov(t)$ and $\Ehyp(t)$.

In order to prove (\ref{th:Ehyp}), we follow the approach already applied in~\cite{gg:2024-JDE-Resonance,gg:2025-JDE-Model}. We write every nontrivial solution in polar coordinates as
\begin{equation*}
    \ul(t)=\frac{1}{\lambda c(t)}\rhol(t)\cos(\thel(t)),
    \qquad\qquad
    \ul'(t)=-\rhol(t)\sin(\thel(t)),
\end{equation*}
where $\rhol:(0,+\infty)\to[0,+\infty)$ and $\thel:(0,+\infty)\to\re$ are solutions for $t>0$ to the system of ordinary differential equations
\begin{eqnarray}
    \rhol'(t) & = & \frac{c'(t)}{c(t)}\rhol(t)\cos^2(\thel(t)),
    \label{eqn:rho}
    \\
    \thel'(t) & = & \lambda+\lambda g(t)\sin(\Phi(t))-\frac{1}{2}\frac{c'(t)}{c(t)}\sin(2\thel(t)).
    \label{eqn:theta}
\end{eqnarray}

By integrating (\ref{eqn:rho}) in the interval $[1/\lambda,t]$ we deduce that
\begin{equation}
    \rhol(t)=\rhol(1/\lambda)\cdot
    \exp\left(\int_{1/\lambda}^{t}\frac{c'(s)}{c(s)}\cos^2(\thel(s))\,ds\right)
    \qquad
    \forall t\in[1/\lambda,t_0].
    \label{est:rho-1}
\end{equation}

Since $\rhol(t)^2=\Ehyp(t)$, in order to prove (\ref{th:Ehyp}) it is enough to establish a uniform upper bound for the oscillatory integral in (\ref{est:rho-1}), independently of $\lambda$ and of $t\in[1/\lambda,t_0]$. In fact, we shall prove the stronger statement that this integral is uniformly bounded in absolute value.

\paragraph{\textmd{\textit{Two useful estimates}}}

For future use, from (\ref{defn:c(t)-example}) we deduce that
\begin{equation}
    c'(t)=g'(t)\sin(\Phi(t))+g(t)\varphi(t)\cos(\Phi(t))
    \qquad
    \forall t>0,
    \label{eqn:c'}
\end{equation}
and hence
\begin{equation}
    |c'(t)|\leq |g'(t)|+g(t)\varphi(t)
    \qquad
    \forall t>0.
    \label{est:c'}
\end{equation}

Moreover, by integrating (\ref{eqn:theta}) in the interval $[1/\lambda,t]$, we obtain that
\begin{equation}
    \thel(t)=\lambda t+\psil(t)
    \qquad
    \forall t>0,
    \label{defn:psil}
\end{equation}
where
\begin{equation*}
    \psil(t):=\thel(1/\lambda)-1+
    \int_{1/\lambda}^t\left(\lambda g(s)\sin(\Phi(s))-
    \frac{1}{2}\frac{c'(s)}{c(s)}\sin(2\thel(s))\right)\,ds.
\end{equation*}

In particular, the time-derivative of $\psil$ is given by
\begin{equation*}
    \psil'(t)=
    \lambda g(t)\sin(\Phi(t))-\frac{1}{2}\frac{c'(t)}{c(t)}\sin(2\thel(t))
    \qquad
    \forall t>0,
\end{equation*}
and hence
\begin{equation}
    |\psil'(t)|\leq
    \lambda g(t)+|c'(t)|\leq
    \lambda g(t)+|g'(t)|+g(t)\varphi(t).
    \qquad
    \forall t>0.
    \label{est:psil'}
\end{equation}

\paragraph{\textmd{\textit{Preliminary transformations of the oscillatory integral}}}

For every $t\in[1/\lambda,t_0]$, the integral in (\ref{est:rho-1}) can be rewritten as
\begin{equation}
    \int_{1/\lambda}^{t}\frac{c'(s)}{c(s)}\cos^2(\thel(s))\,ds=
    \frac{1}{2}\int_{1/\lambda}^{t}\frac{c'(s)}{c(s)}\,ds+
    \frac{1}{2}\int_{1/\lambda}^{t}\frac{c'(s)}{c(s)}\cos(2\thel(s))\,ds.
    \label{eqn:split}
\end{equation}

The first term in the right-hand side is equal to 
\begin{equation*}
    \frac{1}{2}\log\left(\frac{c(t)}{c(1/\lambda)}\right),
\end{equation*}
and therefore it is bounded independently of $\lambda$ because the propagation speed is bounded between two positive constants. As a consequence, it is enough to show that also the second integral in the right-hand side is bounded independently of $\lambda$ and of the solution $\ul$ to (\ref{eqn:ODE}).

Thanks to (\ref{eqn:c'}), the integral appearing in the second term on the right-hand side of (\ref{eqn:split}) is equal to
\begin{equation}
    \int_{1/\lambda}^{t}\frac{g'(s)}{c(s)}\sin(\Phi(s))\cos(2\thel(s))\,ds+
    \int_{1/\lambda}^{t}\frac{g(s)\varphi(s)}{c(s)}\cos(\Phi(s))\cos(2\thel(s))\,ds.
    \label{eqn:2-int}
\end{equation}

The first integral in (\ref{eqn:2-int}) is uniformly bounded, since $c$ is bounded from below by a positive constant and $g'$ is continuous on $[0,t_0]$. Therefore, in what follows we focus only on the second integral. We express $\thel(t)$ in the form (\ref{defn:psil}) and exploit the trigonometric identity
\begin{eqnarray*}
    2\cos(a)\cos(b+c) & = & 
    \cos(a+b)\cos(c) - \sin(a+b)\sin(c)
    \\[1ex]
    & & + \cos(a-b)\cos(c) + \sin(a-b)\sin(c).
\end{eqnarray*}

In this way, we decompose the second integral in (\ref{eqn:2-int}) as the sum of four terms, two of which are
\begin{eqnarray}
    I_1 & := & \frac{1}{2}\int_{1/\lambda}^{t} \frac{g(s)\varphi(s)}{c(s)} \cos(\Phi(s) + 2\lambda s)\cos(2\psil(s))\,ds,
    \label{defn:I1}
    \\[1ex]
    I_2 & := & \frac{1}{2}\int_{1/\lambda}^{t} \frac{g(s)\varphi(s)}{c(s)} \cos(\Phi(s) - 2\lambda s)\cos(2\psil(s))\,ds,
    \label{defn:I2}
\end{eqnarray}
while the other two have the same structure, with $\cos(\Phi(s)\pm2\lambda s)\cos(2\psil(s))$ replaced by $\sin(\Phi(s)\pm2\lambda s)\sin(2\psil(s))$, up to a change of sign.

We claim that these four oscillatory integrals are uniformly bounded with respect to $\lambda$ and $t\in[1/\lambda,t_0]$. We prove this result in the case of $I_1$ and $I_2$, because the other ones are analogous.

\paragraph{\textmd{\textit{Convergence of $I_1$}}}

The key idea is to multiply and divide the integrand in (\ref{defn:I1}) by $\varphi(s)+2\lambda$. At this point we can integrate by parts, and rewrite (\ref{defn:I1}) as
\begin{eqnarray*}
    2I_1 & = &
    \left[\frac{g(s)\varphi(s)}{c(s)(\varphi(s)+2\lambda)}
    \sin(\Phi(s)+2\lambda s)\cos(2\psil(s))\right]_{s=1/\lambda}^{s=t}
    \\[0.5ex]
    & & 
    -\int_{1/\lambda}^{t}\sin(\Phi(s)+2\lambda s)
    \left[\frac{g(s)\varphi(s)\cos(2\psil(s))}{c(s)(\varphi(s)+2\lambda)}\right]'\,ds.
\end{eqnarray*}

The boundary terms are uniformly bounded, since $g$ is bounded and $c$ is bounded below by a positive constant. Now we show that the last integral is absolutely convergent, with a bound independent of both $\lambda$ and $t$. To this end, it is sufficient to obtain a uniform bound for the integral of the absolute value of the derivative appearing in the integrand. Computing this derivative yields five terms, which can be estimated using inequalities such as
\begin{equation*}
    c(s) \geq \frac{1}{2},
    \qquad\qquad
    0 \leq \frac{\varphi(s)}{\varphi(s)+2\lambda} \leq 1,
    \qquad\qquad
    0 \leq \frac{\lambda}{\varphi(s)+2\lambda}\leq \frac{1}{2}.
\end{equation*}

Specifically, for the three terms arising from differentiating the three factors in the numerator, we obtain that
\begin{equation*}
    \left|\frac{g'(s)\varphi(s)\cos(2\psil(s))}{c(s)(\varphi(s)+2\lambda)}\right|\leq
    2|g'(s)|,    
    \qquad
    \left|\frac{g(s)\varphi'(s)\cos(2\psil(s))}{c(s)(\varphi(s)+2\lambda)}\right|\leq
    \frac{2g(s)|\varphi'(s)|}{\varphi(s)},
\end{equation*}
and, recalling (\ref{est:psil'}),
\begin{align*}
    \left|
    -\frac{2g(s)\varphi(s)\psil'(s)\sin(2\psil(s))}
    {c(s)(\varphi(s)+2\lambda)}
    \right|
    &\leq
    \frac{2g(s)\varphi(s)}
    {c(s)(\varphi(s)+2\lambda)}
    \left(\lambda g(s)+|g'(s)|+g(s)\varphi(s)\right)
    \\[1ex]
    &\leq
    2g(s)^2\varphi(s)
    +4g(s)|g'(s)|
    +4g(s)^2\varphi(s).
\end{align*}

As for the two terms arising from differentiating the two factors in the denominator, from (\ref{est:c'}) we obtain that 
\begin{equation*}
    \left|-\frac{c'(s)}{c(s)^2}\cdot
    \frac{g(s)\varphi(s)\cos(2\psil(s))}{\varphi(s)+2\lambda}\right|\leq
    4g(s)|c'(s)|\leq 
    4g(s)|g'(s)|+4g(s)^2\varphi(s),
\end{equation*}
and
\begin{equation*}
    \left|-\frac{\varphi'(s)}{(\varphi(s)+2\lambda)^2}\cdot
    \frac{g(s)\varphi(s)\cos(2\psil(s))}{c(s)}\right|\leq
    \frac{2g(s)|\varphi'(s)|}{\varphi(s)}.
\end{equation*}

In conclusion, the absolute value of the sum of the five terms is bounded by
\begin{equation*}
    2|g'(s)|
    +\frac{4g(s)|\varphi'(s)|}{\varphi(s)}
    +8g(s)|g'(s)|
    +10g(s)^2\varphi(s),
\end{equation*}
and this sum is integrable because $g$ and $g'$ are bounded on $[0,t_0]$, while the remaining terms are integrable by assumptions (\ref{hp:g-phi-int-1}) and (\ref{hp:g-phi-int-2}).

\paragraph{\textmd{\textit{Convergence of $I_2$}}}

Following the approach used for $I_1$, we would like to multiply and divide the integrand in (\ref{defn:I2}) by $\varphi(s) - 2\lambda$. However, the difficulty here is that this quantity changes sign in the interval $[1/\lambda, t_0]$. This is precisely why we introduced $\al$ and $\bl$ in the non-resonance assumptions.  

The idea is to decompose $I_2$ as the sum of the integrals over $[1/\lambda, \al]$, $[\al, \bl]$, and $[\bl, t]$, at least when $t>\bl$ (the other case is simpler, because only one or two intervals need to be considered). The integral on $[\al, \bl]$ can be directly bounded by assumption~(\ref{hp:al-bl}). On $[1/\lambda, \al]$ and $[\bl, t]$, the function $\varphi(s) - 2\lambda$ has constant sign (positive in the first interval and negative in the second), which allows us to proceed as in the case of $I_1$.  
In what follows, we detail the computation on $[1/\lambda, \al]$; the other case is analogous.

After integration by parts, the integral over $[1/\lambda,\al]$ is equal to
\begin{multline*}
\qquad
\left[
\frac{g(s)\varphi(s)}{c(s)(\varphi(s)-2\lambda)}
\sin(\Phi(s)-2\lambda s)\cos(2\psil(s))
\right]_{s=1/\lambda}^{s=\al}
\\[0.5ex]
-\int_{1/\lambda}^{\al}
\sin(\Phi(s)-2\lambda s)
\left[
\frac{g(s)\varphi(s)\cos(2\psil(s))}
{c(s)(\varphi(s)-2\lambda)}
\right]'ds.
\qquad
\end{multline*}

The boundary term is bounded in absolute value by $2\Hl(\al)+2\Hl(1/\lambda)$, and hence it is uniformly bounded because of (\ref{hp:Hl-infty}). We claim that the last integral is absolutely convergent, with a bound independent of $\lambda$.

Indeed, the derivative in the integral is equal to
\begin{equation*}
    \Hl'(s)\cdot\frac{\cos(2\psil(s))}{c(s)}
    -\frac{c'(s)}{c(s)^2}\cdot\Hl(s)\cos(2\psil(s))
    -2\psil'(s)\sin(2\psil(s))\cdot\frac{\Hl(s)}{c(s)},
\end{equation*}
and therefore its absolute value is bounded by
\begin{equation*}
    2|\Hl'(s)|+4|\Hl(s)|\cdot|c'(s)|+4|\Hl(s)|\cdot|\psil'(s)|.
\end{equation*}

Now the integral of $|\Hl'(s)|$ is uniformly bounded because of (\ref{hp:Hl-BV}). Taking (\ref{est:c'}) and (\ref{est:psil'}) into account, the other two terms are bounded by
\begin{equation*}
    4|\Hl(s)|\left(2|g'(s)|+\lambda g(s)+2g(s)\varphi(s)\strut\right).
\end{equation*}

The term $|\Hl(s)|\cdot|g'(s)|$ is uniformly integrable, since
(\ref{hp:Hl-infty}) yields
\[
    \int_{1/\lambda}^{\al}|\Hl(s)|\cdot|g'(s)|\,ds
    \leq
    M_2\int_0^{t_0}|g'(s)|\,ds.
\]

As for the two remaining terms, we write
\[
    \lambda=
    \frac{2\lambda-\varphi(s)}{2}
    +\frac{\varphi(s)}{2},
\]
and hence
\begin{eqnarray*}
    4|\Hl(s)|\left(\lambda g(s)+2g(s)\varphi(s)\right)
    & = &
    2|\Hl(s)|\cdot(2\lambda-\varphi(s))g(s)
    +10|\Hl(s)|\cdot g(s)\varphi(s)
    \\[0.5ex]
    & \leq &
    2|\Hl(s)|\cdot|2\lambda-\varphi(s)|g(s)
    +10|\Hl(s)|\cdot g(s)\varphi(s)
    \\[0.5ex]
    & = &
    2g(s)^2\varphi(s)
    +10|\Hl(s)|\cdot g(s)\varphi(s).
\end{eqnarray*}
The right-hand side is uniformly integrable by assumptions
(\ref{hp:g-phi-int-2}) and (\ref{hp:Hl-g-phi}).

The same argument applies to the integral over $[\bl,t]$, and the two oscillatory integrals involving the sine factors are treated in exactly the same way. Therefore all four terms in the decomposition of (\ref{eqn:2-int}) are uniformly bounded. This proves (\ref{th:Ehyp}), and hence (\ref{th:C1}), completing the proof.
\end{proof}


\subsection{A sufficient condition for non-resonance}

The non-resonance conditions introduced in Definition~\ref{defn:non-res} are tailored to the oscillatory estimates in the proof of Proposition~\ref{prop:non-res}, but they are not particularly convenient to verify directly. We therefore introduce a structural sufficient condition, expressed in terms of the monotonicity properties of $g$ and $\varphi$ and of the auxiliary quantity
\begin{equation}
    P(t):=\frac{\varphi(t)^2}{|\varphi'(t)|}.
    \label{defn:P}
\end{equation}
The advantage of this criterion is that it can be checked directly for the explicit families considered below.

\begin{lemma}[A sufficient condition for non-resonance]\label{lemma:non-res}

Let $t_0$ be a positive real number, let $g:(0,t_0]\to\re$ be a function of class $C^1$, and let $\varphi:(0,t_0]\to\re$ be a function of class $C^2$.

Let us assume that
\begin{enumerate}
\renewcommand{\labelenumi}{(\roman{enumi})}

    \item $g(t)\geq 0$ and $g'(t)\geq 0$ for every $t\in(0,t_0]$;

    \item  $\varphi(t)\to +\infty$ as $t\to 0^+$;

    \item $\varphi(t)>0$, $\varphi'(t)< 0$ and $\varphi''(t)\geq 0$ for every $t\in(0,t_0]$;

    \item $(g\varphi)'(t)\leq 0$ for every $t\in(0,t_0]$, and
    \begin{equation}
        \int_0^{t_0}g(t)^2\varphi(t)\,dt<+\infty;
        \label{hp:non-res-int}
    \end{equation}

    \item the function $P$ defined in (\ref{defn:P}) satisfies $P(t)\to +\infty$ as $t\to 0^+$, $P'(t)\leq 0$ for every $t\in(0,t_0]$, and
    \begin{equation}
        \limsup_{t\to 0^+}g(t)^2 P(t)\log P(t)<+\infty.
        \label{hp:limsup}
    \end{equation}
\end{enumerate}

Then $g$ and $\varphi$ satisfy the non-resonance conditions of Definition~\ref{defn:non-res}.

\end{lemma}

\begin{proof}

Let us first observe that, from the assumption (\ref{hp:limsup}) and the fact that $P(t)\to+\infty$ as $t\to0^+$, we have
\begin{equation}
    g(t)^2P(t)\to 0
    \qquad\text{as }t\to0^+.
    \label{est:g2P}
\end{equation}

For brevity, we introduce the function
\begin{equation*}
    q(t):=\frac{1}{|\varphi'(t)|^{1/2}}
    \qquad
    \forall t\in(0,t_0].
\end{equation*}

Since $\varphi(t)\to +\infty$ as $t\to 0^+$, for every $\lambda$ large enough, there exists a unique $\cl\in(0,t_0]$ such that $\varphi(\cl)=2\lambda$. At this point we define
\begin{equation*}
    \el:=q(\cl),
    \qquad\quad
    \al:=\cl-\el,
    \qquad\quad
    \bl:=\cl+\el,
\end{equation*}
and we claim that the non-resonance conditions of Definition~\ref{defn:non-res} are satisfied with these choices.

\paragraph{\textmd{\textit{Step 1}}}

We show that
\begin{equation}
    \lim_{t\to 0^+}q'(t)=0.
    \label{th:lim-q'}
\end{equation}

To this end, we observe that the condition $P'(t)\leq 0$ is equivalent to
\begin{equation*}
    \varphi(t)\varphi''(t)\leq 2\varphi'(t)^2,
\end{equation*}
and hence
\begin{equation*}
    q'(t)=
    \frac{1}{2}\frac{\varphi''(t)}{|\varphi'(t)|^{3/2}}\leq
    \frac{|\varphi'(t)|^{1/2}}{\varphi(t)}=
    \frac{1}{P(t)^{1/2}}.
\end{equation*}

Due to the sign condition on $\varphi''$, this implies that
\begin{equation*}
    0\leq q'(t)\leq\frac{1}{\sqrt{P(t)}},
\end{equation*}
and the conclusion follows because $P(t)\to +\infty$.

\paragraph{\textmd{\textit{Step 2}}}

We show that
\begin{equation}
    \lim_{t\to 0^+}q(t)=0,
    \qquad\qquad
    \lim_{t\to 0^+}\frac{q(t)}{t}=0,
    \label{th:lim-q-1}
\end{equation}
and
\begin{equation}
    \lim_{t\to 0^+}q(t)\varphi(t)=+\infty,
    \qquad\qquad
    \lim_{t\to 0^+}t\,\varphi(t)=+\infty.
    \label{th:lim-q-2}
\end{equation}

Indeed, from (\ref{th:lim-q'}) we deduce that $q'$ is bounded in a right neighborhood of the origin, and therefore $q(t)$ has a real limit $\ell$ as $t\to0^+$. If $\ell>0$, then $|\varphi'(t)|\longrightarrow 1/\ell^2$, but this is impossible because we assumed that $\varphi(t)\to+\infty$ as $t\to0^+$. This proves the first limit in (\ref{th:lim-q-1}).

At this point the second limit follows from (\ref{th:lim-q'}) by L'Hôpital's rule.

Now we observe that
\begin{equation*}
    q(t)\varphi(t)=\sqrt{P(t)},
\end{equation*}
and therefore the first limit in (\ref{th:lim-q-2}) follows from our assumption that $P(t)\to +\infty$. 

Finally, we observe that
\begin{equation*}
    t\varphi(t)=\frac{t}{q(t)}\cdot q(t)\varphi(t),
\end{equation*}
and hence also the last limit in (\ref{th:lim-q-2}) follows from the previous ones.

\paragraph{\textmd{\textit{Step 3}}}

We show that $\al$ and $\bl$ satisfy (\ref{hp:al-bl-position}) provided that $\lambda$ is large enough. Moreover, since $\varphi$ is strictly decreasing and $\varphi(\cl)=2\lambda$, condition~(\ref{hp:phi-2lambda}) is automatically satisfied.

Indeed, from the definition we immediately have $\cl\to 0^+$, hence also $\el\to 0^+$, and this is enough to establish that $\al<\bl<t_0$ for every large enough $\lambda$. In addition, we observe that
\begin{equation*}
    \lambda\al=
    \lambda(\cl-\el)=
    \lambda\cl\left(1-\frac{\el}{\cl}\right)=
    \frac{1}{2}\varphi(\cl)\cl\left(1-\frac{q(\cl)}{\cl}\right).
\end{equation*}

Therefore, from the second limits in (\ref{th:lim-q-1}) and (\ref{th:lim-q-2}) we deduce that $\lambda\al\to +\infty$, and hence $\al>1/\lambda$ when $\lambda$ is large enough.

\paragraph{\textmd{\textit{Step 4}}}

We show that
\begin{equation}
    \frac{1}{2}\el\leq q(t)\leq\frac{3}{2}\el
    \qquad
    \forall t\in[\al,\bl]
    \label{est:q-al-bl}
\end{equation}
when $\lambda$ is large enough.

Indeed, since $|t-\cl|\leq\el$, from the mean value theorem we deduce that
\begin{equation*}
    |q(t)-\el|=
    |q(t)-q(\cl)|\leq
    \el\cdot\max\{|q'(\tau)|:\tau\in[\al,\bl]\},
\end{equation*}
and from (\ref{th:lim-q'}) we know that the maximum is less than or equal to $1/2$ when $\lambda$ is large enough. This is enough to conclude (\ref{est:q-al-bl}).

\paragraph{\textmd{\textit{Step 5}}}

We prove condition~(\ref{hp:al-bl}) in the non-resonance assumption.

Since $\bl-\al=2\el$, from the monotonicity of $g\varphi$ and the lower bound in (\ref{est:q-al-bl}) with $t=\al$, we deduce that
\begin{equation*}
    \int_{\al}^{\bl} g(t)\varphi(t)\,dt\leq
    2\el g(\al)\varphi(\al)\leq
    4q(\al)g(\al)\varphi(\al)=
    4g(\al)\sqrt{P(\al)}.
\end{equation*}

Recalling (\ref{est:g2P}), this proves that the integral is uniformly bounded in $\lambda$.

\paragraph{\textmd{\textit{Step 6}}}

We show that
\begin{equation}
    \frac{4}{9}\frac{1}{\el}\leq\varphi(\al)-2\lambda\leq\frac{4}{\el}
    \qquad\quad\text{and}\quad\qquad
    \frac{4}{9}\frac{1}{\el}\leq2\lambda-\varphi(\bl)\leq\frac{4}{\el}
    \label{est:Hl-denom}
\end{equation}
when $\lambda$ is large enough, and
\begin{equation}
    \lim_{\lambda\to +\infty}\frac{\varphi(\al)}{2\lambda}
    =
    \lim_{\lambda\to +\infty}\frac{\varphi(\bl)}{2\lambda}
    =1.
    \label{eqn:lim-phi-al-bl}
\end{equation}

Indeed, from the mean value theorem we know that
\begin{equation*}
    \varphi(\al)-2\lambda=
    \varphi(\al)-\varphi(\cl)=
    \el|\varphi'(t_1)|=
    \frac{\el}{q(t_1)^2}
\end{equation*}
for some $t_1\in(\al,\cl)$, and
\begin{equation*}
    2\lambda-\varphi(\bl)=
    \varphi(\cl)-\varphi(\bl)=
    \el|\varphi'(t_2)|=
    \frac{\el}{q(t_2)^2}
\end{equation*}
for some $t_2\in(\cl,\bl)$.

On the other hand, from (\ref{est:q-al-bl}) we know that
\begin{equation*}
    \frac{4}{9}\frac{1}{\el^2}\leq
    \frac{1}{q(t)^2}\leq
    \frac{4}{\el^2}
    \qquad
    \forall t\in[\al,\bl].
\end{equation*}
Applying these inequalities with $t=t_1$ and $t=t_2$ to the previous two identities, we obtain
(\ref{est:Hl-denom}).

In particular,
\begin{equation*}
    1\leq
    \frac{\varphi(\al)}{2\lambda}=
    1+\frac{\varphi(\al)-2\lambda}{2\lambda}
    \leq
    1+\frac{4}{2\lambda\el}
    =
    1+\frac{4}{\varphi(\cl)q(\cl)},
\end{equation*}
while
\begin{equation*}
    1-\frac{4}{\varphi(\cl)q(\cl)}
    =
    1-\frac{4}{2\lambda\el}
    \leq
    1-\frac{2\lambda-\varphi(\bl)}{2\lambda}=
    \frac{\varphi(\bl)}{2\lambda}
    \leq1.
\end{equation*}
Therefore (\ref{eqn:lim-phi-al-bl}) follows from the first limit in
(\ref{th:lim-q-2}).

\paragraph{\textmd{\textit{Step 7}}}

The function $\Hl$ is nondecreasing on both intervals $(0,\cl)$ and $(\cl,t_0)$.

Indeed, its derivative is given by
\begin{equation*}
    \Hl'(t)=
    \frac{\varphi(t)^2g'(t)-2\lambda(g\varphi)'(t)}
    {(\varphi(t)-2\lambda)^2},
\end{equation*}
and therefore $\Hl'(t)\geq0$ because $g'(t)\geq0$ and $(g\varphi)'(t)\leq0$.

\paragraph{\textmd{\textit{Step 8}}}

We show that both $\Hl$ and its total variation are uniformly bounded in
$[1/\lambda,\al]$ and in $[\bl,t_0]$.

Indeed, since $\Hl$ is nondecreasing, positive in $[1/\lambda,\al]$, and negative in $[\bl,t_0]$, both its supremum norm and its total variation are controlled by $\Hl(\al)$ on the first interval and by $|\Hl(\bl)|$ on the second one.

From the lower bounds in (\ref{est:Hl-denom}) and (\ref{est:q-al-bl}) with $t=\al$ and $t=\bl$, we deduce that
\begin{equation*}
    \Hl(\al)
    \leq
    \frac{9}{4}\el g(\al)\varphi(\al)
    \leq
    \frac{9}{2}q(\al)g(\al)\varphi(\al)
    =
    \frac{9}{2}\sqrt{g(\al)^2P(\al)},
\end{equation*}
and, analogously,
\begin{equation*}
    |\Hl(\bl)|
    \leq
    \frac{9}{4}\el g(\bl)\varphi(\bl)
    \leq
    \frac{9}{2}q(\bl)g(\bl)\varphi(\bl)
    =
    \frac{9}{2}\sqrt{g(\bl)^2P(\bl)}.
\end{equation*}

Recalling (\ref{est:g2P}), both quantities tend to zero as $\lambda\to+\infty$. This proves both the uniform boundedness of $\Hl$ and the uniform boundedness of its total variation.

\paragraph{\textmd{\textit{Step 9}}}

We show that the integral
\begin{equation}
    \int_{1/\lambda}^{\al}g(t)\varphi(t)|\Hl(t)|\,dt
    \label{eqn:int-g-phi-Hl-left}
\end{equation}
is bounded independently of $\lambda$. 

To this end, we introduce $\albar$ such that $\varphi(\albar)=3\lambda$. To begin with, we claim that $1/\lambda<\albar<\al$ when $\lambda$ is large enough.

Indeed, the inequality $\albar<\al$ follows from (\ref{eqn:lim-phi-al-bl}), while for the remaining one we observe that
\begin{equation*}
    \lambda\albar=
    \frac{1}{3}\varphi(\albar)\albar\to+\infty
\end{equation*}
because of the second limit in (\ref{th:lim-q-2}).

Now we split the integral in (\ref{eqn:int-g-phi-Hl-left}) into the two subintervals $[1/\lambda,\albar]$ and $[\albar,\al]$. In the first interval we know that $\varphi(t)\geq3\lambda$, so that
\begin{equation*}
    \frac{\varphi(t)}{\varphi(t)-2\lambda}\leq3,
\end{equation*}
and therefore
\begin{equation*}
    \int_{1/\lambda}^{\albar}g(t)\varphi(t)|\Hl(t)|\,dt
    \leq
    3\int_0^{t_0}g(t)^2\varphi(t)\,dt.
\end{equation*}

In the second interval we write
\begin{equation*}
    \int_{\albar}^{\al}g(t)\varphi(t)|\Hl(t)|\,dt=
    \int_{\albar}^{\al}
    g(t)^2P(t)\log P(t)\cdot
    \frac{1}{\log P(t)}\cdot
    \frac{|\varphi'(t)|}{\varphi(t)-2\lambda}\,dt.
\end{equation*}

From (\ref{hp:limsup}) and the monotonicity of $P$ we know that there exists a constant $C$ such that
\begin{equation*}
    g(t)^2P(t)\log P(t)\leq C
    \qquad\quad\text{and}\quad\qquad
    \frac{1}{\log P(t)}\leq\frac{1}{\log P(\cl)},
\end{equation*}
while
\begin{equation*}
    \int_{\albar}^{\al}
    \frac{|\varphi'(t)|}{\varphi(t)-2\lambda}\,dt
    =
    \log
    \frac{\varphi(\albar)-2\lambda}
    {\varphi(\al)-2\lambda}.
\end{equation*}

Recalling that $\varphi(\albar)=3\lambda$, from the lower bound in (\ref{est:Hl-denom}) we obtain
\begin{equation*}
    \frac{\varphi(\albar)-2\lambda}
    {\varphi(\al)-2\lambda}
    =
    \frac{\lambda}{\varphi(\al)-2\lambda}
    \leq
    \frac{9}{4}\el\lambda
    =
    \frac{9}{8}\varphi(\cl)q(\cl)
    =
    \frac{9}{8}\sqrt{P(\cl)},
\end{equation*}
and therefore
\begin{equation*}
    \int_{\albar}^{\al}g(t)\varphi(t)|\Hl(t)|\,dt
    \leq
    C\left(
    \frac12+
    \frac{\log(9/8)}{\log P(\cl)}
    \right).
\end{equation*}

Since $P(\cl)\to+\infty$, this proves that the left-hand side is uniformly bounded in $\lambda$ for $\lambda$ large enough.

\paragraph{\textmd{\textit{Step 10}}}

We prove that the integral
\begin{equation*}
    \int_{\bl}^{t_0}g(t)\varphi(t)|\Hl(t)|\,dt
\end{equation*}
is bounded independently of $\lambda$. 

Let $\blbar$ be defined by $\varphi(\blbar)=\lambda$. From (\ref{eqn:lim-phi-al-bl}) we have $\bl<\blbar$ for $\lambda$ large enough, while $\blbar<t_0$ follows from the monotonicity of $\varphi$ for $\lambda$ large enough.

We split the interval into $[\bl,\blbar]$ and $[\blbar,t_0]$. In the second interval we have $\varphi(t)\leq\lambda$, and therefore, exactly as in Step~9,
\begin{equation*}
    \int_{\blbar}^{t_0}g(t)\varphi(t)|\Hl(t)|\,dt
    \leq
    \int_0^{t_0}g(t)^2\varphi(t)\,dt.
\end{equation*}

It remains to consider the interval $[\bl,\blbar]$. The only additional ingredient with respect to Step~9 is a lower bound for $P(t)$ in terms of $P(\cl)$. From the definition of $P$ and the sign condition on $\varphi''$, we obtain that
\begin{equation*}
    \frac{P'(t)}{P(t)}
    \geq
    2\frac{\varphi'(t)}{\varphi(t)}.
\end{equation*}

We now integrate this differential inequality from $\cl$ to $t\in[\cl,\blbar]$. Recalling that $\varphi(\cl)=2\lambda$ and $\varphi(t)\geq\lambda$ in this interval, we deduce that
\begin{equation*}
    \log\frac{P(t)}{P(\cl)}\geq
    2\log\frac{\varphi(t)}{\varphi(\cl)}\geq
    \log\frac{1}{4},
\end{equation*}
and in particular, since $P(\cl)\to +\infty$ as $\lambda\to +\infty$,
\begin{equation*}
    \log P(t)\geq
    \log P(\cl)-\log 4\geq
    \frac{1}{2}\log P(\cl)
    \qquad
    \forall t\in[\cl,\blbar].
\end{equation*}

At this point we argue as in Step~9 and write
\begin{equation*}
    \int_{\bl}^{\blbar}g(t)\varphi(t)|\Hl(t)|\,dt
    \leq
    \frac{2C}{\log P(\cl)}
    \int_{\bl}^{\blbar}
    \frac{|\varphi'(t)|}{2\lambda-\varphi(t)}\,dt
    =
    \frac{2C}{\log P(\cl)}\log\frac{2\lambda-\varphi(\blbar)}{2\lambda-\varphi(\bl)}.
\end{equation*}

Recalling that $\varphi(\blbar)=\lambda$, from the lower bound in (\ref{est:Hl-denom}), we obtain
\begin{equation*}
    \frac{2\lambda-\varphi(\blbar)}{2\lambda-\varphi(\bl)}
    =
    \frac{\lambda}{2\lambda-\varphi(\bl)}
    \leq
    \frac{9}{4}\el\lambda
    =
    \frac{9}{8}\varphi(\cl)q(\cl)
    =
    \frac{9}{8}\sqrt{P(\cl)}.
\end{equation*}
Therefore
\begin{equation*}
    \int_{\bl}^{\blbar}g(t)\varphi(t)|\Hl(t)|\,dt
    \leq
    2C\left(
    \frac{1}{2}
    +
    \frac{\log(9/8)}{\log P(\cl)}
    \right),
\end{equation*}
which is uniformly bounded for $\lambda$ large enough.
\end{proof}


\subsection{A simple example based on powers}

We now consider the propagation speed defined by (\ref{defn:c(t)-example}) starting with
\begin{equation}
    g(t):=t^\alpha
    \qquad\quad\text{and}\quad\qquad
    \varphi(t):=\frac{1}{t^{\beta+1}},
    \label{defn:g-phi-power}
\end{equation}
where $\alpha$ is a positive integer and $\beta$ is a real number such that
\begin{equation*}
    \alpha<\beta<2\alpha.
\end{equation*}

With this choice, $g$ is of class $C^\infty$ up to the origin. More precisely, we choose $t_0>0$ small enough that $t_0^\alpha<1/2$, we set $g(t)=t^\alpha$ on $[0,t_0]$, and then extend $g$ smoothly to $[0,+\infty)$ with values in $[0,1/2]$. This modification does not affect the argument, since only the behavior near the origin is relevant for the frequency-dependent estimates.

\paragraph{\textmd{\textit{Well-posedness in Sobolev spaces}}}

In order to show that in this case the abstract wave equation (\ref{eqn:basic}) is well posed in Sobolev spaces, it is enough to show that $g$ and $\varphi$ fit into the setting of Proposition~\ref{prop:non-res}.

The integrability conditions (\ref{hp:g-phi-int-1}) and (\ref{hp:g-phi-int-2}) follow immediately from
\begin{equation*}
    \frac{g(t)|\varphi'(t)|}{\varphi(t)}
    =
    (\beta+1)t^{\alpha-1}
    \qquad\quad\text{and}\quad\qquad
    g(t)^2\varphi(t)=t^{2\alpha-\beta-1}
\end{equation*}
since $\alpha>0$ and $\beta<2\alpha$.

In order to show the non-resonance conditions, we now verify the assumptions of Lemma~\ref{lemma:non-res}. Clearly, $g$ is nonnegative and nondecreasing, while $\varphi$ is positive, decreasing, convex, and tends to $+\infty$ as $t\to0^+$. Moreover, $g(t)\varphi(t)=t^{\alpha-\beta-1}$ is decreasing because $\beta>\alpha$, while (\ref{hp:non-res-int}) coincides with (\ref{hp:g-phi-int-2}).

Finally, the function defined in (\ref{defn:P}) is given by
\begin{equation*}
    P(t)=\frac{1}{\beta+1}\frac{1}{t^\beta}.
\end{equation*}
In particular, $P(t)\to+\infty$ as $t\to0^+$ and $P'(t)<0$. Furthermore,
\begin{equation*}
    g(t)^2P(t)\log P(t)
    =
    \frac{t^{2\alpha-\beta}}{\beta+1}
    \log\left(\frac{1}{(\beta+1)t^\beta}\right)
    \longrightarrow0
    \qquad
    \text{as }t\to0^+,
\end{equation*}
because $2\alpha>\beta$.

Therefore all the assumptions of Lemma~\ref{lemma:non-res} are satisfied, and Proposition~\ref{prop:non-res} applies.

\paragraph{\textmd{\textit{Lack of fractional Sobolev regularity}}}

We claim that, for every positive real number $T$, the propagation speed defined by (\ref{defn:c(t)-example}) with the choices (\ref{defn:g-phi-power}) satisfies
\begin{equation}
    c\not\in W^{s,1}((0,T))
    \qquad
    \forall s\geq s_*:=\frac{\alpha+1}{\beta+1}.
    \label{th:c-no-ws}
\end{equation}
We observe that $s_*<1$ because $\alpha<\beta$, and therefore, by the standard Sobolev inclusions, it is enough to prove (\ref{th:c-no-ws}) when $s_*\leq s<1$.

To this end, we observe that (\ref{defn:Phi}) in this case reads as
\begin{equation}
    \Phi(t):=\frac{1}{\beta}\left(\frac{1}{t_0^\beta}-\frac{1}{t^\beta}\right)
    \qquad
    \forall t>0,
    \label{defn:Phi-power}
\end{equation}
and, for every positive integer $k$, we define $a_k$, $b_k$, $c_k$, $d_k$ in such a way that
\begin{gather}
    \Phi(a_k)=\frac{\pi}{6}-2k\pi,
    \qquad\qquad
    \Phi(b_k)=\frac{5\pi}{6}-2k\pi,
    \label{defn:ak-bk-power}
    \\[1ex]
    \Phi(c_k)=\pi-2k\pi,
    \qquad\qquad
    \Phi(d_k)=2\pi-2k\pi.
    \label{defn:ck-dk-power}
\end{gather}

We observe that $d_{k+1}<a_k<b_k<c_k<d_k$ for every positive integer $k$. Moreover,
\[
    \sin(\Phi(t))\geq\frac{1}{2}
    \qquad
    \forall t\in[a_k,b_k],
\]
and therefore, from (\ref{defn:c(t)-example}),
\begin{equation*}
    c(t)\geq 1+\frac{1}{2}g(t)\geq 1+\frac{1}{2}g(a_k)
    \qquad
    \forall t\in[a_k,b_k].
\end{equation*}
Similarly,
\begin{equation*}
    c(\tau)\leq 1
    \qquad
    \forall \tau\in[c_k,d_k].
\end{equation*}
It follows that
\begin{equation}
    \frac{|c(t)-c(\tau)|}{|t-\tau|^{1+s}}
    \geq
    \frac{1}{2}\frac{g(a_k)}{(d_k-a_k)^{1+s}}
    \qquad
    \forall (t,\tau)\in[a_k,b_k]\times[c_k,d_k].
    \label{est:ws-below}
\end{equation}

Now we observe that, for every $T>0$, the square $(0,T)\times(0,T)$ contains all the rectangles $[a_k,b_k]\times[c_k,d_k]$ for $k$ large enough, and these rectangles are pairwise disjoint. Therefore, the integral of the left-hand side of (\ref{est:ws-below}) over $(0,T)\times(0,T)$ can be estimated from below by the sum of the integrals over these rectangles. As a consequence, we deduce that $c\not\in W^{s,1}((0,T))$ whenever
\begin{equation}
    \sum_{k=k_T}^{\infty}
    (b_k-a_k)(d_k-c_k)\cdot
    \frac{g(a_k)}{(d_k-a_k)^{1+s}}
    =+\infty
    \label{series:power}
\end{equation}
for some sufficiently large $k_T$.

Finally, from (\ref{defn:Phi-power}), (\ref{defn:ak-bk-power}), and (\ref{defn:ck-dk-power}), a straightforward computation shows that $g(a_k)$ has the same order as $k^{-\alpha/\beta}$, while $b_k-a_k$, $d_k-a_k$, and $d_k-c_k$ have the same order as $k^{-(\beta+1)/\beta}$. Hence
\begin{equation*}
    (b_k-a_k)(d_k-c_k)\cdot
    \frac{g(a_k)}{(d_k-a_k)^{1+s}}
\end{equation*}
has the same order as
\begin{equation*}
    \left(\frac{1}{k}\right)^{
    1+\frac{\beta+1}{\beta}(s_*-s)}.
\end{equation*}
Therefore, the series in (\ref{series:power}) diverges if and only if $s\geq s_*$. This completes the proof of (\ref{th:c-no-ws}).


\subsection{A refined example based on exponentials}

The power-like examples of the previous subsection yield propagation speeds that fail to belong to $W^{s,1}$ above a threshold that can be made arbitrarily close to $1/2$, but never equal to $1/2$. We now refine the construction in order to reach the critical threshold and complete the proof of Theorem~\ref{thm:main-irregular}.

We consider the propagation speed defined by (\ref{defn:c(t)-example}) starting with
\begin{equation}
    g(t):=\frac{e^{-1/t}}{\sqrt{t}\,|\log t|}
    \qquad\quad\text{and}\quad\qquad
    \varphi(t):=e^{2/t}.
    \label{defn:g-phi-exp}
\end{equation}
As in the case of powers, these functions are defined as above only in a right neighborhood $(0,t_0]$ of the origin, with $t_0$ small enough. We then extend $g$ smoothly to $[0,+\infty)$ with values in $[0,1/2]$, and $\varphi$ smoothly to $(0,+\infty)$ with positive values.

\paragraph{\textmd{\textit{Well-posedness in Sobolev spaces}}}

As in the previous example, it is enough to verify the assumptions of Proposition~\ref{prop:non-res}.

The integrability conditions (\ref{hp:g-phi-int-1}) and (\ref{hp:g-phi-int-2}) follow from
\begin{equation*}
    \frac{g(t)|\varphi'(t)|}{\varphi(t)}
    =
    \frac{2e^{-1/t}}{t^{5/2}|\log t|}
    \qquad\quad\text{and}\quad\qquad
    g(t)^2\varphi(t)
    =
    \frac{1}{t\log^2 t},
\end{equation*}
which are integrable in a right neighborhood of the origin.

We now verify the assumptions of Lemma~\ref{lemma:non-res}. If $t_0$ is small enough, then $g$ is nonnegative and nondecreasing, while $\varphi$ is positive, decreasing, convex, and tends to $+\infty$ as $t\to0^+$. Moreover,
\begin{equation*}
    g(t)\varphi(t)
    =
    \frac{e^{1/t}}{\sqrt{t}\,|\log t|}
\end{equation*}
is decreasing in a right neighborhood of the origin, while (\ref{hp:non-res-int}) coincides with (\ref{hp:g-phi-int-2}).

Finally, the function defined in (\ref{defn:P}) is given by
\begin{equation*}
    P(t)=\frac{t^2}{2}e^{2/t}.
\end{equation*}
Therefore $P(t)\to+\infty$ as $t\to0^+$ and, after reducing $t_0$ if necessary,
$P'(t)<0$. Furthermore,
\begin{equation*}
    g(t)^2P(t)\log P(t)
    =
    \frac{t}{2\log^2 t}
    \log\left(\frac{t^2}{2}e^{2/t}\right)
    \longrightarrow0
    \qquad
    \text{as }t\to0^+.
\end{equation*}

Therefore all the assumptions of Lemma~\ref{lemma:non-res} are satisfied, and Proposition~\ref{prop:non-res} applies.

\paragraph{\textmd{\textit{Lack of fractional Sobolev regularity}}}

We claim that, for every positive real number $T$, the propagation speed defined by (\ref{defn:c(t)-example}) with the choices (\ref{defn:g-phi-exp}) satisfies
\begin{equation}
    c\not\in W^{s,1}((0,T))
    \qquad
    \forall s>\frac{1}{2}.
    \label{th:c-no-ws-exp}
\end{equation}

By the standard Sobolev inclusions, it is enough to prove~(\ref{th:c-no-ws-exp}) for $1/2<s<1$. As in the case of powers, we define $a_k$, $b_k$, $c_k$, and $d_k$ by (\ref{defn:ak-bk-power}) and (\ref{defn:ck-dk-power}), where now $\Phi$ is the antiderivative of the function $\varphi$ in (\ref{defn:g-phi-exp}). The same argument as in the previous example reduces the proof of (\ref{th:c-no-ws-exp}) to showing that the series in (\ref{series:power}) diverges for every $s>1/2$.

We shall prove that
\begin{equation}
    a_k\sim\frac{2}{\log k},
    \qquad
    e^{-1/a_k}\sim\frac{1}{\sqrt{\pi k}\log k},
    \label{asympt:exp-1}
\end{equation}
and
\begin{equation}
    b_k-a_k\sim\frac{2}{3}\frac{1}{k\log^2 k},
    \qquad
    d_k-a_k\sim\frac{11}{6}\frac{1}{k\log^2 k},
    \qquad
    d_k-c_k\sim\frac{1}{k\log^2 k}.
    \label{asympt:exp-2}
\end{equation}
It follows from (\ref{asympt:exp-1}) that
\begin{equation*}
    g(a_k)
    \sim
    \frac{1}{\sqrt{2\pi}}
    \frac{1}{(k\log k)^{1/2}\log(\log k)}.
\end{equation*}
Therefore, up to a positive multiplicative constant depending on $s$,
\begin{equation*}
    (b_k-a_k)(d_k-c_k)\cdot
    \frac{g(a_k)}{(d_k-a_k)^{1+s}}
\end{equation*}
is asymptotic to
\begin{equation}
    \frac{1}{(k\log k)^{1/2}\log(\log k)}
    \cdot
    \frac{1}{(k\log^2 k)^{1-s}}.
    \label{asympt:series-exp}
\end{equation}
For every $s>1/2$, the series whose general term is given by (\ref{asympt:series-exp}) diverges, and hence (\ref{series:power}) holds. This proves (\ref{th:c-no-ws-exp}).

It remains to prove (\ref{asympt:exp-1}) and (\ref{asympt:exp-2}). The starting point is
\begin{equation*}
    \Phi(t)\sim-\frac{1}{2}t^2e^{2/t}
    \qquad
    \text{as }t\to0^+.
\end{equation*}

Let $u_k$ be any sequence such that
\begin{equation*}
    \Phi(u_k)=A-Bk
\end{equation*}
for some $A\in\mathbb R$ and $B>0$. Since
\[
    k=\frac{A-\Phi(u_k)}{B},
\]
we have
\begin{equation*}
    u_k\log k
    =
    u_k\log\left(u_k^2e^{2/u_k}\right)
    +
    u_k\log\left(
        \frac{A-\Phi(u_k)}
        {B u_k^2e^{2/u_k}}
    \right)
    \longrightarrow 2,
\end{equation*}
and hence
\begin{equation}
    u_k\sim\frac{2}{\log k}.
    \label{asympt:uk-exp}
\end{equation}
Moreover,
\begin{equation*}
    k(\log k)^2e^{-2/u_k}
    =
    (u_k\log k)^2
    \cdot
    \frac{\Phi(u_k)}{u_k^2e^{2/u_k}}
    \cdot
    \frac{k}{\Phi(u_k)}
    \longrightarrow
    \frac{2}{B},
\end{equation*}
and therefore
\begin{equation}
    e^{-1/u_k}
    \sim
    \sqrt{\frac{2}{B}}\,
    \frac{1}{\sqrt{k}\log k}.
    \label{asympt:exp-uk}
\end{equation}

Let now $v_k$ be another sequence such that
\begin{equation*}
    \Phi(v_k)=C-Bk
\end{equation*}
for some $C>A$. Since $\Phi'=\varphi$ and $\varphi$ is decreasing, the mean value theorem yields
\begin{equation*}
    \frac{C-A}{\varphi(u_k)}
    \leq
    v_k-u_k
    \leq
    \frac{C-A}{\varphi(v_k)}.
\end{equation*}
By (\ref{asympt:exp-uk}),
\[
    \varphi(u_k)\sim\varphi(v_k)\sim\frac{B}{2}k\log^2 k,
\]
and therefore
\begin{equation}
    v_k-u_k
    \sim
    \frac{2(C-A)}{B}\,
    \frac{1}{k\log^2 k}.
    \label{asympt:uk-vk}
\end{equation}

We now apply these relations to the sequences defined in (\ref{defn:ak-bk-power}) and (\ref{defn:ck-dk-power}), for which $B=2\pi$. Relations (\ref{asympt:uk-exp}) and (\ref{asympt:exp-uk}) yield (\ref{asympt:exp-1}), while (\ref{asympt:uk-vk}), with the corresponding values of $A$ and $C$, yields (\ref{asympt:exp-2}).

This completes the proof of Theorem~\ref{thm:main-irregular}.
\qed

\begin{rmk}
\begin{em}

The propagation speed obtained in the proof through the exponential choices of the amplitude and the phase as in (\ref{defn:g-phi-exp}) has an additional borderline H\"older regularity property. More precisely, after setting $c(0):=1$, one has
\[
    c\in C^{0,\alpha}([0,T])
    \qquad
    \text{for every } T>0 \text{ and every } \alpha<\frac12,
\]
whereas
\[
    c\notin C^{0,1/2}([0,T])
    \qquad
    \text{for every } T>0.
\]
Thus this example is H\"older continuous of every order strictly smaller than $1/2$, but not of order $1/2$.

\end{em}
\end{rmk}

\setcounter{equation}{0}
\section{Open problems}\label{sec:open}

\paragraph{\textmd{\textit{The boundary between good and bad propagation speeds}}}

The results of this paper suggest a three-region picture for strictly hyperbolic propagation speeds.

There is first a region where uniform energy estimates are guaranteed. In the variational language developed above, this region is precisely the class $\Sp_2$ associated with bounded second-order minimum values. By the saturation results of Section~\ref{sec:variational}, the same class is obtained from every higher-order variational problem. For propagation speeds in $\Sp_2$, the absolute-integrability approach is sufficient to guarantee well-posedness in the Sobolev scale without derivative loss.

Outside $\Sp_2$ the situation becomes substantially less rigid. On the one hand, counterexamples show that pathological behavior, including derivative loss, can be residual in suitable classes of propagation speeds~\cite{GhisiGobbino2023Residual}. On the other hand, the constructions of Section~\ref{sec:construction} show that propagation speeds outside $\Sp_2$ may still satisfy uniform energy estimates. For these exceptional coefficients, oscillatory cancellations compensate for the lack of absolute integrability required by the variational theory. Thus, beyond $\Sp_2$, good and bad propagation speeds may coexist.

This naturally raises the question of whether there is a further threshold beyond which such a coexistence is no longer possible. In other words, one may ask whether there exists a regularity region where oscillatory cancellations can no longer prevent derivative loss.

The constructions of Section~\ref{sec:construction} suggest that the fractional Sobolev exponent $1/2$ may play a critical role in this respect. The power-like examples approach this threshold arbitrarily closely, while the refined exponential example fails to belong to $W^{s,1}$ for every $s>1/2$, without settling the endpoint. Attempts to accelerate further the oscillations lead again to borderline contributions at $s=1/2$. This suggests that the obstruction at the endpoint may be structural rather than an artifact of the particular examples considered here.

It is therefore natural to ask whether $W^{1/2,1}$ regularity is necessary for uniform energy estimates. More precisely, one may ask whether a propagation speed $c$ for which the abstract wave equation is well posed in $D(A^{1/2})\times H$ for every Hilbert space $H$ and every nonnegative self-adjoint operator $A$ must satisfy
\begin{equation*}
    c\in W^{1/2,1}_{\mathrm{loc}}.
\end{equation*}

A positive answer would identify $W^{1/2,1}$ as a natural boundary between a region where good and bad propagation speeds may coexist and a region where uniform energy estimates are impossible. A negative answer would reveal the existence of an even more irregular regime in which oscillatory cancellations can still compensate for the mechanisms responsible for derivative loss.


\subsubsection*{\centering Acknowledgments}

Both authors are members of the Italian {\selectlanguage{italian}%
``Gruppo Nazionale per l'Analisi Matematica, la Probabilit\`{a} e le loro Applicazioni'' (GNAMPA) of the ``Istituto Nazionale di Alta Matematica'' (INdAM)}. 

The authors acknowledge the MIUR Excellence Department Project awarded to the Department of Mathematics, University of Pisa, CUP I57G22000700001.



\label{NumeroPagine}

\end{document}